\documentclass[11pt]{article}

\usepackage[margin=2.5cm]{geometry}

\usepackage{tikz}
\usetikzlibrary{positioning, arrows.meta, decorations.pathreplacing}

\usepackage{amsthm}
\usepackage{bm}
\usepackage{amsmath}
\usepackage{amsfonts}
\usepackage{color}
\usepackage{amssymb}
\usepackage{graphicx}
\usepackage{enumerate}
\usepackage{amsthm,amscd}
\usepackage[all]{xy}
\usepackage{pifont} 
\usepackage{authblk}

\allowdisplaybreaks

\def\marrow{{\boldmath {\marginpar[\hfill$\rightarrow \rightarrow$]{$\leftarrow \leftarrow$}}}}

\def\andrey#1{{\sc Andrey: }{\marrow\sf \textcolor{blue}{#1}}}
\newcommand{\m}{\mathcal}

\usepackage[colorlinks=true,anchorcolor=blue,filecolor=blue,linkcolor=red,urlcolor=blue,citecolor=blue]{hyperref} 

\newtheorem{theorem}{Theorem}[section]

\newtheorem{corollary}[theorem]{Corollary}
\newtheorem{definition}[theorem]{Definition}

\newtheorem{lemma}[theorem]{Lemma}

\newtheorem{proposition}[theorem]{Proposition}

\newtheorem*{remark*}{Remark} 
\newtheorem{claim}[theorem]{Claim}

\newtheorem{obser}[theorem]{Observation}
\newtheorem{assumption}[theorem]{Assumption}

\RequirePackage[normalem]{ulem} 
\RequirePackage{color}\definecolor{RED}{rgb}{1,0,0}\definecolor{BLUE}{rgb}{0,0,1} 
\providecommand{\DIFaddbegin}{} 
\providecommand{\DIFaddend}{} 
\providecommand{\DIFdelbegin}{} 
\providecommand{\DIFdelend}{} 
\providecommand{\DIFaddbeginFL}{} 
\providecommand{\DIFaddendFL}{} 
\providecommand{\DIFdelbeginFL}{} 
\providecommand{\DIFdelendFL}{} 
\newcommand{\DIFscaledelfig}{0.5}
\RequirePackage{settobox} 
\RequirePackage{letltxmacro} 
\newsavebox{\DIFdelgraphicsbox} 
\newlength{\DIFdelgraphicswidth} 
\newlength{\DIFdelgraphicsheight} 
\LetLtxMacro{\DIFOincludegraphics}{\includegraphics} 
\newcommand{\DIFaddincludegraphics}[2][]{{\color{blue}\fbox{\DIFOincludegraphics[#1]{#2}}}} 
\newcommand{\DIFdelincludegraphics}[2][]{
\sbox{\DIFdelgraphicsbox}{\DIFOincludegraphics[#1]{#2}}
\settoboxwidth{\DIFdelgraphicswidth}{\DIFdelgraphicsbox} 
\settoboxtotalheight{\DIFdelgraphicsheight}{\DIFdelgraphicsbox} 
\scalebox{\DIFscaledelfig}{
\parbox[b]{\DIFdelgraphicswidth}{\usebox{\DIFdelgraphicsbox}\\[-\baselineskip] \rule{\DIFdelgraphicswidth}{0em}}\llap{\resizebox{\DIFdelgraphicswidth}{\DIFdelgraphicsheight}{
\setlength{\unitlength}{\DIFdelgraphicswidth}
\begin{picture}(1,1)
\thicklines\linethickness{2pt} 
{\color[rgb]{1,0,0}\put(0,0){\framebox(1,1){}}}
{\color[rgb]{1,0,0}\put(0,0){\line( 1,1){1}}}
{\color[rgb]{1,0,0}\put(0,1){\line(1,-1){1}}}
\end{picture}
}\hspace*{3pt}}} 
} 
\LetLtxMacro{\DIFOaddbegin}{\DIFaddbegin} 
\LetLtxMacro{\DIFOaddend}{\DIFaddend} 
\LetLtxMacro{\DIFOdelbegin}{\DIFdelbegin} 
\LetLtxMacro{\DIFOdelend}{\DIFdelend} 
\DeclareRobustCommand{\DIFaddbegin}{\DIFOaddbegin \let\includegraphics\DIFaddincludegraphics} 
\DeclareRobustCommand{\DIFaddend}{\DIFOaddend \let\includegraphics\DIFOincludegraphics} 
\DeclareRobustCommand{\DIFdelbegin}{\DIFOdelbegin \let\includegraphics\DIFdelincludegraphics} 
\DeclareRobustCommand{\DIFdelend}{\DIFOaddend \let\includegraphics\DIFOincludegraphics} 
\LetLtxMacro{\DIFOaddbeginFL}{\DIFaddbeginFL} 
\LetLtxMacro{\DIFOaddendFL}{\DIFaddendFL} 
\LetLtxMacro{\DIFOdelbeginFL}{\DIFdelbeginFL} 
\LetLtxMacro{\DIFOdelendFL}{\DIFdelendFL} 
\DeclareRobustCommand{\DIFaddbeginFL}{\DIFOaddbeginFL \let\includegraphics\DIFaddincludegraphics} 
\DeclareRobustCommand{\DIFaddendFL}{\DIFOaddendFL \let\includegraphics\DIFOincludegraphics} 
\DeclareRobustCommand{\DIFdelbeginFL}{\DIFOdelbeginFL \let\includegraphics\DIFdelincludegraphics} 
\DeclareRobustCommand{\DIFdelendFL}{\DIFOaddendFL \let\includegraphics\DIFOincludegraphics} 

\begin{document}

\title{Structure and properties of large cross-intersecting families}
\author{
Yang Huang\thanks{Moscow Institute of Physics and Technology; 
E-mail: \url{1060393815@qq.com}.}
\quad \quad 
Andrey Kupavskii\thanks{Moscow Institute of Physics and Technology; 
E-mail: \url{kupavskii@ya.ru}. }
}
\date{\today}

\maketitle

\vspace{-0.5cm}
\begin{abstract}

The study of intersecting families, initiated by Erd\H{o}s, Ko, and Rado, is a central topic in extremal combinatorics. A classical stability result of Hilton and Milner determines the largest non-trivial intersecting family, and in subsequent works researchers developed structural stability results via the notion of diversity.

In this paper, we study cross-intersecting families. We establish a structural theorem for large cross-intersecting pairs, extending Kupavskii's theorem from intersecting families to the cross-intersecting setting. Our result characterizes extremal cross-intersecting pairs in terms of their diversity parts and maximal cross-intersecting extensions. As corollaries, we obtain cross-intersecting analogues of several classical theorems, including those of Han--Kohayakawa and Huang--Peng.

A key ingredient in the proof is a new shifting method, called the
$S_{U,V}^{Q}$-shift, which not only preserves global intersection
properties but also maintains certain local substructures after shifting.
We expect this method to be useful elsewhere, and it is already one of the key tools in establishing a product analogue of the Hilton--Milner theorem.

\end{abstract}




\section{Introduction}
Let $[n]=\{1,\dots, n\}$ denote the standard $n$-element set. 
For a set $X$ and an integer $k$, let $2^{X}$ denote the power set of $X$, 
and let ${X\choose k}$ denote the family of all $k$-element subsets of $X$. A family $\mathcal{F}$ is called  intersecting if for any $A, B\in \m F$ we have $A\cap B\ne \emptyset$. A {\it star} is a family whose sets contain a fixed element. 
A {\it full star} consists of all $k$-sets containing a fixed element.

The study of problems on intersecting families was initiated by Erd\H{o}s, Ko, and Rado \cite{EKR1961}, who proved that for any $n> 2k$, the largest intersecting families in ${[n]\choose k}$ are full stars. 
A family $\mathcal{F}$ is called {\it non-trivial} if $\mathcal{F}$ is not a star.  
Moreover, they also asked what is the largest non-trivial intersecting family. 
This question was resolved by Hilton and Milner \cite{HM1967}, who showed that the largest non-trivial intersecting family is  isomorphic  to $\mathcal{HM}_{n,k}:=\{ [2,k+1] \} \cup \{H\in \binom{[n]}{k}: 1\in H, H\cap [2,k+1] \neq \emptyset\}$.  
The work of Hilton and Milner \cite{HM1967} provides the first {\it stability} result for the Erd\H{o}s--Ko--Rado theorem, where a stability result refers to a result that gives an (ideally optimal) trade-off  between how far an intersecting family $\m F$ is from a star, and how large $\m F$ is.  
Kostochka and Mubayi \cite{KM2016}
studied the structure of large non-trivial intersecting families for $k\ge 4$ and $n$ large enough.
Subsequently,  Han and Kohayakawa \cite{HK2017} 
investigated the largest intersecting families that are not contained in full stars, nor in Hilton--Milner's families. Furthermore, 
the first author and Peng \cite{HP2024} studied the largest intersecting families that are not contained in full stars, nor in Hilton--Milner's families, nor in Han--Kohayakawa's families.

It is convenient to state the stability results of Erd\H{o}s--Ko--Rado theorem in terms of the diversity. For a family $\mathcal{F}$, we denote by $\Delta(\mathcal{F})$ the maximum degree of $\mathcal{F}$. The {\it diversity} of $\m F$ is defined as 
\[ \gamma(\m F) := |\m F|-\Delta(\m F).  \] 
In other words, $\gamma (\m F)$ is the number of sets in $\m F$ not containing the element of maximum degree. The following result of Frankl \cite{Frankl-max-degree}, slightly strengthened and formulated in diversity terms by Kupavskii and Zakharov~\cite{KZ2018}, gives such an optimal stability result.

\begin{theorem}[Frankl \cite{Frankl-max-degree};  Kupavskii--Zakharov \cite{KZ2018}] \label{thm:FKZ}
Let $n>2k$ be positive integers. Suppose that \( \mathcal{F} \subset \binom{[n]}{k} \) is intersecting and $\gamma(\m F)\ge {n-u-1\choose  n-k-1}$ for some real $u$, where $3\le u\le k$. Then
\[
|\mathcal{F}| \leq \binom{n-1}{k-1} - \binom{n-u-1}{k-1} + \binom{n-u-1}{n-k-1}.
\]
\end{theorem}

Setting $u=k$, Theorem \ref{thm:FKZ} recovers the Hilton--Milner theorem. The bound in Theorem \ref{thm:FKZ} is tight when $u$ is an integer, as witnessed by setting $v=u$ in the following family: 
\begin{equation}\label{example}
\mathcal{H}_{k,v,u}:= \left\{H\in {[n]\choose k}: 1\in H, [2,v+1]\cap H\ne \emptyset \right\} \cup \left\{H\in {[n]\choose k}: [2,u+1]\subset H \right\}.\end{equation}
Unfortunately, Theorem~\ref{thm:FKZ} does not give an exact stability result for intersecting families $\m F$ with diversity $\gamma(\m F) \ge 2$. 
To fill this gap, Kupavskii \cite{Kup2018, 11-23-2} proved the following generalization, which is a conclusive form of the above results. Rather than just working with the distance between $\mathcal{F}$ and a star, it bounds the size of an intersecting family in terms of the structure of its diversity part. Before stating the result, we need to give some notation.

\begin{definition}
    We say that a family $\mathcal{M}$ is { \it intersection-minimal} if for any $M'\in \mathcal{M}$, we have 
    \[ \left|\cap_{M\in \mathcal{M}\setminus \{M'\}} M \right| 
    > \left|\cap_{M\in \mathcal{M}}M \right|. \]
\end{definition}
Note that a family consisting of a single set is regarded as intersection-minimal.
For an element $x\in [n]$, we denote  
$\mathcal{F}(\bar{x}):=\{F\in \mathcal{F} : x\notin F\}$ and 
$\mathcal{F}(x) :=\{F\setminus \{x\} : F\in \mathcal{F}, x\in F \}$. 
For each $i\in [k]$, we write $I_i:=[i+1,k+i]$ and
\[
\mathcal{J}_{i}:=\{I_1, I_i\}\cup \left\{F\in \tbinom{[n]}{ k}: 1\in F, F\cap I_1\ne \emptyset, F\cap I_i\ne \emptyset \right\}.
\]

\begin{theorem}[Kupavskii \cite{Kup2018, 11-23-2}]\label{11-23-2}
Assume that \( n > 2k \geq 8 \). Consider an intersecting family \( \mathcal{F} \subset \binom{[n]}{k} \) with $1$ being the element of maximum degree. Take an intersection-minimal subfamily \( \mathcal{M} \subset \mathcal{F}(\bar{1}) \) such that
$| \cap_{M \in \mathcal{M}} M | = m.
$
Take the (unique) maximal intersecting family \( \mathcal{F}' \), such that \( \mathcal{F}'(\bar{1}) = \mathcal{M} \). If \( m \geq 3 \), then we have
\[
|\mathcal{F}| \leq |\mathcal{F}'|,
\]
and for \( k \geq 5 \), the equality is possible if and only if \( \mathcal{F} \) is isomorphic to \( \mathcal{F}' \). Moreover, if \( \mathcal{F} \) is as above and
$|\bigcap_{F \in \mathcal{F}(\bar{1})} F| \leq m
$
for some \( m \geq 3 \), then
\begin{equation} \label{eq-Kup-more}
|\mathcal{F}| \leq |\m J_{k-m+1}|,
\end{equation}
and, for \( k \geq 5 \), equality is possible only if \( \mathcal{F} \) is isomorphic to \( \m J_{k-m+1} \).
\end{theorem}


 For two families $\m A$ and $ \m B$, we say that they are {\it cross-intersecting} if $A\cap B\ne \emptyset$ for any $A\in \m A$ and $B\in \m B$.  
In order to establish the stability of Erd\H{o}s--Ko--Rado theorem, Hilton and Milner \cite{HM1967} had to deal with cross-intersecting families. Actually, problems on cross-intersecting families naturally arise in this context, e.g., when decomposing an intersecting family into the part that contains a given element and the part that does not contain it. Another connection between cross-intersecting and intersecting families is via the famous Kruskal--Katona theorem \cite{12-1-2, 12-1-3}. As it was noted by Hilton, 
the following theorem can be obtained by the Kruskal--Katona theorem.

\begin{theorem}[Daykin \cite{Day1973}]
    Let $n > a+b$ be integers and 
$\m A\subset {[n]\choose a}$, $\m B\subset {[n]\choose b}$ be cross-intersecting families. 
If $|\m A|\ge {n-1\choose a-1}$ and $|\m B| \ge {n-1\choose b-1}$, 
then both $\mathcal{A}$ and $\m B$ are full stars with the same center. 
\end{theorem}

Frankl and Kupavskii \cite{FK2021jcta} managed to generalize Theorem \ref{thm:FKZ} to cross-intersecting families.

\begin{theorem}[Frankl--Kupavskii \cite{FK2021jcta}] 
\label{thm:FK}
Let \( a \), \( b \) be positive integers, \( n \geq a + b \), and let \( \mathcal{A} \subset \binom{[n]}{a} \) and \( \mathcal{B} \subset \binom{[n]}{b} \) be cross-intersecting. Suppose that
\begin{align*}
|\mathcal{A}| & \geq \binom{n-1}{a-1} - \binom{n-v-1}{a-1} + \binom{n-u-1}{n-a-1}, \\
|\mathcal{B}| &\geq \binom{n-1}{b-1} - \binom{n-u-1}{b-1} + \binom{n-v-1}{n-b-1}
\end{align*}
for some real \( 3 \leq u \leq a \), \( 3 \leq v \leq b \) and, moreover, that at least one of the displayed inequalities is strict. Then
\[
\gamma(\mathcal{A}) < \binom{n-u-1}{n-a-1} \quad \text{and} \quad \gamma(\mathcal{B}) < \binom{n-v-1}{n-b-1},
\]
moreover, both families share a (unique) element of maximum degree.
\end{theorem}

The bound of Theorem \ref{thm:FK} is tight when $u,v$ are integers, as seen 
 by setting $\m A = \m H_{a,v,u}$ and $ \m B = \m H_{b,u,v}$. Actually, Theorem~\ref{thm:FK} is significantly harder to prove, but the effort pays off: Theorem~\ref{thm:FK} is more powerful and already found a couple of applications in questions on intersecting families, e.g., in the paper \cite{FK2020}, but also in the questions on intersecting antichains (private communication with Bal\'azs Patk\'os). With some extra effort, Frankl and Kupavskii \cite{FK2021jcta} managed to extend Theorem~\ref{thm:FK} to a sharp stability result for the Kruskal--Katona theorem.

Let $x$ be an element of maximum degree in $\mathcal{F}$ (chosen arbitrarily if there are multiple choices). Denote by $\mathcal{F}_\Delta = \mathcal{F}(x)$ the {\it degree part}, and by $\mathcal{F}_\gamma = \mathcal{F}(\bar x)$ the {\it diversity part}.   
For two cross-intersecting families $\m F\subset {[n]\choose k}$ and $\m G \subset {[n]\choose \ell}$, we say that {\it $\m F$ is  maximal cross-intersecting with $\m G$} if there is no set $F\in {[n]\choose k}\setminus \m F$ such that $\m F \cup \{F\}$ and $\m G$ are still cross-intersecting.
We say that 
{\it $\m F$ and $\m G$ are maximal cross-intersecting families}, if, for any pair of cross-intersecting families $\m F'\subset {[n]\choose k}$ and $\m G'\subset {[n]\choose \ell}$ with $\m F\subset \m F'$ and $\m G\subset \m G'$, we necessarily have $\m F= \m F'$ and $\m G= \m G'$.
The main result of this paper is a generalization of Theorem~\ref{11-23-2} to the case of cross-intersecting families.

\begin{theorem}[Main result]\label{thm:main}
Let $a\ge b\ge 4$, $n> a+b$ be integers, $\mathcal{A}\subset {[n]\choose a}$ and $\mathcal{B}\subset {[n]\choose b}$ be cross-intersecting families. 
Take intersection-minimal subfamilies $\mathcal{M}\subset \mathcal{A}_\gamma$ and $\mathcal{T}\subset \mathcal{B}_\gamma$ with $|\cap_{M\in \mathcal{M}}M|\ge a-b+3$ and $|\cap_{T\in \mathcal{T}}T|\ge 3$.
Define cross-intersecting families $\mathcal{A}',\mathcal{B}'$ as follows:
they share the same maximum-degree element; 
their diversity parts are isomorphic to $\mathcal{M}$ and $\mathcal{T}$, respectively; 
 and their degree parts are maximal cross-intersecting with $\m T$ and $\m M$, respectively. 
 If 
\[ |\mathcal{A}|\ge |\mathcal{A}'|  \quad \text{and} \quad |\mathcal{B}|\ge |\mathcal{B}'|, \] 
 then their elements of maximal degree coincide, $|\mathcal{A}|= |\mathcal{A}'|$ and $|\mathcal{B}|= |\mathcal{B}'|$. 
\end{theorem}

It is also possible to characterize equality cases (which are in most cases isomorphic to $\m F', \m G'$), but this is tedious and we decided to omit it. While Theorem~\ref{thm:FK} with $u=a, v=b$ gives a cross-intersecting variant of the Hilton--Milner Theorem,
Theorem \ref{thm:main} allows us to obtain cross-intersecting variants of the Han--Kohayakawa theorem and the Huang--Peng theorem, as well as the second part of Theorem~\ref{11-23-2}, i.e. (\ref{eq-Kup-more}). We state those as a corollary without proof. 
For $m \in [a]$ and $t \in [b]$,
we take $A_1, A_2\in {[2,n]\choose a}$ such that $A_1\cap A_2=[2,m+1]$. 
Similarly, choose $B_1, B_2\in {[2,n]\choose b}$ such that $B_1\cap B_2=[2,t+1]$.
Define
\begin{align*}
&\mathcal{A}^* :=\Big\{H\in \tbinom{[n]}{a}: 1\in H, H\cap B_1 \ne \emptyset, H\cap B_2 \ne \emptyset \Big\} \cup \{A_1, A_2\},\\
&\mathcal{B}^* :=\Big\{H\in \tbinom{[n]}{ b}: 1\in H, H\cap A_1 \ne \emptyset, H\cap A_2 \ne \emptyset \Big\} \cup \{B_1, B_2\}.
\end{align*}

\begin{corollary}\label{12-4-1}
Let $a\ge b\ge 5$, $n> a+b$ be integers, $\mathcal{A}\subset {[n]\choose a}$ and $\mathcal{B}\subset {[n]\choose b}$ be cross-intersecting families  
such that $|\cap_{A\in \mathcal{A}_\gamma}A|\le m$ and $|\cap_{B\in \mathcal{B}_\gamma}B|\le t$, where $m\ge a-b+3$ and $t\ge 3$.
Suppose that $|\mathcal{A}|\ge |\mathcal{A}^*|$ and $|\mathcal{B}|\ge |\mathcal{B}^*|$, then $|\mathcal{A}|=|\mathcal{A}^*|$ and $|\mathcal{B}|=|\mathcal{B}^* |$.
\end{corollary}

\noindent{\bf A new shifting-based technique.}
By refining the classical Daykin shift, we introduce a new shifting method, called the $S_{U,V}^Q$-shift. Traditional shifting operations preserve the size of the family and its overall properties (intersection, cross-intersection, etc.), whereas our method additionally guarantees that certain local substructures remain preserved after shifting. Moreover, the method applies to very general types of substructures. 

The main purpose of developing this method is to prove Proposition~\ref{10-25-3}, which constitutes one of the most technically involved proof in this paper. The method has already prove to be useful: in a separate work, the first author applied it  to solve the product version of the Hilton--Milner problem \cite{huang-product}. We believe that this method has substantial potential for further development and possesses independent research significance. We will introduce it in Section \ref{secpro3}.

\noindent{\bf Organization.}
In Section \ref{secpre}, we provide the necessary preliminaries.
In Section \ref{secmain}, we present the proof of the main theorem. First, modulo some simpler lemmas, we give a proof of the main theorem in the warm-up case $a=b\ge 5$. Second, assuming several Lemmas, we provide the complete proof of Theorem \ref{thm:main}.
The following sections are devoted to proving the lemmas. 

\section{Preliminaries}\label{secpre}

For a family $\mathcal{F}$, denote by $\partial_{\ell}\mathcal{F}$ the {\it $\ell$-shadow of $\m F$}: the collection of all $\ell$-sets contained in at least one set from $\mathcal{F}$. If $\mathcal{F} \subset {[n]\choose k}$,
then we abbreviate $\partial_{k-1}\mathcal{F}$ as $\partial \mathcal{F}$.
For a set $F$, we denote by $\min F$ and $\max F$ the smallest and largest element of $F$, respectively.
We use the lexicographic order, or lex order for short, on $2^{[n]}$.
For two sets $A,B\subset [n]$, we write $A\prec B$ if either
$B\subset A$, or $\min(A\setminus B)<\min(B\setminus A)$.
Throughout this paper, the notation $A\prec B$ always refers to
the lex order. This is different from the colexicographic order,
or colex order, where one usually writes $A\prec B$ if either
$A\subset B$, or $\max(A\setminus B)<\max(B\setminus A)$.

We use $\mathcal{L}(X,r,k)$ and $\mathcal{C}(X,r,k)$ to denote the families consisting of the first $r$ $k$-sets of $X$ in lexicographic (lex) and colexicographic (colex) order, respectively. For a family $\mathcal{F} \subset \binom{[n]}{k}$, we say that $\mathcal{F}$ is \textit{$L$-initial on $[n]$} if $\mathcal{F} = \mathcal{L}([n], |\mathcal{F}|, k)$; otherwise, we say $\mathcal{F}$ is \textit{not $L$-initial}. Similarly, a pair of families $(\mathcal{A}, \mathcal{B})$ is \textit{not $L$-initial} if at least one of $\mathcal{A}$,  $\mathcal{B}$ is not $L$-initial.

Recall  the classical Kruskal--Katona theorem \cite{12-1-2, 12-1-3}. 
\begin{theorem}[Kruskal--Katona Theorem]
If $\mathcal{F}\subset {[n]\choose k}$, then $|\partial \mathcal{F}|\ge|\partial \mathcal{C}([n], |\mathcal{F}|,k)|$. 
\end{theorem}

We will use the following version of the Kruskal--Katona theorem.

\begin{theorem}[Daykin \cite{12-1-4}]\label{12-1-1}
Let $n\ge a+b$, and let $\mathcal{A}\subset {[n]\choose a},\mathcal{B}\subset {[n]\choose b}$ be cross-intersecting. Then $\mathcal{L}([n], |\mathcal{A}|,a)$, $\mathcal{L}([n], |\mathcal{B}|,b)$ are also cross-intersecting.
\end{theorem}

As a corollary of Theorem \ref{12-1-1}, we get the following.

\begin{proposition}\label{1-4-1}
Let $n\ge a+b$ and $i$ be integers, and let $\m A \subset {[n]\choose a}$ and $\m B \subset {[n]\choose b}$ be cross-intersecting. If $|\m A|\ge {n\choose a}-{n-i\choose a}$, then $|\m B|\le {n-i\choose b-i}$.
\end{proposition}
\begin{proof}
By Theorem \ref{12-1-1}, we may assume that $\m A=\m L([n], |\m A|, a)$ and $\m B=\m L([n], |\m B|, b)$. Since $|\m A|\ge {n\choose a}-{n-i\choose a}$, all $a$-sets $A\subset [n]$ with $A\cap [i]\ne \emptyset$ belong to $\m A$. 
As $\m A$ and $\m B$ are cross-intersecting and $n\ge a+b$,  every set $B\in \m B$ must contain $[i]$. This implies that $|\m B|\le {n-i\choose b-i}$.
\end{proof}

\begin{theorem}[Lov\'{a}sz \cite{12-2-3}]\label{12-2-2}
If $\mathcal{F}\subset {[n]\choose k}$ is such that  $|\mathcal{F}|={x\choose k}$ for some real $x\ge k$, then 
$$|\partial \mathcal{F}|\ge {x \choose k-1}. $$
\end{theorem}

For $1\le i <j \le n$ and $F\subset [n]$, the classical $(i,j)$-shift, introduced by Erd\H{o}s, Ko, and Rado, is defined as follows:
$S_{i,j}(F):=(F\setminus \{j\}) \cup \{i\}$ if $F\cap \{i,j\}=\{j\}$, and $S_{i,j}(F):=F$ otherwise. For a family $\m F$, set 
\[
S_{i,j}(\mathcal{F}) := \{ S_{i,j}(F) : F \in \mathcal{F} \} \cup \{ F \in \mathcal{F} : S_{i,j}(F) \in \mathcal{F} \}.
\]
Daykin \cite{Day1973} gave an alternative proof of the Kruskal--Katona theorem. In that paper, he introduced an important notion of $S_{U,V}$-shifts, which is a generalization of an $S_{i,j}$-shift. Let $U,V$ be disjoint sets of the same size. Define $S_{U,V}(F):=(F\setminus V) \cup U$ if $F\cap (V\cup U)=V$, and $S_{U,V}(F):=F$ otherwise. For a family, define 
\[
S_{U,V}(\mathcal{F}) := \{ S_{U,V}(F) : F \in \mathcal{F} \} \cup \{ F \in \mathcal{F} : S_{U,V}(F) \in \mathcal{F} \}.
\]
Throughout the paper, when using the notation $S_{U,V}$, we always assume that $U,V$ have the same size and are disjoint.
We say that $S_{U,V}$-shift {\it acts trivially on $\m F$} (or on a pair of families $(\m A, \m B)$) if $S_{U,V}(\m F)=\m F$ (resp. $(S_{U,V}(\m A), S_{U,V}(\m B))=(\m A, \m B)$).

For the lex order on $\binom{[n]}{k}$, let $i_{\text{lex}}(S)$ denote the position of a set $S$ in this order. For two families $\mathcal{F}_1, \mathcal{F}_2 \subset \binom{[n]}{k}$, we write 
$\mathcal{F}_1 \prec \mathcal{F}_2$ if 
$
\sum_{S \in \mathcal{F}_1} i_{\text{lex}}(S) \leq \sum_{S \in \mathcal{F}_2} i_{\text{lex}}(S).
$
In case of strict inequality, we write  $\mathcal{F}_1 \precneqq \mathcal{F}_2$.
For pairs $(\mathcal{F}_1, \mathcal{G}_1)$ and  $(\mathcal{F}_2, \mathcal{G}_2)$, we write $(\mathcal{F}_1, \mathcal{G}_1) \precneqq  (\mathcal{F}_2, \mathcal{G}_2)$ to mean $\mathcal{F}_1 \prec \mathcal{F}_2$, $\mathcal{G}_1 \prec \mathcal{G}_2$, and at least one of $\mathcal{F}_1 \precneqq \mathcal{F}_2$ and  $\mathcal{G}_1 \precneqq \mathcal{G}_2$ holds.

For a family $\m F\subset {[n]\choose k}$ (or a pair of families $\mathcal{A}\subset {[n]\choose a}$ and $\mathcal{B}\subset {[n]\choose b}$), 
we can find a pair of sets $(U,V)$ satisfying the following property whenever $\m F$ (resp. $(\mathcal{A},\mathcal{B})$) is not $L$-initial.

\begin{itemize}
\item[{\bf P}] Let $U,V\subset [n]$ satisfy 
$U\cap V=\emptyset$, $|U|=|V|$ and $U\prec V$. 
We say that $(U,V)$ satisfy {\bf P} for a family $\m F$ if
 $S_{U,V}(\m F)\precneqq \m F$, and 
for any $V'\subsetneqq V$ and any $U'\subsetneqq U$ with $U'\prec V'$ and $|U'|=|V'|$, we have $S_{U',V'}(\m F)=\m F$. 
\end{itemize}

We say that a pair $(U,V)$ satisfy {\bf P} for a pair of families $(\m A, \m B)$ if $(U,V)$ satisfy {\bf P} for at least one of $\m A$ and $\m B$.
 
The following lemma captures the key idea of the proof of Daykin~\cite{Day1973}, adapted to cross-intersecting families.

\begin{lemma}\label{SUV}
Let $n\ge a+b$, \(\mathcal{A}\subset {[n]\choose a}\) and \(\mathcal{B}\subset {[n]\choose b}\) be cross-intersecting.
Let $(U,V)$ satisfy {\bf P} for $(\m A, \m B)$. Then \(S_{U,V}(\mathcal{A})\) and \(S_{U,V}(\mathcal{B})\) are also cross-intersecting.
\end{lemma}

Hilton \cite{Hilton1977} showed that if 
$\m A,\m B\subset {[n]\choose k}$ are cross-intersecting and $n\ge 2k\ge4$, then
$|\m A|+|\m B|\le {n\choose k}.$
Frankl \cite{Frankl-max-degree} extended Hilton's result as follows.
\begin{theorem}[Frankl \cite{Frankl-max-degree}]\label{12-3-3}
Let \( a, b, n \) be nonnegative integers, \( n > a + b \), \( a \geq b \).  
Suppose that \( |Y| = n \), \( \mathcal{A} \subset \binom{Y}{a} \) and \( \mathcal{B} \subset \binom{Y}{b} \) are cross-intersecting. Then  
\[
|\mathcal{A}| + |\mathcal{B}| \leq \binom{n}{a}
\]  
with equality holding if and only if \( \mathcal{A} = \binom{Y}{a} \), \( \mathcal{B} = \emptyset \) or  
\( a = b \), \( \mathcal{A} = \emptyset \) and \( \mathcal{B} = \binom{Y}{b} \).
\end{theorem}

Hilton and Milner \cite{HM1967} proved that for two non-empty $n$-vertex $k$-uniform cross-intersecting families with $n\ge 2k$, the sum of their sizes is at most ${n\choose k}-{n-k\choose k}+1$.
Frankl and Tokushige \cite{FT1992} generalized this result as follows.

\begin{theorem}[Frankl--Tokushige \cite{FT1992}]\label{12-3-2}
If $\mathcal{A} \subset \binom{[n]}{a}$ and $\mathcal{B} \subset \binom{[n]}{b}$ are non-empty cross-intersecting families with $n > a + b$ and $a \ge b$, then
\[
|\mathcal{A}| + |\mathcal{B}| \leq \binom{n}{a} - \binom{n-b}{b} + 1.
\]
If $a>b$, then the equality holds only if $\mathcal{B} = \{B\}$ and $\mathcal{A} = \{A \in \binom{[n]}{a} : A \cap B \neq \emptyset \}$.
\end{theorem}

In the proof of Lemma \ref{lemkey}, we shall need the following result. We note that a similar result, but with a weaker bound on $\m A$, appeared in \cite{KZ2018}.



\begin{theorem}\label{10-12-1}
Let $n\ge a+b$, $a\ge b$ and $\mathcal{A}\subset {[n]\choose a}$ and $\mathcal{B}\subset {[n]\choose b}$ be cross-intersecting families. If
$|\mathcal{A}|\leq {n-a+b \choose b},$
then
$|\mathcal{A}|+|\mathcal{B}|\leq {n \choose b}.$
\end{theorem}
The above result strengthens a result of Kupavskii and Zakharov \cite{KZ2018}, which assumes the bound $|\mathcal{A}|\leq {n-a+b-1 \choose b-1}$. We prove it via the local unimodality lemma of Huang and Peng \cite{12-3-1}; the argument is self-contained and is given at the end of the paper.

\section{Proof of Theorem \ref{thm:main}}\label{secmain}

For a family $\m F$, we denote by $\cap\, \m F$ the common intersection of $\m F$, i.e. $\cap\,\m F=\cap_{F\in \m F}F$.
Two families $\mathcal{A}$ and $\mathcal{B}$ are called {\it non-empty cross-intersecting} if they are cross-intersecting and $\mathcal{A}, \mathcal{B} \neq \emptyset$.

Let us take the families $\m A, \m B$ as in the statement of the theorem. First, we give the proof in the much simpler case $a=b\ge 5$.
\subsection{Warm-up: the proof in the case $a=b\ge 5$}

In the case $a=b$, we aim to reduce the problem for two cross-intersecting families to that for a single intersecting family, enabling us to apply Theorem \ref{11-23-2}.
We construct two families as follows:
\begin{equation}\label{deff1f2}
\m F_1=\{A\in \m A: 1\in A\}\cup \m B(\bar{1}), \,\, \m F_2=\{B\in \m B: 1\in B\}\cup \m A(\bar{1}).
\end{equation}
Our goal is to establish the following two inequalities:  
\begin{equation}\label{f1f2}
    |\m F_1|\le |\m A'(1)|+|\m T|, \,\, |\m F_2|\le |\m B'(1)|+|\m M|.
\end{equation}
Once (\ref{f1f2}) is proved, we have 
$$|\m A|+|\m B|=|\m F_1|+|\m F_2|\le|\m A'(1)|+|\m T| + |\m B'(1)|+|\m M|=|\m A'|+|\m B'|.$$
Together with the assumptions $|\m A|\ge |\m A'|$ and $|\m B|\ge |\m B'|$, this forces $|\m A|= |\m A'|$ and $|\m B|= |\m B'|$, as required. 
To verify (\ref{f1f2}), 
note that if $\m F_1$ were intersecting, with $\Delta(\m F_1)=|\m F_1(1)|$ and $\m T \subset \m B(\bar1) \subset \m F_1(\bar{1})$, then Theorem \ref{11-23-2} would directly yield $|\m F_1|\le |\m A'(1)|+|\m T|$;
a similar argument works for $\mathcal{F}_2$. 
However, the reduction is obstructed by the fact that $\m F_1, \m F_2$ are not guaranteed to be intersecting, since $\m A(\bar1), \m B(\bar1)$ may not be intersecting. 
To overcome this difficulty, we employ the following key lemmas.

The first lemma provides lower bounds on the sizes of $\m A'$ and $\m B'$ (and is valid for not necessarily equal $a,b$). 

\begin{lemma}\label{11-23-3}
Let $\mathcal{A}', \mathcal{B}'$ be as in Theorem~\ref{thm:main}. Then
\begin{align}\label{11-30-1}
&|\mathcal{A}'|\ge {n-1\choose a-1}-{n-4\choose a-1}+{n-1-(a-b+3)\choose a-(a-b+3)},\\ 
\label{11-30-2}
&|\mathcal{B}'|\ge {n-1\choose b-1}-{n-4\choose b-1}+{n-4\choose b-3},
\end{align}
with the following exceptional cases:
\begin{itemize}
\item  $a=b=4$, $|\mathcal{M}|=2$ and $|\mathcal{T}|=1$. In this case, we have  $|\mathcal{B}'| \ge {n-1\choose b-1}-{n-4\choose b-1}+{n-4\choose b-3}-1$ and  $|\mathcal{A}'| > {n-1\choose a-1}-{n-4\choose a-1}+{n-4\choose a-3}$. 
\item $a=b=4$, $|\mathcal{T}|=2$ and $|\mathcal{M}|=1$. In this case, we have  $|\mathcal{A}'| \ge {n-1\choose a-1}-{n-4\choose a-1}+{n-4\choose a-3}-1$ and  $|\mathcal{B}'| > {n-1\choose b-1}-{n-4\choose b-1}+{n-4\choose b-3}$. 
\end{itemize}
\end{lemma}

The second lemma provides an upper bound on the sum of sizes of non-empty cross-intersecting families.

\begin{lemma}\label{lemf}
Let $|X|\ge 2k\ge 8$, $\m F\subset {X\choose k}$ and $\m G\subset {X\choose k-1}$ be non-empty cross-intersecting. 
Assume that  $\m F$ contains an intersection-minimal subfamily $\m Y$, and let $\m G'\subset{X\choose k-1}$ be the family of all sets that intersect every member of $\m Y$. 
If $|\m F|\le {|X|-3\choose k-3}$ and $|\cap \, \m Y|\ge 3$,
then  
$|\m F|+|\m G|\le |\m Y|+|\m G'|.$
\end{lemma}

\begin{remark*}
Note that Lemma \ref{lemf}, although not explicitly stated in the works \cite{Kup2018, 11-23-2}, can be obtained by the same argument that is used in the proof of Theorem~\ref{11-23-2} (see also \cite{Kup2018, 11-23-2}). 
In this paper we prove a more general statement—Lemma \ref{lemkey}—which implies Lemma \ref{lemf}.
\end{remark*}

Now, we give a proof of Theorem~\ref{thm:main} for  $a=b\ge 5$, assuming that  Lemmas \ref{11-23-3} and \ref{lemf} hold.

\begin{proof}[{\bf Proof of Theorem \ref{thm:main} for  $a=b\ge 5$}]
Since $a=b\ge 5$, Lemma \ref{11-23-3} gives $|\m A|\ge |\m A'|\ge {n-1\choose a-1}-{n-4\choose a-1}+{n-4\choose a-3}$ and $|\m B|\ge |\m B'|\ge {n-1\choose b-1}-{n-4\choose b-1}+{n-4\choose b-3}$. Applying Theorem \ref{thm:FK} with $a=b$ and $u=v=3$ yields
\begin{equation*}
    \gamma(\m A), \gamma(\m B) \le {n-4\choose a-3}={n-4\choose b-3},
\end{equation*}
and both families $\m A, \m B$ share the same element of maximum degree. W.l.o.g., we assume $\m A_{\Delta}=\m A(1)$ and  $\m B_{\Delta}=\m B(1)$. Then $\m M \subset \m A(\bar1)$ and $\m T \subset \m B(\bar1)$.

Let $\m F_1, \m F_2$ be defined as in (\ref{deff1f2}). As outlined at the beginning of this subsection, it suffices to establish \eqref{f1f2}.
Note that $\m A(1)\subset {[2,n]\choose a-1}$ and $\m B(\bar{1}) \subset {[2,n]\choose b}$ are cross-intersecting.  
Applying Lemma \ref{lemf} with
\[
\mathcal{F}=\mathcal{B}(\bar1),\; \mathcal{G}=\mathcal{A}(1),\; \mathcal{G}'=\mathcal{A}'(1),\;
\mathcal{Y}=\mathcal{T},\; X=[2,n],\; k=b=a,
\]
we obtain 
$
|\m A(1)|+|\m B(\bar{1})|\le |\m T|+|\m A'(1)|
$.
Consequently, $|\m F_1|=|\m A(1)|+|\m B(\bar{1})|\le |\m A'(1)|+|\m T|$. 
An analogous argument gives $|\m F_2|\le |\m B'(1)|+|\m M|$, thereby verifying \eqref{f1f2}. 
Since $|\mathcal{A}| \ge |\mathcal{A}'|$ and $|\mathcal{B}| \ge |\mathcal{B}'|$, this forces $|\mathcal{A}| = |\mathcal{A}'|$ and $|\mathcal{B}| = |\mathcal{B}'|$, which completes the proof for the case $a=b\ge 5$.
\end{proof}

\subsection{The general case}
In this subsection, we give a proof of Theorem \ref{thm:main} for all $a\ge b$ modulo several lemmas.
We begin with a key lemma that generalizes Lemma \ref{lemf} and captures the core argument from the proof of Theorem \ref{11-23-2}. 

\begin{lemma}\label{lemkey}
Let $X$ be a set with $|X|>k+r$ and $\min \{k,r\}\ge 3$ and let $\m F\subset {X\choose k}$ and $\m G\subset {X\choose r}$ be non-empty cross-intersecting. 
Suppose $\m F$ contains an intersection-minimal subfamily $\m Y$, and let $\m G'\subset{X\choose r}$ be the family of all sets that intersect every member of $\m Y$. Then the following hold.
\begin{itemize}
\item[(i)] If $k<r$ and $|\cap \, \m Y|\ge 1$, then  
$|\m F|+|\m G|\le |\m Y|+|\m G'|.$
\item[(ii)] If  $k\ge r$, 
$|\m F|\le {|X|-(k-r+1)\choose r-1}$ and $|\cap \, \m Y|\ge k-r+2$,
then the same inequality holds.
\end{itemize}
\end{lemma}

\begin{remark*}
Lemma \ref{lemkey} generalizes Lemma \ref{lemf}. 
Setting $r=k-1$ recovers the setting of Lemma \ref{lemf}, while crucially relaxing the upper bound on $|\m F|$ from ${|X|-3\choose k-3}$ to ${|X|-2\choose k-2}$. This relaxed bound extends its applicability and is essential for the proof of Lemma \ref{lemmain}.
\end{remark*}

In view of the proof for the $a=b$ case, a natural approach to the proof of the general case, is to try to  directly derive Theorem~\ref{thm:main} by combining Lemma~\ref{11-23-3}, Theorem~\ref{thm:FK} and Lemma~\ref{lemkey}. However, this approach fails. The underlying reason is that the lower bounds on $|\m A'|, |\m B'|$ are not strong enough to  satisfy the requirements of Theorem \ref{thm:FK}. 

Nevertheless, we can still apply Lemma \ref{lemkey} to handle the case where the diversity of each family is relatively small; see Lemma \ref{10-25-2}. 
Before stating this result, we simplify the structure of $\m A', \m B'$ from the theorem. W.l.o.g., we assume that the element of maximum degree in both families $\m A', \m B'$ is $1$. 
Furthermore, we may w.l.o.g. assume that $\cap\, \m A'(\bar1) = [2,m+1]$ and $\cap\, \m B'(\bar 1) = [2,t+1]$.
This normalization is justified because $\mathcal{M}$ and $\mathcal{T}$ are intersecting, and the only non-trivial cross-intersection constraints we need to preserve are between $(\mathcal{B}'(1), \mathcal{M})$ and $(\mathcal{A}'(1), \mathcal{T})$.
For clarity, we summarize the assumptions of Theorem~\ref{thm:main} and this simplifying assumption below.

\begin{assumption}\label{12-14-2}
Let $a\ge b\ge 4$, $n> a+b$ be integers, $\mathcal{A}\subset {[n]\choose a}$ and $\mathcal{B}\subset {[n]\choose b}$ be maximal cross-intersecting families. 
Take intersection-minimal subfamilies $\mathcal{M}\subset \mathcal{A}_\gamma$ and $\mathcal{T}\subset \mathcal{B}_\gamma$ such that  $m:=|\cap\, \mathcal{M}|\ge a-b+3$ and  $t:=|\cap\,\mathcal{T}|\ge 3$. Choose such a pair $(\mathcal A,\mathcal B)$ maximizing
$|\mathcal A|+|\mathcal B|$.

Take families  $\m A'$, $\m B'$ such that 
\begin{itemize}
\item The diversity parts $\m A'_\gamma, \m B'_\gamma\subset 2^{[2,n]}$ are isomorphic to $\m M, \m T,$ respectively, with $\cap \m A'_\gamma = [2,m+1]$ and $\cap \m B'_\gamma = [2,t+1]$.
\item Both families have $1$ as the element of maximum degree. Moreover, up to isomorphism,
$\m A'(1)$ and $\m B'(1)$ are maximal  cross-intersecting with $\m T$ and $\m M$, respectively. 
\end{itemize}
\end{assumption}

We emphasize that Assumption~\ref{12-14-2} can be seen  as a part of the setup for Theorem~\ref{thm:main}. 
Now we give the statement of Lemma \ref{10-25-2}.

\begin{lemma}\label{10-25-2}
In Assumption \ref{12-14-2}, if, additionally, both families $\m A$ and $\m B$ share an element of maximum degree, $\gamma(\mathcal{A})\le {n-1-(a-b+2)\choose a-(a-b+2)}$ and $\gamma(\mathcal{B})\le {n-3\choose b-2}$, 
then 
$
|\mathcal{A}|+|\mathcal{B}|\le |\mathcal{A}'|+|\mathcal{B}'|.
$
\end{lemma}

As noted earlier, we cannot apply Theorem~\ref{thm:FK} directly to obtain the diversity bounds satisfying the assumptions of Lemma \ref{10-25-2}. To bridge this gap,
we establish the following key lemma, which reduces the proof of Theorem~\ref{thm:main} to the case covered by Lemma~\ref{10-25-2}. 

\begin{lemma}\label{lemmain}
Under Assumption~\ref{12-14-2}, if additionally, $|\m A|\ge |\m A'|$ and $|\m B|\ge |\m B'|$, then there exist cross-intersecting families $\m A^a\subset {[n]\choose a}$, $\m B^b\subset {[n]\choose b}$ such that the following hold:
\begin{enumerate}
\item $|\m A^a|\ge |\m A|$ and $ |\m B^b|\ge |\m B|$;
\item  $\Delta(\mathcal{A}^a)=|\m A^a(i)|\ge {n-2\choose a-2}+\gamma(\mathcal{A}^a)$ and   $\Delta(\mathcal{B}^b)=|\m B^b(i)|\ge {n-2\choose b-2}+\gamma(\mathcal{B}^b)$ for a common $i\in [n]$;
\item $\m A^a_\gamma$ and $\m B^b_\gamma$  contain isomorphic copies of $\m M$ and $\m T$, respectively;
\item $\gamma(\m A^a)\le {n-1-(a-b+2)\choose a-(a-b+2)}$ and $\gamma(\m B^b)\le {n-3\choose b-2}$.
\end{enumerate}
\end{lemma}

In the proof of Lemma~\ref{lemmain}, we construct the families $\m A^a$ and $\m B^b$ gradually. 
While the first two conditions are not difficult to derive from the bounds from Lemma~\ref{11-23-3}, 
a major challenge is to simultaneously reduce the diversity of both families while preserving copies of $\m M$ and $\m T$ within the diversity parts. Other essential properties, such as sizes and the cross-intersecting condition, must also be maintained throughout.
To this end, we introduce a refined version of Daykin's $S_{U,V}$-shift method tailored to our setting and analyze its key properties in Section~\ref{tool}.

We now prove Theorem \ref{thm:main}, assuming Lemmas~\ref{lemmain} and ~\ref{10-25-2}.
\begin{proof}[{\bf Proof of Theorem \ref{thm:main}}]
Apply Lemma~\ref{lemmain} to obtain families $\m A^a$ and $\m B^b$ that satisfy the conditions 1--4 in Lemma~\ref{lemmain}. 
Then $\m A^a$ and $\m B^b$ also satisfy the assumptions of Lemma~\ref{10-25-2}. Applying Lemma~\ref{10-25-2} with $\m A^a$ and $\m B^b$ playing the roles of $\m A$ and $\m B$, respectively, yields $|\m A|+|\m B|\le |\m A^a|+|\m B^b|\le |\m A'|+|\m B'|$.
Together with the assumptions $|\m A|\ge |\m A'|$ and $|\m B|\ge |\m B'|$, this forces $|\m A|=|\m A'|$ and $|\m B|=|\m B'|$.
This completes the proof of Theorem \ref{thm:main}. 
\end{proof}

In the forthcoming sections, we shall prove Lemmas \ref{11-23-3}, \ref{10-25-2} and \ref{lemmain}.

\section{Establishing the lower bounds: Proof of Lemma \ref{11-23-3}}\label{lower}
Let us introduce some notation.  For a family $\m F$ and sets $X, Y$ with $X\subset Y$, we use the following standard conventions:
\begin{align*}
    \mathcal{F}[X, Y] &:=\big\{F: F\cap Y = X,\; F\in \mathcal{F} \big\}, &
    \mathcal{F}(X, Y) &:=\big\{F\setminus Y: F\cap Y = X,\; F\in \mathcal{F} \big\},\\[3pt]
    \mathcal{F}(X) &:=\{F\setminus X: X\subset F, F\in \m F\}, &
    \mathcal{F}(\bar X) &:=\{F: F\cap X=\emptyset, F\in \m F\}.
\end{align*}
We say that a set $C$ is a {\it cover} of a family $\m F$ if $C$ intersects every member of $\m F$.

Before proving Lemma \ref{11-23-3}, we make some simplifying assumptions. Since $\m A'$ and $\m B'$ have the same element of maximum degree,  we may assume w.l.o.g. that $\m A'_{\Delta}=\m A'(1)$ and $\m B'_{\Delta}=\m B'(1)$. Moreover, we may assume $\m A'(\bar{1})=\m M$ and $\m B'(\bar{1})=\m T$ with $$[2,m+1]=\cap \, \m M,\ \ \ \ \ \ [2,t+1]=\cap \, \m T.$$

Write $\mathcal{M}=\{M_1, \dots, M_{k}\}$.
Since $\mathcal{M}$ is intersection-minimal, for each $i\in [k]$ there exists 
$
p_i\in \cap\,(\mathcal{M}\setminus \{M_i\})\setminus \cap \,\mathcal{M}$;
w.l.o.g. $\{p_1, \dots, p_k\}=[m+2, m+1+k]$.

Write $\m T=\{T_1, T_2, \dots, T_{\ell}\}$. Since $\mathcal{T}$ is intersection-minimal, for each $i\in [\ell]$ there exists $q_i\in \cap\,(\mathcal{T}\setminus \{T_i\})\setminus (\cap \,\mathcal{T})$;
w.l.o.g. $\{q_1, \dots, q_{\ell}\}=[t+2, t+\ell+1]$.

In the proof of Lemma \ref{11-23-3}, 
we will use the following two simple claims. 

\begin{claim}\label{minsize}
    Let $\m F \subset {[n]\choose k}$ be intersection-minimal with $|\cap \m F|=s$. Then $|\m F|\le k-s+1$.
\end{claim}

\begin{proof}
Let $I:=\cap \, \m F$ and $\m G:=\m F(I):=\{G_1, G_2, \dots, G_m\}$. 
Then $\cap \, \m G=\emptyset$, and for every $i\in [m]$, there exists $k_i\in \cap (\m G\setminus G_i)$. Moreover, $k_i\ne k_j$ for each $i\ne j$ and $i,j \in [m]$. 
Thus, for every $i\in [m]$, we have $\{k_1, k_2, \dots, k_m\}\setminus \{k_i\} \subset G_i$. Thereby, $|G_i|\ge m-1$. 
Note that $G_i \in \m F(I)$ and $|I|=s$, we have $|G_i|=k-s$.
Thus, $|\m G|=m\le k-s+1$.
\end{proof}

We know that for integers $a>b$, if $n\ge a+b$, then ${n\choose a}\ge {n\choose b}$, with equality if and only if $n=a+b$. From this, we immediately obtain the following claim.
\begin{claim}
    If $n\ge 2k+1$ and $1\le i<j\le k$, then 
    \begin{equation}\label{n>2k}
    {n-i-j\choose k-i}-{n-i-j\choose k-j}\ge k-i.
    \end{equation}
\end{claim}

Define 
\begin{equation}\label{hki}
\m H_{k,i}:=\{H\in \tbinom{[n]}{k}: 1\in H, H\cap [2,i+1]\ne \emptyset\}\cup \{H\in \tbinom{[2,n]}{k}: [2, i+1]\subset H\}.
\end{equation}

The family $\m H_{k,i}$ is intersecting, and 
\begin{equation}\label{sizeH}
    |\m H_{k,i}|={n-1\choose k-1}-{n-i-1\choose k-1}+{n-i-1\choose k-i}.
\end{equation}

For a family $\m F$ and an element $x$ of maximal degree in $\m F$, let us write $$\m F_\Delta = \m F(x),\ \ \ \ \ \m F_\gamma = \m F(\bar x).$$ 

Denote by $\mathcal{H}\subset {[n]\choose a}$ the (unique) inclusion-maximal intersecting family such that $\m H_{\Delta}=\m H(1)$ and $\mathcal{H}_{\gamma}=\mathcal{M}$.
Then if $a=b$, we have
$|\mathcal{H}(1)|=|\mathcal{B}'(1)|$ and 
\begin{equation}\label{12-6-1}
|\mathcal{H}|=|\mathcal{B}'(1)|+|\mathcal{M}|.
\end{equation} 
Note that $\m H_{a,3}$ is an intersecting family with $\Delta(\m H_{a,3})=|\m H_{a,3}(1)|$ and $\m M \subset \m H_{a,3}(\bar{1})$. 
By Theorem~\ref{11-23-2}, we have 
\begin{equation}\label{5-7-1} 
    |\mathcal{H}_{a,3}| \le |\mathcal{H}|
\end{equation}
The following proposition strengthens this bound needed in  the proof of Lemma~\ref{11-23-3}.

\begin{proposition}
  If $n>2a$, $m=3$ and $a\ge 5$, then 
  \begin{equation}\label{compare}
      |\m H|-|\m H_{a,3}|\ge a-3.
  \end{equation}
\end{proposition}

\begin{proof}
Since $m=3$ and $a\ge 5$, we have $k\ge 2$, $[2,4]=\cap \, \m M$ and $M_i \cap \{5,6\}\ne \emptyset$ for each $i\in [k]$. 
Since $\m H[\{1,5,6\},[6]] \cap \m H_{a,3}[[2,4],[6]]=\emptyset$, we may define 
\[
\m H'_{a,3}:=\m H_{a,3}\setminus \m H_{a,3}[[2,4],[6]] \cup \m H[\{1,5,6\},[6]].
\]
One checks that $\m H'_{a,3}$ is an intersecting family and $\m M \subset \m H'_{a,3}(\bar{1})$. 
Moreover, 
$$\m H(\{1,5,6\},[6])={[7,n]\choose a-3}=\m H_{a,3}([2,4],[6]),$$
hence,
\begin{equation}\label{H'a3}
|\m H'_{a,3}|=|\m H_{a,3}|. 
\end{equation}

We claim that we may assume 
\begin{equation}\label{H7}
    \m H(\{1,5,7\},[7])={[8,n]\choose a-3}.
\end{equation}
     
Indeed, if $k\ge 3$, then since $\{p_1, \dots, p_k\}=[5, 4+k]$ and in particular $\{5,7\}$ is a cover of $\m M([2,4])$, therefore every $a$-set $H\subset [n]$ with $\{1,5,7\}\subset H$ belongs to $\m H$, which gives (\ref{H7}). 
If $k=2$, then $5\in M_2$, $6\in M_1$ and $5\not\in M_1$.
since $a\ge 5$, we have $|M_1\setminus [2,4]|\ge 2$.  Thus, $M_1$ contains some element $q\not\in \{5,6\}=\{p_1,p_2\}$; relabeling $q$ as 7, we get $\{5,7\}$ covering $\m M([2,4])$.
This proves (\ref{H7}). 

Now every set of $\m H'_{a,3}(\bar{1})$ contains $[2,4]$ and intersects $\{5,6\}$. 
Define 
\[
\m H''_{a,3}:=\m H'_{a,3}\setminus \m H'_{a,3}[\{2,3,4,6\},[7] ]\cup \m H[\{1,5,7\},[7]].
\]
Then $\m H''_{a,3}$ is intersecting and $\m M \subset \m H''_{a,3}(\bar{1})$. By Theorem \ref{11-23-2}, we have
\begin{equation}\label{H''a3}
|\m H''_{a,3}|\le |\m H|.
\end{equation}

Note that $\m H'_{a,3}(\{2,3,4,6\},[7] )={[8,n]\choose a-4}$. 
Applying (\ref{H7}), we obtain

\begin{equation}\label{H-H}
    |\m H''_{a,3}|-|\m H'_{a,3}|={n-7\choose a-3}-{n-7\choose a-4}\overset{(\ref{n>2k})}{\ge} a-3.
\end{equation}

Combining (\ref{H''a3}), (\ref{H'a3}), and (\ref{H-H}) yields ( \ref{compare}).
\end{proof}

Now we are ready to give the proof of Lemma~\ref{11-23-3}.
\begin{proof}[{\bf Proof of Lemma~\ref{11-23-3}}]  

First we prove inequality \eqref{11-30-2} and the description of the exceptional cases.

Recall that $|\cap \, \m M|=m\ge a-b+3$. 
If $m\ge b$, then since $n>2b$ and $b\ge 4$, we obtain
\[
|\mathcal{B}'|\ge {n-1\choose b-1}-{n-1-b\choose b-1}+1>{n-1\choose b-1}-{n-4\choose b-1}+{n-4\choose b-3},
\]
which gives (\ref{11-30-2}),
so we may assume $m \le b-1$ in what follows.

{\bf Case 1: $m=3$.} In this case, we have $a=b$. 

If $|\mathcal{M}|\le |\mathcal{T}|$, then by \eqref{12-6-1} and \eqref{5-7-1}, we have
\[
|\mathcal{B}'| = |\mathcal{B}'(1)| + |\mathcal{T}| 
               \ge |\mathcal{B}'(1)| + |\mathcal{M}|
               = |\mathcal{H}|
               \ge |\mathcal{H}_{b,3}|
               \ge \tbinom{n-1}{b-1} - \tbinom{n-4}{b-1} + \tbinom{n-4}{b-3},
\]
which is exactly \eqref{11-30-2}. 

We may therefore assume $|\mathcal{M}| > |\mathcal{T}|$.

Consider first the case $a=b=4$. Since $\m M$ and $\m T$ are intersection-minimal,
Claim~\ref{minsize} gives $|\m M|, |\m T|\le a-m+1=2$; combined with $|\mathcal{M}|> |\mathcal{T}|$, we get $|\mathcal{M}|=2$ and $|\mathcal{T}|=1$. 
Consequently, 
\[
|\mathcal{A}'| > \tbinom{n-1}{a-1} - \tbinom{n-4}{a-1} + \tbinom{n-4}{a-3}
\quad\text{and}\quad
|\mathcal{B}'| = \tbinom{n-1}{b-1} - \tbinom{n-4}{b-1} + \tbinom{n-4}{b-3} - 1.
\]
This establishes the exceptional case in the lemma.

Now let $a=b\ge 5$. By Claim~\ref{minsize}, we have $|\m M|\le a-m+1=b-2$, and so by (\ref{12-6-1}), we obtain
\begin{align*}
|\mathcal{B}'| 
   &= |\mathcal{H}(1)| + |\mathcal{T}|
     = |\mathcal{H}| - |\mathcal{M}| + |\mathcal{T}|\ge |\mathcal{H}| - |\mathcal{M}| + 1
     \ge |\mathcal{H}| - (b-3) \\
   &\overset{(\ref{compare})}{\ge} |\mathcal{H}_{b,3}|\ge \tbinom{n-1}{b-1} - \tbinom{n-4}{b-1} + \tbinom{n-4}{b-3}.
\end{align*}
Thus inequality \eqref{11-30-2} is established in this case as well.

{\bf Case 2: $m\ge 4$.}
Recall that we also have $m \le b-1$. Also, recall that we assumed that $\m M$ has the following specific form: 
$\mathcal{M}=\{M_1, \dots, M_{k}\}$, and $[2,m+1+k]\setminus M_i =\{m+1+i\}$.

Define a $b$-uniform family $\mathcal{M}'$ such that 
\begin{itemize}
    \item $\mathcal{M}'$ is intersection‑minimal with $\cap \,\mathcal{M}' = [2, m+1]$;
    \item If $b\le m+k-1$, then let $\m M'=\{[2,m+1]\cup R: R\in {[m+2,b+2]\choose b-m}\}$;
    \item If $b\ge m+k$, then just delete any $(a-b)$ elements outside $[2,m+1+k]$ from each $M\in \m M$.
\end{itemize}
It is easy to see that
\begin{equation}\label{coverm'}
\text{If $I$ is a cover of $\m M'$, then $I$ is a cover of $\m M$ as well.} 
\end{equation}

Since $a\ge b$, define  by replacing each $M\in M$ with a $b$-set containing $[2, m+1]$, chosen so that 
Since $4\le m\le b-1$ and $\Delta(\m H_{b,4})=|\m H_{b,4}(1)|$, we have $\mathcal{M}' \subset \mathcal{H}_{b,4}(\bar{1})$. 

Let $\mathcal{H}'\subset {[n]\choose b}$ be the (unique) inclusion-maximal intersecting family for which $\mathcal{H}'_{\Delta} = \mathcal{H}(1)$ and $\mathcal{H}'_{\gamma} = \mathcal{M}'$.
By Theorem~\ref{11-23-2} we have $|\mathcal H
_{b,4}|\le |\mathcal{H}'|$. 
Consequently,  
\begin{align*}
|\mathcal{H}'| &\ge |\mathcal{H}_{b,4}|
   = \binom{n-1}{b-1} - \binom{n-4}{b-1} + \binom{n-5}{b-2} + \binom{n-5}{b-4} \\
   &\stackrel{\eqref{n>2k}}{\ge} 
      \binom{n-1}{b-1} - \binom{n-4}{b-1} 
      + \binom{n-5}{b-3} + \binom{n-5}{b-4} + b-2 \\
   &= \binom{n-1}{b-1} - \binom{n-4}{b-1} + \binom{n-4}{b-3} + b-2.
\end{align*}

 Claim~\ref{minsize} gives $|\mathcal{M}'|\le b-m+1< b-2$.
By (\ref{coverm'}), $|\mathcal{B}'(1)|\ge|\mathcal{H}'(1)|$.
Therefore
\begin{align*}
|\mathcal{B}'|
   &= |\mathcal{B}'(1)| + |\mathcal{T}| \\
   &\ge |\mathcal{H}'(1)| + |\mathcal{M}'| - |\mathcal{M}'| + |\mathcal{T}| \\
   &= |\mathcal{H}'| - \bigl(|\mathcal{M}'| - |\mathcal{T}|\bigr) \\
   &\ge \binom{n-1}{b-1} - \binom{n-4}{b-1} + \binom{n-4}{b-3} + b-2
        - \bigl(|\mathcal{M}'| - |\mathcal{T}|\bigr) \\
   &> \binom{n-1}{b-1} - \binom{n-4}{b-1} + \binom{n-4}{b-3},
\end{align*}
where the last inequality holds since $|\mathcal{M}'| < b-2$. 
Thus \eqref{11-30-2} holds also in this case.

Next, we turn to \eqref{11-30-1}. When $a=b$, this follows from \eqref{11-30-2} by symmetry, so we only need to consider the case $a>b$. 

If $t\ge 4$, then 
\begin{align*}
|\mathcal{A}'|
&>{n-1\choose a-1}-{n-5\choose a-1}\\
&\ge {n-1\choose a-1}-{n-4\choose a-1}+{n-5\choose b-3}\\
&\ge {n-1\choose a-1}-{n-4\choose a-1}+{n-1-(a-b+3)\choose a-(a-b+3)},
\end{align*}
where the second inequality uses $n\ge a+b$  and $a>b$.
This  gives \eqref{11-30-1}.

Recall that, by our assumption, $\m T = \{T_1,\ldots, T_\ell\}$ and $[2,t+\ell+1]\setminus T_i=\{t+i+1\}$. Now suppose $t=3$. 
Since $b\ge 4$, we have $\ell\ge 2.$
By maximality, every $a$-set containing $\{1,5,6\}$ belongs to $\mathcal{A}'$,
hence
\begin{equation}\label{11-30-4}
|\mathcal{A}'(1)|\ge {n-1\choose a-1}-{n-4\choose a-1}+{n-6\choose a-3}.
\end{equation}

If $a-b\ge 2$, then 
\[
{n-1-(a-b+3)\choose a-(a-b+3)}\le {n-6\choose b-3}< {n-6\choose a-3},
\]
where the last inequality uses $n>a+b$. Thus, \eqref{11-30-1} follows from \eqref{11-30-4}.

Finally, assume $a=b+1$, so that $n\ge 2a$. 
If $b=4$, then the inequality
${n-6\choose a-3}\ge {n-5\choose b-3}$ together with \eqref{11-30-4} yields \eqref{11-30-1}. 

Write $\m T=\{T_1, T_2, \dots, T_{\ell}\}$. Since $\mathcal{T}$ is intersection-minimal, for each $i\in [\ell]$ there exists $q_i\in \cap\,(\mathcal{T}\setminus \{T_i\})\setminus (\cap \,\mathcal{T})$;
w.l.o.g. $\{q_1, \dots, q_{\ell}\}=[t+2, t+\ell+1]$. 
For $b\ge 5$ and $t=3$ we distinguish three cases.
\begin{itemize}
\item[($\ell=1$).] 
Then $|\mathcal{A}'| > |\mathcal{A}'(1)| 
   \ge \binom{n-1}{a-1} - \binom{n-6}{a-1}
   > \binom{n-1}{a-1} - \binom{n-4}{a-1} + \binom{n-5}{a-4}$ (since $n\ge 2a$),
   which implies \eqref{11-30-1}.

\item[($\ell=2$).] 
We may assume $\{2,3,4,5,8\}\subset T_1$ and $\{2,3,4,6,7\}\subset T_2$.
Consequently, every $a$-set containing at least one of the triples
$\{1,5,6\},\{1,5,7\},\{1,6,8\},\{1,7,8\}$ belongs to $\mathcal{A}'$. 
Hence
$$|\mathcal{A}'| 
   > |\mathcal{A}'(1)| \ge \tbinom{n-1}{a-1} - \tbinom{n-4}{a-1} 
        + \tbinom{n-6}{a-3} + \tbinom{n-7}{a-3} + 2\tbinom{n-8}{a-3} 
   \ge \tbinom{n-1}{a-1} - \tbinom{n-4}{a-1} + \tbinom{n-5}{b-3},
$$
which gives \eqref{11-30-1}.

\item[($\ell\ge 3$).] 
Now every $a$-set containing at least one of the triples
$\{1,5,6\},\{1,5,7\},\{1,6,7\}$ lies in $\mathcal{A}'$. 
Therefore
$$
|\mathcal{A}'|
   > |\mathcal{A}'(1)|
   \ge \tbinom{n-1}{a-1} - \tbinom{n-4}{a-1} 
        + \tbinom{n-6}{a-3} + 2\tbinom{n-7}{a-3}
   \ge \tbinom{n-1}{a-1} - \tbinom{n-4}{a-1} + \tbinom{n-5}{b-3},
$$
and \eqref{11-30-1} follows.
\end{itemize}
This completes the proof of Lemma~\ref{11-23-3}.
\end{proof}

\section{Small diversity: Proof of Lemmas~\ref{lemkey} and \ref{10-25-2}}\label{secsmall}

We first show the proof of Lemma \ref{10-25-2} by assuming Lemma \ref{lemkey}, and later we give the detailed  proof of Lemma \ref{lemkey}.

\begin{proof}[{\bf Proof of Lemma \ref{10-25-2}}]
Let $\mathcal{A}, \mathcal{B}, \mathcal{M}, \mathcal{T}, \mathcal{A}', \mathcal{B}'$, $m$ and $t$ be as in the statement of Lemma~\ref{10-25-2}.
Since $\m A$ and $\m B$ share the same element of maximum degree, we may assume
w.l.o.g. $\m A_\Delta = \m A(1)$ and $\m B_\Delta = \m B(1)$, so that $|\m A(\bar1)|=\gamma(\m A)$ and $|\m B(\bar1)|=\gamma(\m B)$.
Applying Lemma~\ref{lemkey} (ii) with
\[
\m F=\m A(\bar 1), \m G= \m B(1), \m G'=\m B'(1), \m Y=\m M, X=[2,n], k=a, r=b-1, 
\]
we verify the hypotheses of Lemma~\ref{lemkey}(ii): $|X|>k+r$; $|\cap \m Y|=m\ge a-b+3= k-r+2$; $|\m F|=|\m A(\bar1)|=\gamma(\m A)\le{n-1-(a-b+2)\choose a-(a-b+2)}= {n-1-(a-b+2)\choose b-2}={|X|-(k-r+1)\choose r-1}$. 
Hence Lemma~\ref{lemkey}(ii) yields
$$|\m A(\bar 1)|+|\m B(1)|\le |\m A'(\bar 1)|+|\m B'(1)|.$$

Applying Lemma~\ref{lemkey} with
\[
\m F=\m B(\bar 1), \m G= \m A(1), \m G'=\m A'(1), \m Y=\m T, X=[2,n], k=b, r=a-1, 
\]
we see that $k=b\le a-1+1= r+1$. 
If $k<r$, then Lemma~\ref{lemkey} (i) applies directly. Otherwise $k\in \{r,r+1\}$, in which case $|\cap\m Y|=t\ge 3 \ge k-r+2$ and $|\m F|=\gamma(\m B)\le {n-3\choose b-2}\le {|X|-(k-r+1)\choose r-1}$. In either case, Lemma~\ref{lemkey} yields  
$$|\m B(\bar 1)|+|\m A(1)|\le |\m B'(\bar 1)|+|\m A'(1)|.$$

Adding these two inequalities gives $|\m A|+|\m B|\le |\m A'|+|\m B'|$.
\end{proof}

\begin{proof}[{\bf Proof of Lemma~\ref{lemkey}}]
Choose $\m F, \m G$ satisfying the hypothesis of the lemma with $|\m F|+|\m G|$ being the largest. Moreover, among those, we choose a pair with the smallest  $|\m F|$. Put $s:=|\cap \m Y|$.
We first make some simplifying assumptions:
\begin{itemize}
\item Assume $X = [n]$ and $\cap \m Y=[s]$. 
\item Assume that $\mathcal{F},\mathcal{G}$ are $S_{i,j}$-shifted for all $i\in[s]$, $j>i$. 
\end{itemize}

Let us explain why we can make such simplification. 
The first is straightforward.
For the second, since $[s]=\cap \m Y$, shifting $S_{i,j}$ with $i\in [s]$, $j>i$ leaves $\m Y$ unchanged while preserving the cross-intersection property.

Our goal is to show that $\m F=\m Y$, from which $|\m G|\le |\m G'|$ follows.  

For each $i\in [n]$ define 
\[
    \m F_i:=\m F([i-1],[i]) \quad \text{and} \quad \m G^i:=\m G(\{i\},[i]).
\]

\begin{claim}\label{1-8-1}
For each $i\in [n]$ we have $ |\mathcal{F}_i|\le {n-1-(k-r+1)\choose r-1}$.
\end{claim}

\begin{proof}
Note that $\m F_i\subset {[i+1,n]\choose k-i+1}$.
If $i\ge k-r+2$, then 
$|\m F_i|\le {n-i\choose k-i+1}\le {n-(k-r+2)\choose k-(k-r+1)}={n-1-(k-r+1)\choose r-1}$. 
In particular, if $k<r$, then Claim \ref{1-8-1} holds for all $i$. So it remains to consider the case $k\ge r$. In this case, $s\ge k-r+2>0$.

Take any $i\in [k-r+1]$ and suppose, on the contrary, that 
\begin{equation}\label{fi>}
    |\mathcal{F}_i|> {n-1-(k-r+1)\choose r-1}.
\end{equation}

Since $k-r+1\le s$, the family $\m F$ is $S_{i,j}$-shifted for $i\le k-r+1$ and $j>i$. This implies that every $A\in \partial \mathcal{F}_i$ belongs to $\m F([i])$.

Assume that 
\begin{equation}\label{eqshadow}
|\partial \mathcal{F}_i |\ge {n-1-(k-r+1)\choose r-2}.
\end{equation} 
Then combining (\ref{fi>}) and (\ref{eqshadow}) yields
\[
|\mathcal{F}|\ge |\mathcal{F}_i|+|\mathcal{F}([i])|
\ge|\mathcal{F}_i|+|\partial\mathcal{F}_i|
>{n-(k-r+1)\choose r-1},
\]
a contradiction.

It remains to prove \eqref{eqshadow}. Choose $\mathcal{X}\subset \mathcal{F}_i$ and a real number $x$ such that
\[
|\mathcal{X}|= {n-1-(k-r+1)\choose r-1}={x\choose k-i+1}.
\]
We first claim that $x< n-i-1$. Indeed, since $i< k-r+2$, we have 
\[
{n-i-1\choose k-i+1} > {n-(k-r+2)-1\choose k-(k-r+2)+1}
={n-1-(k-r+1)\choose r-1}.
\]

By the Kruskal--Katona Theorem~\ref{12-2-2} we have   
\[
|\partial \mathcal{F}_i | \ge |\partial \mathcal{X}|\ge {x\choose k-i}.
\]
Furthermore,
\begin{align*}
{x\choose k-i}
&=\frac{k-i+1}{x+i-k}{x\choose k-i+1}
=\frac{k-i+1}{x+i-k}{n-1-(k-r+1)\choose r-1}\\
&=\frac{k-i+1}{x+i-k}\cdot \frac{n-k}{r-1}{n-1-(k-r+1)\choose r-2}\\
&>{n-1-(k-r+1)\choose r-2},
\end{align*}
where the last inequality holds because $i<k-r+2$ implies $k-i+1\ge r$, and $x< n-i-1$ implies $x+i-k<n-k-1<n-k$. This proves \eqref{eqshadow}.
\end{proof}

\begin{claim}\label{1-8-2}
We have $[s]\subset A$ for every $A\in \m F$. 
\end{claim}

\begin{proof}
Note that $s>0$. 
We will show that for each $i\in [s]$, we have $ \mathcal{F}_i= \emptyset$. This will imply the claim. 

Arguing indirectly, choose $i\in [s]$ that is the smallest integer such that $ \mathcal{F}_i\ne \emptyset$.

Note that $\mathcal{F}_i\subset {[i+1,n]\choose k-i+1}$ and $\mathcal{G}^i\subset {[i+1,n]\choose r-1}$ are cross-intersecting. We claim that 
\begin{equation}
    \label{eqcross1} |\m F_i|+|\mathcal{G}^i|\le {n-i\choose r-1}.
\end{equation}

Indeed, if $r-1\ge k-i+1$, then \eqref{eqcross1} is guaranteed by Theorem \ref{12-3-3}. Otherwise $k\ge r+i-1\ge r$ and Claim~\ref{1-8-1} gives $|\m F_i|\le {n-1-(k-r+1)\choose r-1}$, so Theorem \ref{10-12-1} (with $\m F_i, \m G^i, n-i, k-i+1, r-1$ playing the roles of $\m A, \m B,n,a,b$, respectively) yields \eqref{eqcross1}.

Recall the notation $\m F\big[[i-1], [i]\big]=\{F\in \m F: [i-1]\subset F, i\not\in F\}$.
Define the new families 
\[
\tilde{\m F}:=\m F\setminus \m F\big[[i-1], [i]\big], \quad 
\tilde{\m G}:=\m G\cup \big\{G\in {[n]\choose r}: G\cap [i]=\{i\}\big\}.
\]
The families $\tilde{\m F}, \tilde{\m G}$ are also cross-intersecting since all the remaining sets in $\m F$ contain $[i]$ and thus intersect the sets newly added to $\m G$. Moreover, \eqref{eqcross1} implies that $|\tilde{\m F}|+|\tilde{\m G}|\ge |\m F|+|\m G|$. Finally, $|\tilde{\m F}|<|\m F|$. This contradicts the choice of $\m F, \m G$, proving the claim.
\end{proof}

Put $\mathcal{Y}=\{Y_1, \dots, Y_{\ell}\}$. Since $\mathcal{Y}$ is intersection-minimal, for each $Y_j\in \mathcal{Y}$, $j\in [\ell]$, there exists
\[
v_j\in \Big(\cap (\mathcal{Y}\setminus \{Y_j\})\Big)\setminus \Big(\cap\mathcal{Y}\Big).
\]
We may assume that $\{v_1, \dots, v_\ell\}=[s+1, s+\ell]$. In this notation, for each $j\in [\ell]$, we have $[s+\ell]\setminus \{s+j\}\subset Y_j$, and  $$Y_j\setminus [s+j-1] \in \m F_{s+j}.$$
So $\m F_{s+j}\ne \emptyset$ for all $j\in [\ell]$. 
Next we show that $Y_j\setminus [s+j-1]$ is the unique member of $\m F_{s+j}$. 

\begin{claim}\label{12-14-1}
For every $i\in [\ell]$, we have $|\m F_{s+i}|=1$.  
\end{claim}

\begin{proof}
Suppose for contradiction that $|\mathcal{F}_{s+i}|\ge 2$ for some $i\in [\ell]$, and take the smallest such $i$. Then for every $j\in [i-1]$, we have $\m F_{s+j}=\{Y_j\setminus [s+j-1]\}$ and $s+i\in Y_j\setminus [s+j-1]$ since $[s+\ell]\setminus \{s+j\}\subset Y_j$ and $j<i$. Consequently, we have
\begin{equation}\label{contains+i}
    \text{Every set in $\m F\setminus \m F\big[[s+i-1],[s+i]\big]$ contains $s+i$.}
\end{equation}

We first show that Claim \ref{12-14-1} holds provided that  
\begin{equation}\label{1-15-2}
|\mathcal{F}_{s+i}|+|\mathcal{G}^{s+i}|\le {n-s-i\choose r-1}-{n-k-1\choose r-1}+1.
\end{equation}
To this end, define 
\begin{align*}
    \hat{\m G}&:=\{G\in \tbinom{[n]}{r}: G\cap [s+i]=\{s+i\}, G\cap Y_{i}\ne \emptyset\},\\
    \tilde{\m F}&:=(\m F\setminus \m F\big[[s+i-1], [s+i]\big])\cup \{Y_{i}\},\\
    \tilde{\m G}&:=\m G\cup \hat{\m G}.
\end{align*}
Note that $[s+i-1]\subset Y_i$ and $s+i\not\in Y_i$.
We have $\m Y \cap \m F\big[[s+i-1], [s+i]\big]=\{Y_i\}$, hence $\m Y \subset \tilde{\m F}$.
By (\ref{contains+i}), every set in $\m F\setminus \m F\big[[s+i-1],[s+i]\big]$ intersects every set in $\hat{\m G}$. 
Thus $\tilde{\m F}, \tilde{\m G}$ are cross-intersecting and $\m Y \subset \tilde{\m F}$,
so they satisfy the hypotheses of  Lemma \ref{lemkey}. 
Since $|\m F|+|\m G|$ is maximum, we have 
\begin{equation}\label{sumlower}
    |\tilde{\m F}|+|\tilde{\m G}|\le |\m F|+|\m G|.
\end{equation}
On the other hand, since $[s+i-1]\subset Y_i$ and $s+i\not\in Y_i$,
\begin{equation}\label{hatg}
    |\hat{\m G}|={n-s-i\choose r-1}-{n-k-1\choose r-1}.
\end{equation}
Note that $\{G\cup \{s+i\}: G\in \m G^{s+i}\}= \hat{\m G} \cap \m G$.
We have 
$$|\tilde{\m G}|=|\m G|+|\hat{\m G}|-|\m G^{s+i}|.$$
As $|\m F\big[[s+i-1], [s+i]\big]|=|\m F_{s+i}|$ and $Y_i\in \m F\big[[s+i-1], [s+i]\big]$, we have 
$$|\tilde{\m F}|=|\m F|-|\m F_{s+i}|+1.$$
 Combining (\ref{1-15-2}) and (\ref{hatg}), we obtain 
\[
|\tilde{\m F}|+|\tilde{\m G}|=|\m F|-|\m F_{s+i}|+1+|\m G|+|\hat{\m G}|-|\m G^{s+i}|\ge |\m F|+|\m G|.
\]
By (\ref{sumlower}), $|\tilde{\m F}|+|\tilde{\m G}|=|\m F|+|\m G|$.
However, $|\tilde{\m F}|<|\m F|$ and $\m Y \subset \tilde{\m F}$, which contradicts the minimal choice of $\m F$. 

In the remainder of the proof, we verify (\ref{1-15-2}). 
Note that $\m F_{s+i}\subset\tbinom{[s+i+1,n]}{k-(s+i-1)}$ and $\m G^{s+i}\subset \tbinom{[s+i+1,n]}{r-1}$ are cross-intersecting. 

Suppose $k-(s+i-1)<r-1$. 
If $\m G^{s+i}=\emptyset$, then since $n>k+r$ and $r\ge 3$, we obtain
\[
|\mathcal{F}_{s+i}|+|\mathcal{G}^{s+i}|\le {n-(s+i)\choose k-(s+i-1)}< {n-s-i\choose r-1}-{n-k-1\choose r-1}+1,
\]
where the last inequality uses ${n-(s+i)\choose k-(s+i-1)}=\sum_{j=1}^{k-(s+i-2)}{n-(s+i)-j\choose k-(s+i-2)-j}$ and ${n-s-i\choose r-1}-{n-k-1\choose r-1}=\sum_{j=1}^{k-(s+i-1)}{n-(s+i)-j\choose r-2}$.
Next, assume that $\m G^{s+i}\ne \emptyset$. Since $\m F_{s+i}\ne \emptyset$, the two families are non-empty cross-intersecting families. By Theorem \ref{12-3-2} (with $r-1,k-(s+i-1),n-(s+i)$ in the roles of $a,b,n$, respectively), (\ref{1-15-2}) holds. 

Next, assume that $k-(s+i-1)\ge r-1$. Then $s\le k-r+2-i$. 
If $k\ge r$, then by the hypotheses of Lemma \ref{lemkey}, 
$s\ge k-r+2$, a contradiction. Therefore,  $k<r$, so $s+i\le k-r+2\le 1$. Notice that $s\ge 1$ and $i\ge 1$.  
A contradiction. This proves (\ref{1-15-2}) and Claim \ref{12-14-1}.
\end{proof}


By Claim \ref{12-14-1} and the observation that $Y_i\setminus [s+i-1]\in \m F_{s+i}$, we have $\m F_{s+i}=\{Y_i\setminus [s+i-1]\}$ for each $i\in [\ell]$. 
Combined with Claim~\ref{1-8-2}, every $A\in \m F$ contains $[s]$, and every $A\in \m F\setminus \m Y$ contains $[s+\ell]$ (since $A$ does not contribute to any $\m F_{s+i}$). 
We now show  $\m F=\m Y$.  
For each $Y_i$, $i\in [\ell]$, select one element $w_i\in Y_i\setminus [s+\ell]$. (Note that these $w_i$ may coincide.) 
Put
$I=\{w_1, \dots, w_\ell\}$ and denote
\begin{align*}
\mathcal{F}_I&:=\{ A\setminus [s+\ell]: A\in   \mathcal{F}, A\cap I=\emptyset \},\\
\mathcal{G}^I&:=\{B\setminus I: B\in \mathcal{G}, I\subset B, B\cap [s+\ell] =\emptyset\}.
\end{align*}
The families $\mathcal{F}_I \subset {[s+\ell+1,n]\setminus I \choose k-s-\ell}$ and 
$\mathcal{G}^I\subset {[s+\ell+1,n]\setminus I \choose r-|I|}$ are cross-intersecting. 
Moreover, we have $k-s-\ell< r-|I|$. Indeed, if $s\ge k-r+2$, then $k-s-\ell\le k-s-|I|\le k-(k-r+2)-|I|<r-|I|$; if $s\le k-r+1$, then $k<r$ (by the condition of Lemma \ref{lemkey}), so $s=0$, and hence $k-s-\ell=k-|I|<r-|I|$ again.  Thus, 
\[
|\mathcal{F}_I|+|\mathcal{G}^I|\le {n-s-\ell-|I|\choose r-|I|}
\]
by Theorem~\ref{12-3-3}. We claim that $\mathcal{F}_I=\emptyset$. 

Indeed, suppose that $\mathcal{F}_I\neq\emptyset$. Define
\[
\tilde{\m F}:=\m F\setminus \{A\in \m F: A\cap I=\emptyset\}, 
\qquad 
\tilde{\m G}:=\m G\cup \{B\in \tbinom{[n]}{r}: I\subset B,\ B\cap [s+\ell]=\emptyset\}.
\]
Then $\tilde{\m F},\tilde{\m G}$ are cross-intersecting, and
\[
|\tilde{\m F}|+|\tilde{\m G}|\ge |\m F|+|\m G|,
\]
while $|\tilde{\m F}|<|\m F|$, contradicting the choice of $\m F,\m G$. Thus $\mathcal{F}_I=\emptyset$.

It remains to show that for any set $A\in \m F\setminus \m Y$ we can find $I$ as above so that $A\cap I = \emptyset$. Indeed, since 
\[
|Y_i\setminus [s+\ell]| = k-s-(\ell-1) > k-s-\ell = |A\setminus [s+\ell]|,
\]
the set $Y_i\setminus [s+\ell]$ is not contained in $A$, so
we may pick $f_i\in Y_i\setminus (A\cup [s+\ell])$. Set $I = \{f_1,\ldots, f_\ell\}$; then $A\cap I = \emptyset.$ 
Since the above conclusion holds for any valid choice of $I$, we must conclude that $\m F\setminus \m Y$ is empty. Thus $\m F = \m Y$, and by the definition of $\m G'$, $|\m G| \le |\m G'|$. Hence $|\m F|+|\m G|\le |\m Y|+|\m G'|$, completing the proof.
\end{proof}

\section{Large diversity: Proof of Lemma \ref{lemmain}}\label{seclarge}
We begin with the following result. 
\begin{proposition}\label{1-4-3}
Let $n>a+b$, $\m A \subset {[n]\choose a}$, $\m B \subset {[n]\choose b}$ be cross-intersecting families. If $|\m A|\ge |\m A'|$ and 
 $\gamma(\mathcal{B})> {n-3\choose b-2}$, then $\gamma(\mathcal{A})> {n-1-(a-b+2)\choose a-(a-b+2)}$. 
\end{proposition}
\begin{proof}
 Suppose for contradiction that $\gamma(\mathcal{A})\le {n-1-(a-b+2) \choose a-(a-b+2)}$. 
 W.l.o.g., we may assume $\mathcal{A}_{\Delta}=\mathcal{A}(1)$. 
 
 First, consider the case $a>b$. 
 By Lemma \ref{11-23-3},
 we have $|\m A|\ge |\m A'|> {n-2\choose a-2}+{n-3\choose a-2}+{n-4\choose a-2}$, combining this with  
  $$\gamma(\mathcal{A})\le {n-1-(a-b+2) \choose a-(a-b+2)}\le {n-4 \choose b-2}\le {n-4 \choose a-2},$$ where the last inequality holds by $a>b$ and $n>a+b$, we  have 
$|\mathcal{A}(1)|\ge{n-2\choose a-2}+{n-3\choose a-2}$.
Note that $\m A(1), \m B(\bar1)$ are cross-intersecting.
Applying Proposition \ref{1-4-1} to $\m A(1), \m B(\bar1)$ on $[2,n]$ with $i=2$ gives $\gamma(\mathcal{B})\le |\mathcal{B}(\bar{1})|\le{n-3\choose b-2}$, a contradiction.

Next, assume $a=b$. In this case, our assumption is $\gamma(\mathcal{A})=|\m A(\bar{1})|\le {n-3 \choose a-2}$. 
By Lemma \ref{11-23-3}, we have $|\m A|\ge {n-2\choose a-2}+2{n-3\choose a-2}-1$ (the $-1$ from the exceptional case in Lemma \ref{11-23-3}), therefore, 
\begin{equation}\label{a1upper}
    |\mathcal{A}(1)|\ge{n-2\choose a-2}+{n-3\choose a-2}-1.
\end{equation} 
To end the proof of Proposition \ref{1-4-3}, it suffices to
show that $\gamma(\mathcal{B})\le |\mathcal{B}(\bar{1})|\le{n-3\choose b-2}$. 

Suppose for contradiction that $ |\mathcal{B}(\bar{1})|\ge{n-3\choose b-2}+1$.
Let 
$$\m F=\m L([2,n], |\m A(1)|, a-1), \,\,\m G=\m L([2,n],|\m B(\bar{1})|,b).$$ 
Since $\m B(\bar{1})\subset {[2,n]\choose b}$ and $\m A(1)\subset {[2,n]\choose a-1}$ are cross-intersecting,
by Theorem \ref{12-1-1}, $\m F, \, \m G$ are cross-intersecting. 
Since $|\m G|=|\m B(\bar1)|\ge {n-3\choose b-2}+1$,
all sets $B\subset [2,n]$ with $\{2,3\}\subset B$ belong to $\m G$, and $\{2,4, 5, \dots, b+2\}\in \m G$. 
Since $n>a+b=2a$, for every $A\in \m F$, we have  $A\cap \{2,3\}\ne \emptyset$ and $A\cap \{2,4,5,\dots,b+2\}\ne \emptyset$, which yields 
$|\mathcal{A}(1)|\le {n-2\choose a-2}+{n-3\choose a-2}-{n-b-2\choose a-2}<{n-2\choose a-2}+{n-3\choose a-2}-1$ (since ${n-b-2\choose a-2}>1$ when $a\ge 4$ and $n>2a$), contradicting \eqref{a1upper}. 
\end{proof}

We will use Proposition~\ref{1-4-3} in the proof of Lemma~\ref{lemmain} to reduce case analysis.
In the proof, we also need the following observation.
\begin{obser}\label{11-27-1}
Let $\mathcal{H}\subset {[n]\choose h}$ satisfy $\Delta(\mathcal{H})=|\mathcal{H}(1)|\ge {n-2\choose h-2}+\gamma(\mathcal{H})$. Define $\mathcal{H}'$ by 
$\mathcal{H}'(1)=\mathcal{H}(1)$ and $\mathcal{H}'(\bar{1})=\mathcal{L}([2,n], \gamma(\mathcal{H}), h)$.
Then $\Delta(\mathcal{H}')=|\mathcal{H}'(1)|$. 
Moreover, if we further replace $\mathcal{H}'(1)$ by $\mathcal{L}([2,n], |\mathcal{H}(1)|, h-1)$, then the new family still has maximum degree at element $1$. 
\end{obser}

To prove Lemma \ref{lemmain}, we need the following three propositions. Their  proofs are postponed to the next sections. 

\begin{proposition}\label{12-15-1}
In Assumption \ref{12-14-2}, if, additionally, $|\mathcal{A}|\ge |\mathcal{A}'|$, $|\mathcal{B}|\ge |\mathcal{B}'|$, $\gamma(\mathcal{A})\le {n-1-(a-b+2)\choose a-(a-b+2)}$ and $\gamma(\mathcal{B})\le {n-3\choose b-2}$, then 
$\Delta(\mathcal{A})\ge {n-2\choose a-2}+\gamma(\mathcal{A})$ and $\Delta(\mathcal{B})\ge {n-2\choose b-2}+\gamma(\mathcal{B})$, and $\mathcal{A}$ and $\mathcal{B}$ share the same element of maximum degree.
\end{proposition}

\begin{proposition}\label{11-29-1}
Let $n>a+b$, $\m A \subset {[n]\choose a}, \m B \subset {[n]\choose b}$ be cross-intersecting. If $|\mathcal{A}|\ge |\mathcal{A}'|$ and $|\mathcal{B}|\ge |\mathcal{B}'|$, then there exist cross-intersecting families $\mathcal{A}_a, \mathcal{B}_b$ such that the following hold:
\begin{enumerate}
\item $|\mathcal{A}_a|\ge |\mathcal{A}|$, $|\mathcal{B}_b|\ge |\mathcal{B}|$;
\item $\Delta(\mathcal{A}_a)=|\m A_a(1)|\ge {n-2\choose a-2}+\gamma(\mathcal{A}_a)$, $\Delta(\mathcal{B}_b)=|\m B_b(1)|\ge {n-2\choose b-2}+\gamma(\mathcal{B}_b)$;
\item $\gamma(\mathcal{A}_a) \ge \min \{{n-1-(a-b+3)\choose a-(a-b+3)}, \gamma(\mathcal{A})\}$,
 $\gamma(\mathcal{B}_b) \ge \min \{{n-4\choose b-3}, \gamma(\mathcal{B})\} $.
\end{enumerate}
\end{proposition}

\begin{proposition}\label{10-25-3}
Lemma \ref{lemmain} holds provided that $\Delta(\mathcal{A})=|\mathcal{A}(1)|\ge {n-2\choose a-2}+\gamma(\mathcal{A})$ and 
$\Delta(\mathcal{B})=|\mathcal{B}(1)|\ge {n-2\choose b-2}+\gamma(\mathcal{B})$.
\end{proposition}

Now we give the proof of Lemma \ref{lemmain}, assuming Propositions \ref{12-15-1}, \ref{11-29-1} and \ref{10-25-3}. 
For $\m F\subset 2^{[n]}$ and $i,j\in [n]$, we write
$\m F(i\bar j):=\{F\setminus \{i\}: i\in F, j\not\in F, F\in \m F\}$.

\begin{proof}[{\bf Proof of Lemma \ref{lemmain}}]
Suppose first that $\gamma(\mathcal{A})\le {n-1-(a-b+2)\choose a-(a-b+2)}$ and $\gamma(\mathcal{B})\le {n-3\choose b-2}$.
Then setting $\m A^a=\m A$ and $\m B^b=\m B$, all four conclusions of Lemma \ref{lemmain} follow directly from 
 Proposition \ref{12-15-1} and Assumption~\ref{12-14-2}. 
 
Otherwise, $\gamma(\mathcal{A})> {n-1-(a-b+2)\choose a-(a-b+2)}$ or $\gamma(\mathcal{B})> {n-3\choose b-2}$.  
By Proposition \ref{1-4-3}, the latter implies the former, so we may assume $\gamma(\mathcal{A})> {n-1-(a-b+2)\choose a-(a-b+2)}$ in the sequel.

Suppose that $\gamma(\mathcal{B})<{n-4\choose b-3}$. W.l.o.g. assume $\m B_{\Delta}=\m B(1)$.
By Lemma \ref{11-23-3}, $\Delta(\mathcal{B})=|\mathcal{B}(1)|\ge {n-1\choose b-1}-{n-4\choose b-1}$. 
 Since $\mathcal{B}(1)$ and $\mathcal{A}(\bar{1})$ are cross-intersecting, by Proposition \ref{1-4-1},
 $\gamma(\mathcal{A})\le |\mathcal{A}(\bar{1})|\le {n-4\choose a-3}$. Therefore, 
 $\Delta(\mathcal{A})\ge |\mathcal{A}(1)|>{n-2\choose a-2}+{n-3\choose a-2}+{n-4\choose a-2}-1$ (by Lemma \ref{11-23-3}). 
 We claim that $\m A_{\Delta}=\m A(1)$. 
 Indeed, for any $i\ne 1$, we have $|\m A(i)|\le |\m A(\{1,i\})|+|\m A(i\bar{1})|\le {n-2\choose a-2}+|\m A(\bar{1})|\le {n-2\choose a-2}+{n-4\choose a-3}<|\mathcal{A}(1)|$.
Since $\gamma(\mathcal{A})\le {n-4\choose a-3}$, we have $\Delta(\mathcal{A})= |\mathcal{A}(1)|>{n-2\choose a-2}+\gamma(\mathcal{A})$. Similarly, $\Delta(\mathcal{B})= |\mathcal{B}(1)|>{n-2\choose b-2}+\gamma(\mathcal{B})$. 
 By Proposition \ref{10-25-3}, we are done. 

Suppose next $\gamma(\mathcal{B})\ge{n-4\choose b-3}$.  
Recall that  $\gamma(\mathcal{A})> {n-1-(a-b+2)\choose a-(a-b+2)}$. 
By Proposition \ref{11-29-1}, there exist cross-intersecting families $\mathcal{A}_a, \mathcal{B}_b$ (note the subscripts, used here to distinguish from $\mathcal{A}^a, \mathcal{B}^b$ in Lemma~\ref{lemmain}) such that
\begin{itemize}
    \item $|\mathcal{A}_a|\ge |\mathcal{A}|$ and $|\mathcal{B}_b|\ge |\mathcal{B}|$;
    \item $\Delta(\mathcal{A}_a)=|\m A_a(1)|\ge {n-2\choose a-2}+\gamma(\mathcal{A}_a)$ and 
$\Delta(\mathcal{B}_b)=|\m B_b(1)|\ge {n-2\choose b-2}+\gamma(\mathcal{B}_b)$;
    \item $\gamma(\mathcal{A}_a)\ge {n-1-(a-b+3)\choose a-(a-b+3)}$ and $\gamma(\mathcal{B}_b)\ge {n-4\choose b-3}$,
\end{itemize}
where the third item uses $\gamma(\mathcal{A})> {n-1-(a-b+2)\choose a-(a-b+2)}\ge {n-1-(a-b+3)\choose a-(a-b+3)}$ and $\gamma(\mathcal{B})\ge{n-4\choose b-3}$. 

Define $\mathcal{A}'_a, \mathcal{B}'_b$ such that $\mathcal{A}'_a(1)=\mathcal{L}([2,n], |\mathcal{A}_a(1)|, a-1)$, $\mathcal{A}'_a(\bar{1})=\mathcal{L}([2,n], |\mathcal{A}_a(\bar{1})|, a)$, and similarly for   $\mathcal{B}'_b$. By Theorem \ref{12-1-1}, $\mathcal{A}'_a, \mathcal{B}'_b$ are cross-intersecting. 
By Observation \ref{11-27-1}, $\Delta(\mathcal{A}'_a)=|\mathcal{A}'_a(1)|$, $\Delta(\mathcal{B}'_b)=|\mathcal{B}'_b(1)|$. 
Moreover, as $\gamma(\mathcal{A}_a)\ge {n-1-(a-b+3)\choose a-(a-b+3)}$ and $\gamma(\mathcal{B}_b)\ge {n-4\choose b-3}$, we have $\gamma(\mathcal{A}'_a)\ge {n-1-(a-b+3)\choose a-(a-b+3)}$ and $\gamma(\mathcal{B}'_b)\ge {n-4\choose b-3}$.
Hence, $\mathcal{A}'_a(\bar1)$ contains all $a$-sets $A$ with $[2,a-b+4]\subset A \subset [2,n]$, and consequently it contains an isomorphic copy of $\m M$ (since $|\cap \m M|\ge a-b+3$). 
Similarly,
$\mathcal{B}'_b(\bar1)$ contains a copy of $\m T$.
Thus, $\m A'_a, \m B'_b$ satisfy the conditions of Proposition \ref{10-25-3}; taking $\m A^a=\m A'_a$ and $\m B^b=\m B'_b$ completes the proof. 
\end{proof}

\section{Proof of Proposition \ref{12-15-1}}\label{secpro1} 
In this section, we assume Assumption \ref{12-14-2} together with $|\mathcal{A}|\ge |\mathcal{A}'|$, $|\mathcal{B}|\ge |\mathcal{B}'|$, $\gamma(\m A)\le{n-1-(a-b+2)\choose a-(a-b+2)}$ and $\gamma(\m B)\le {n-3\choose b-2}$.


\begin{claim}\label{5-1-1}
    $\Delta(\mathcal{A})\ge {n-2\choose a-2}+\gamma(\mathcal{A}) \,\text{ and }\, \Delta(\mathcal{B})\ge {n-2\choose b-2}+\gamma(\mathcal{B}).$
\end{claim}

\begin{proof}[\bf{Proof of Claim \ref{5-1-1}}]
Consider first the case $a>b$ or $a=b>4$. 
If $\Delta(\m B)<{n-2\choose b-2}+\gamma(\m B)$, then 
$$|\m B|=\Delta(\m B)+\gamma(\m B)<{n-2\choose b-2}+2{n-3\choose b-2} = {n-1\choose b-1}-{n-4\choose b-1}+{n-4\choose b-3}\overset{\eqref{11-30-2}}{\le}|\m B'|,$$
contradicting $|\m B|\ge |\m B'|$. An analogous argument applied to $\Delta(\m A)$ contradicts \eqref{11-30-1}. 

Now consider the case $a=b=4.$ By the $\m A \leftrightarrow \m B$ symmetry when $a=b$, it suffices to prove $\Delta(\mathcal{B})\ge {n-2\choose b-2}+\gamma(\mathcal{B})$. 
Suppose for contradiction that $\Delta(\mathcal{B})< {n-2\choose b-2}+\gamma(\mathcal{B})$. By Lemma \ref{11-23-3}, we obtain
\begin{align*}
&{n-1\choose b-1}-{n-4\choose b-1}+{n-4\choose b-3}-1\le |\m B'|\le |\m B|=\Delta(\m B)+\gamma(\m B)\\
&<{n-2\choose b-2}+2\gamma(\m B)\le {n-2\choose b-2}+2{n-3\choose b-2}={n-1\choose b-1}-{n-4\choose b-1}+{n-4\choose b-3}.
\end{align*}
Hence $|\mathcal{B}|=\binom{n-1}{b-1}-\binom{n-4}{b-1}+\binom{n-4}{b-3}-1$, 
$\Delta(\mathcal{B})=\binom{n-2}{b-2}+\binom{n-3}{b-2}-1$, and 
$\gamma(\mathcal{B})=\binom{n-3}{b-2}$.
Assume w.l.o.g. $\m B_\Delta = \m B(1)$.

Let $\m B_1$ be the family obtained by replacing $\m B(1)$ with $\m L([2,n],|\m B(1)| ,b-1)$ and  $\m B(\bar 1)$  with  $\m L([2,n],|\m B(\bar 1),b)$. Then $\m B_1$ is exactly $\m H_{b,2}$ with one set removed. 
Let $\m A_1\subset \binom{[n]}{a}$ be the largest family cross-intersecting with $\m B_1$.
We claim $\m A_1=\m H_{a,2}$. Indeed, the maximum family cross-intersecting with $\m H_{b,2}$ is $\m H_{a,2}$, and since $n>a+b$, removing a single set from $\m H_{b,2}$ does not enlarge this maximum family.
Moreover, $|\m A_1|\ge |\m A|$. To see this, we replace $\m A(1)$ and $\m A(\bar 1)$ by $\m L([2,n],|\m A(1)| ,a-1)$ and $\m L([2,n],|\m A(\bar 1)| ,a)$, respectively, to obtain a family $\m A_2$. 
By Theorem~\ref{12-1-1}, the pair $(\m A_2, \m B_1)$ remains cross-intersecting. 
Since $\m A_1$ is maximum, $|\m A_1|\ge |\m A_2|=|\m A|$, as claimed.
However, $(\m A_1,\m B_1)$ is not inclusion-maximal: 
regardless of which set is missing from $\m B_1$, adding it back yields the cross-intersecting pair $(\m H_{a,2}$, $\m H_{b,2})$.
Moreover, $\m H_{a,2}$, $\m H_{b,2}$ contain isomorphic copies of $\m M, \m T$ in their diversity parts, respectively. 
This contradicts the maximality of $|\m A|+|\m B|$. 
\end{proof}

\begin{claim}\label{5-1-2}
Both families $\m A$ and $\m B$ share same element of maximum degree.
\end{claim}
\begin{proof}[\bf{Proof of Claim \ref{5-1-2}}]
We argue by contradiction. Without loss of generality, assume $\m A_\Delta = \m A(1)$ and $\m B_\Delta = \m B(2)$. 
We have
\begin{align*}
|\m B|&=|\m B(\{1,2\})|+|\m B(2 \bar 1)|+\gamma(\m B)\le {n-2\choose b-2}+|\m B(2 \bar 1)|+{n-3\choose b-2},\\
|\m A|&=|\m A(\{1,2\})|+|\m A(1 \bar 2)|+\gamma(\m A)\le {n-2\choose a-2}+|\m A(1 \bar 2)|+{n-1-(a-b+2)\choose a-(a-b+2)}.
\end{align*}
Comparing these upper bounds with the lower bounds given by \eqref{11-30-1} and \eqref{11-30-2} yields $|\m A(1\bar 2)|\ge {n-3\choose a-2}-1$ and $|\m B (2\bar{1})|\ge {n-3\choose b-2}-1.$ 
Moreover, both inequalities are strict unless $a=b=4$ (in which case at least one of them is strict).
The families $\m A(1\bar 2)\subset {[3,n]\choose a-1}$ and $\m B(2\bar 1)\subset {[3,n]\choose b-1}$ are cross-intersecting.
Applying Theorem~\ref{thm:FK} to $\m A(1\bar 2)$ and $\m B(2\bar 1)$ with $n-2, a-1, b-1$ in the roles of $n, a, b$, and $u=a-1, v=b-1$ (using $n>a+b$ and $a,b\ge 4$), we conclude that both families have diversity $0$. 
Hence all sets in $\m A(1\bar 2)\cup \m B(2\bar 1)$ contain a common element, which we may take to be $3$. Then 
\begin{align*}
    |\m B|&\le {n-2\choose b-2}+2{n-3\choose b-2},\\
    |\m A|&\le {n-2\choose a-2}+{n-3\choose a-2}+{n-1-(a-b+2)\choose a-(a-b+2)}.
\end{align*}
If $a>b$, then using $n>a+b$, we have 
${n-1-(a-b+2)\choose a-(a-b+2)}={n-1-(a-b+2)\choose b-2}\le {n-4\choose b-2}<{n-4\choose a-2}$. 
Then, by (\ref{11-30-1}), 
$|\m A|<{n-2\choose a-2}+{n-3\choose a-2}+{n-4\choose a-2}< |\m A'|$, contradicting $|\m A|\ge |\m A'|$.  

Now suppose $a=b$. 
As $|\m A|\ge |\m A'|$ and $ |\m B|\ge |\m B'|$, equality must hold in \eqref{11-30-1}, \eqref{11-30-2}, and in the two upper bounds displayed above.
This forces every set in $\m A_\gamma \cup \m B_\gamma$ to intersect $[3]$ in at least two elements; otherwise some of the upper bounds on $|\m A(\{1,2\})|$, $|\m A(1\bar 2)|$, $|\m B(\{1,2\})|$, $|\m B(2\bar 1)|$ would be strict, violating the equality. 
In fact, equality forces $\m A$ to contain every $a$-set  intersecting $[3]$ in at least two elements, and similarly $\m B$ for $b$-sets. 
Therefore, all elements in $[3]$ have the same degree, and the families are isomorphic to $\m H_{a,2}$ and $\m H_{b,2}$. 
We may select the same element of $[3]$ as the maximum-degree element for both $\m A$ and $\m B$, contradicting the assumption.
\end{proof}  
Combining Claims \ref{5-1-2} and \ref{5-1-1} completes the proof of Proposition \ref{12-15-1}.

\section{Proof of Proposition \ref{11-29-1}}\label{secpro2}
For $\mathcal{F}\subset\binom{[n]}{k}$ we write 
$\mathcal{F}^c:=\{[n]\setminus F:F\in\mathcal{F}\}$. Recall that for  $i\in[k]$ we write 
$\partial_i(\mathcal{F}):=\{A\in\binom{[n]}{i}:A\subset F\text{ for some }F\in\mathcal{F}\}$ 
for the $i$-shadow of $\mathcal{F}$.

We argue by contradiction. Suppose there exists a pair $(\mathcal{A},\mathcal{B})$ satisfying the hypotheses of Proposition~\ref{11-29-1} but for which no $(\mathcal{A}_a,\mathcal{B}_b)$ as described exists. Among all such counterexamples, choose $(\mathcal{A},\mathcal{B})$ that is minimal in the following sense: no other counterexample $(\mathcal{C},\mathcal{D})$ satisfies $(\mathcal{C},\mathcal{D})\precneqq(\mathcal{A},\mathcal{B})$.

{\bf Case 1: $\gamma(\mathcal{A})\ge {n-3\choose a-2}+{n-4\choose a-2}$ and $\gamma(\mathcal{B})\ge {n-3\choose b-2}+{n-4\choose b-2}$.}

If $\m A$ and $\m B$ are L-initial, then
$|\m A(\bar1)|=\gamma(\m A)>0$ and $|\m B(\bar1)|=\gamma(\m B)>0$, 
we get $|\m A|>{n-1\choose a-1}$ and $|\m B|>{n-1\choose b-1}$, contradicting the fact that they are cross-intersecting.  
Hence $(\m A, \m B)$ is not L-initial, and there exists a pair $(U,V)$ satisfying {\bf P} for $(\m A, \m B)$. 

\begin{claim}\label{keyclaim}
 $\gamma(S_{U,V}(\m A))<{n-1-(a-b+3)\choose a-(a-b+3)}$ or $\gamma(S_{U,V}(\m B))<{n-4\choose b-3}$.
\end{claim}

\begin{proof}
Suppose otherwise.  
By Lemma \ref{SUV}, $S_{U,V}(\mathcal{A})$ and $ S_{U,V}(\mathcal{B})$ are cross-intersecting. 
Clearly, $|S_{U,V}(\m A)|=|\m A|\ge |\m A'|$ and $|S_{U,V}(\m B)|=|\m B|\ge |\m B'|$. Then $S_{U,V}(\m A), S_{U,V}(\m B)$ satisfy the hypotheses of Proposition~\ref{11-29-1}. By the minimality of $(\mathcal{A},\mathcal{B})$ and $(S_{U,V}(\mathcal{A}),S_{U,V}(\mathcal{B}))\precneqq(\mathcal{A},\mathcal{B})$, the pair $(S_{U,V}(\mathcal{A}),S_{U,V}(\mathcal{B}))$ is not a counterexample, so there exist $\mathcal{A}_a,\mathcal{B}_b$ satisfying conditions (1)–(3) for it.
Since $\gamma(S_{U,V}(\mathcal{A}))\ge\binom{n-1-(a-b+3)}{a-(a-b+3)}$, condition (3) for $S_{U,V}(\mathcal{A})$ gives $\gamma(\mathcal{A}_a)\ge\binom{n-1-(a-b+3)}{a-(a-b+3)}\ge\min\{\binom{n-1-(a-b+3)}{a-(a-b+3)},\gamma(\mathcal{A})\}$, so condition (3) for the original $\mathcal{A}$ also holds. Similarly for $\mathcal{B}$. Hence $\mathcal{A}_a,\mathcal{B}_b$ satisfy (1)–(3) for $(\mathcal{A},\mathcal{B})$, contradicting the choice of $(\mathcal{A},\mathcal{B})$ as a counterexample.
\end{proof}

Given that $\gamma(\mathcal{A})\ge {n-3\choose a-2}+{n-4\choose a-2}$ and $\gamma(\mathcal{B})\ge {n-3\choose b-2}+{n-4\choose b-2}$ (the condition of Case 1), 
if $|U|=|V|\ge 2$, then $\gamma(S_{U,V}(\mathcal{A}))\ge {n-3\choose a-2}$ and  $\gamma(S_{U,V}(\mathcal{B}))\ge {n-3\choose b-2}$, contradicting Claim \ref{keyclaim}. Thus $|U|=|V|=1$, i.e. this $S_{U,V}$-shift is an $S_{i,j}$-shift of $(\m A, \m B)$ for some $1\le i<j\le n$. 
We show that 
\begin{itemize}
\item[(i)] $\gamma(S_{i, j}(\mathcal{B}))\le {n-4\choose b-3}$;
\item[(ii)] $|\mathcal{B}(i \bar{j})|, |\mathcal{B}(\bar{i} j)|\ge 2{n-4\choose b-2}$;
\item[(iii)] $|\mathcal{B}(i \bar{j}) \cap \mathcal{B}(\bar{i} j)|\le {n-4\choose b-3}$.
\end{itemize}

To prove (i), assume $\gamma(S_{i, j}(\mathcal{B}))> {n-4\choose b-3}$.
Then Claim \ref{keyclaim} gives $\gamma(S_{i, j}(\mathcal{A}))< {n-1-(a-b+3)\choose a-(a-b+3)}$. 
By (\ref{11-30-1}), $\Delta(S_{i, j}(\mathcal{A})):=|S_{i, j}(\mathcal{A})(1)|\ge{n-1\choose a-1}-{n-4\choose a-1}$. 
As $S_{i, j}(\mathcal{A})$ and $S_{i, j}(\mathcal{B})$ are cross-intersecting,  
$\gamma(S_{i, j}(\mathcal{B}))\le |S_{i, j}(\mathcal{B})(\bar{1})|\le{n-4\choose b-3}$, a contradiction. Thus (i) holds. 
Using (i) together with $\gamma(\mathcal{B})\ge {n-3\choose b-2}+{n-4\choose b-2}$, yields $\gamma(\mathcal{B})-\gamma(S_{i,j}(\mathcal{B}))\ge 2{n-4\choose b-2}$. 
Since the $S_{i,j}$-shift only increases the degree of $i$, the strict reduction of $\gamma(\m B)$ implies that  $i$ becomes the unique  maximum-degree element in $S_{i,j}(\m B)$. 
The sets in $\mathcal{B}(\bar{i}\bar{j})$ and $\mathcal{B}(i \bar{j}) \cap \mathcal{B}(\bar{i} j)$ stay intact after $S_{i,j}$-shift, and provide lower bounds on  $\gamma(S_{i,j}(\mathcal{B}))$, i.e. $|\mathcal{B}(\bar{i}\bar{j})|\le \gamma(S_{i,j}(\mathcal{B}))\le{n-4\choose b-3}$ and $|\mathcal{B}(i \bar{j}) \cap \mathcal{B}(\bar{i} j)|\le \gamma(S_{i,j}(\mathcal{B}))\le{n-4\choose b-3}$.  This establishes (iii). Consequently, we have 
$
|\mathcal{B}(\bar{i}j)|=|\mathcal{B}(\bar{i})|-|\mathcal{B}(\bar{i}\bar{j})|\ge \gamma(\mathcal{B})-|\mathcal{B}(\bar{i}\bar{j})|\ge 2{n-4\choose b-2}, 
$
and similarly, $|\mathcal{B}(i \bar{j})|\ge 2{n-4\choose b-2}$, proving (ii).

We remind the reader of the notation $\mathcal{F}[X, Y]=\big\{F: F\cap Y = X,\; F\in \mathcal{F} \big\}$. 
\begin{claim}\label{10-28-2}
The pair $(\m A, \m B)$ is shifted for all $S_{k,\ell}$ with $\{i, j\}\cap \{k,\ell\}=\emptyset$.
\end{claim}

\begin{proof}
Suppose for contradiction that there exist $k,\ell\in [n]\setminus \{i,j\}$ and $k<\ell$ such that $S_{k,\ell}(\m B)\ne \m B$. Then by (i), we have 
$\gamma(S_{k, \ell}(\mathcal{B}))\le {n-4\choose b-3}.$
Since the $S_{k,\ell}$-shift only increases the degree of $k$, the strict reduction of $\gamma(\m B)$ implies that  $k$ becomes the unique  maximum-degree element in $S_{k,\ell}(\m B)$. Thus $|\m B(\bar k)|\le {n-4\choose b-3}$.
On the other hand,
\begin{align*}
|\m B(\bar k)|&\ge |\m B(\bar k \bar \ell)|\ge |\m B[\{i\},\{i,j,k,\ell\}]|+|\m B[\{j\},\{i,j,k,\ell\}]|\\
& \ge |\m B[\{i\},\{i,j,k\}]|+|\m B[\{j\},\{i,j,k\}]|-\tbinom{n-4}{b-2}\\
& \ge |\m B[\{i\},\{i,j\}]|+|\m B[\{j\},\{i,j\}]|-\tbinom{n-3}{b-2}-\tbinom{n-4}{b-2}\\
&\overset{(ii)}{\ge}4\tbinom{n-4}{b-2}-\tbinom{n-3}{b-2}-\tbinom{n-4}{b-2}\\
&>\tbinom{n-4}{b-3},
\end{align*}
where the last inequality holds by $n\ge 2b$. This contradicts $|\m B(\bar k)|\le {n-4\choose b-3}$.
\end{proof}

We now invoke a result of Frankl and Kupavskii \cite{FK2021jcta}, which shows that any two shifted families are positively correlated. 
\begin{lemma}[Frankl--Kupavskii \cite{FK2021jcta}]\label{10-28-3}
Let $\mathcal{F}_1, \mathcal{F}_2\subset {[n]\choose k}$ be shifted. Then 
\begin{equation}
|\mathcal{F}_1\cap \mathcal{F}_2|\ge |\mathcal{F}_1| |\mathcal{F}_2|/{n\choose k}.
\end{equation}
\end{lemma}

We may w.l.o.g. assume $\{i,j\}=\{1,2\}$. 
The pair $(\m A, \m B)$ is shifted for all $S_{k,\ell}$ with $\{i, j\}\cap \{k,\ell\}=\emptyset$.
By Claim~\ref{10-28-2}, $\mathcal{B}$ is $S_{k,\ell}$-shifted for all ${k,\ell}\subset[3,n]$, which implies that both $\mathcal{B}(1\bar{2})$ and $\mathcal{B}(\bar{1}2)$ are shifted families on $[3,n]$.
Combining Lemma \ref{10-28-3} with (ii) gives 
\[
|\mathcal{B}(1\bar{2})\cap\mathcal{B}(\bar{1}2)| \ge \frac{|\mathcal{B}(1\bar{2})| |\mathcal{B}(\bar{1}2)|}{\binom{n-2}{b-1}} \ge \frac{4\binom{n-4}{b-2}^2}{\binom{n-2}{b-1}} > \binom{n-4}{b-3},
\]
where the last inequality follows from $n\ge 2b$ and ${n-4\choose b-3}{n-2\choose b-1}\le {n-3\choose b-2}^2<4{n-4\choose b-2}^2$. This contradicts (iii), completing Case 1.

{\bf Case 2:
$
\gamma(\mathcal{A})\le {n-3\choose a-2}+{n-4\choose a-2}-1$ or $\gamma(\mathcal{B})\le {n-3\choose b-2}+{n-4\choose b-2}-1$.}

{\bf Case 2.1: 
$\gamma(\mathcal{A})\le {n-3\choose a-2}+{n-4\choose a-2}-1$.} 
W.l.o.g. assume $\m A_{\Delta}=\m A(1)$.
Using Lemma~\ref{11-23-3} together with $|\mathcal{A}|\ge|\mathcal{A}'|$ yields
\begin{equation}\label{a1lower}
    |\m A(1)|=|\mathcal{A}|-\gamma(\m A)\ge\binom{n-2}{a-2}+\binom{n-1-(a-b+3)}{a-(a-b+3)}.
\end{equation}

Define $\mathcal{C}, \mathcal{D}$ as follows: 
$$\mathcal{C}(1)=\mathcal{L}([2,n],|\mathcal{A}(1)|, a-1), \,\,\mathcal{C}(\bar{1})=\mathcal{L}([2,n],|\mathcal{A}(\bar{1})|, a),$$ 
and similarly for $\mathcal{D}$. Then the following hold:
\begin{itemize}
    \item[(a)] Every set of $\mathcal{C}(\bar{1})$ contains element $2$ (by $|\m C(\bar{1})|=|\m A(\bar1)|=\gamma(\m A)<{n-3\choose a-2}+{n-4\choose a-2}\le {n-2\choose a-1}$).
    \item[(b)] $\gamma(\mathcal{C})=\min \{|\mathcal{C}(\bar12)|, |\mathcal{C}(1\bar{2})|\}$ (by (a)). 
    \item[(c)] $\mathcal{C}(\{1,2\})={[3,n]\choose a-2}$ (by $|\m C(1)|=|\m A(1)|$ and (\ref{a1lower})). 
\end{itemize}

By Theorem \ref{12-1-1}, $\mathcal{C}, \mathcal{D}$ are cross-intersecting.
Let $\mathcal{D}'\subset {[n]\choose b}$ be such that 
$\m D'(1)$ and $\m D'(\bar1)$ are L-initial on $[2,n]$, and 
$\m D'$ is maximal cross-intersecting with $\m C$.  
Then $\mathcal{D}\subset \mathcal{D}'$. Moreover, 
\begin{itemize}
    \item[(d)] $\mathcal{D}'(\{1,2\})={[3,n]\choose b-2}$ (by (a) and since $\m D'$ is maximal).
    \item[(e)] The families $\mathcal{D}'(1\bar{2})$ and $ \mathcal{D}'(\bar{1}2)$ are maximal cross-intersecting with $\mathcal{C}(\bar{1}2)$ and $\mathcal{C}(1\bar{2})$, respectively (since $\m D'$ is maximal).
    \item[(f)] Every set of $\mathcal{D}'(\bar{1})$ contains element $2$ (by (c) and $n>a+b$).
    \item[(g)] $\gamma(\mathcal{D}')= \min \{ |\mathcal{D}'(\bar{1}2)|, |\mathcal{D}'(1\bar{2})|\}\ge \min \{ |\mathcal{D}(\bar{1}2)|, |\mathcal{D}(1\bar{2})|\}$ (by (f)). 
    \item[(h)] $|\mathcal{C}(\bar12)|\le |\mathcal{C}(1\bar{2})| \Leftrightarrow|\mathcal{D}'(\bar{1}2)|\le |\mathcal{D}'(1\bar{2})|$.
\end{itemize}

\begin{proof}[Proof of (h)]
Assume $|\mathcal{C}(\bar12)|\le |\mathcal{C}(1\bar{2})|$.
Since $\m C(\bar12)$ and $\m C(\bar12)$ are L-initial on $[3,n]$, $\mathcal{C}(\bar12)\subset \mathcal{C}(1\bar{2})$.
Consequently, 
$(\mathcal{C}(\bar12))^c\subset (\mathcal{C}(1\bar{2}))^c$.
Furthermore, 
$\partial_{b-1}(\mathcal{C}(\bar12)^c)\subset \partial_{b-1}(\mathcal{C}(1\bar{2})^c)$.
So
$|\partial_{b-1}(\mathcal{C}(\bar{1}2)^c)| \le |\partial_{b-1}(\mathcal{C}(1\bar{2})^c)|$. Using (e), 
\begin{align*}
    &|\mathcal{D}'(1\bar{2})|=|{[3,n]\choose b-1}\setminus \partial_{b-1}(\mathcal{C}(\bar{1}2)^c)|={n-2\choose b-1}-|\partial_{b-1}(\mathcal{C}(\bar{1}2)^c)|,\\
    &|\mathcal{D}'(\bar{1}2)|=|{[3,n]\choose b-1}\setminus \partial_{b-1}(\mathcal{C}(1\bar{2})^c)|={n-2\choose b-1}-|\partial_{b-1}(\mathcal{C}(1\bar{2})^c)|.
\end{align*}
Combining these with $|\partial_{b-1}(\mathcal{C}(\bar{1}2)^c)| \le |\partial_{b-1}(\mathcal{C}(1\bar{2})^c)|$, we obtain $|\mathcal{D}'(\bar{1}2)|\le |\mathcal{D}'(1\bar{2})|$.
A symmetric argument shows that if  $|\mathcal{C}(\bar12)|> |\mathcal{C}(1\bar{2})|$, then $|\mathcal{D}'(\bar{1}2)|> |\mathcal{D}'(1\bar{2})|$. 
\end{proof}

We now show that $\m C, \m D'$ can play the roles of $\m A_a, \m B_b$ for $\m A, \m B$, thereby contradicting the choice of $(\m A, \m B)$. Since  $|\m C|=|\m A|$ and $|\m D'|\ge |\m D|=|\m B|$, we only need to verify conditions 2 and 3 in Proposition \ref{11-29-1}. 

Suppose $|\mathcal{C}(\bar12)|\le |\mathcal{C}(1\bar{2})|$. 
By (h), $|\mathcal{D}'(\bar12)|\le |\mathcal{D}'(1\bar{2})|$.
Then $\m C_{\Delta}=\m C(1)$ and $\m D_{\Delta}=\m D(1)$. 
By (b) and (c), 
\begin{align*}
    &\gamma(\mathcal{C})=|\mathcal{C}(\bar12)|=|\mathcal{A}(\bar{1})|=\gamma(\m A)\ge \min \{{n-1-(a-b+3)\choose a-(a-b+3)}, \gamma(\mathcal{A})\},\\
    &\Delta(\mathcal{C})=|\m C(1)|={n-2\choose a-2}+|\m C(1\bar2)|\ge {n-2\choose a-2}+|\m C(\bar12)|= {n-2\choose a-2}+\gamma(\mathcal{C}).
\end{align*}
By (d) and (f), 
\begin{align*}
    &\gamma(\mathcal{D}')=|\mathcal{D}'(\bar1)|\ge |\mathcal{D}(\bar1)|=|\m B(\bar1)|\ge \gamma(\m B)\ge \min \{{n-4\choose b-3}, \gamma(\mathcal{B})\},\\
    &\Delta(\mathcal{D}')=|\m D'(1)|={n-2\choose b-2}+|\m D'(1\bar2)|\ge {n-2\choose b-2}+|\m D'(\bar12)|= {n-2\choose b-2}+\gamma(\mathcal{D}').
\end{align*}

Suppose $|\mathcal{C}(\bar12)|> |\mathcal{C}(1\bar{2})|$. 
By (h), $|\mathcal{D}'(\bar12)|> |\mathcal{D}'(1\bar{2})|$.
Then $\m C_{\Delta}=\m C(2)$ and $\m D_{\Delta}=\m D(2)$. 
By (\ref{a1lower}), $|\m C(1)|=|\m A(1)|=\Delta(\mathcal{A})\ge {n-2\choose a-2}+{n-1-(a-b+3)\choose a-(a-b+3)}$. Using (c), 
\begin{align*}
    &\gamma(\mathcal{C})=|\mathcal{C}(1\bar2)|= |\mathcal{C}(1)|-{n-2\choose a-2}\ge {n-1-(a-b+3)\choose a-(a-b+3)} \ge \min \{{n-1-(a-b+3)\choose a-(a-b+3)}, \gamma(\mathcal{A})\},\\
    &\Delta(\mathcal{C})={n-2\choose a-2}+|\m C(\bar12)|> {n-2\choose a-2}+|\m C(1\bar2)|= {n-2\choose a-2}+\gamma(\mathcal{C}).
\end{align*}
Since $\m D_{\Delta}=\m D(2)$, $\gamma(\mathcal{D}')=|\mathcal{D}'(1\bar{2})|$. Using (d), 
\begin{align*}
    &\Delta(\mathcal{D}')=|\m D'(2)|={n-2\choose b-2}+|\m D'(\bar12)|> {n-2\choose b-2}+|\m D'(1\bar2)|= {n-2\choose b-2}+\gamma(\mathcal{D}').
\end{align*}
Note that $\m C(\bar1)$ is L-initial on $[2,n]$, $|\m C(\bar12)|=|\m A(\bar{1})|=\gamma(\mathcal{A})< {n-3\choose a-2}+{n-4\choose a-2}$. 
Then every set in $\m C(\bar12)$ must intersect $\{3,4\}$.
Since (e) holds, all $b$-sets $B$ with $\{3,4\}\subset B\subset [3, n]$ belong to $\m D'(1\bar2)$. 
We obtain 
$$\gamma(\mathcal{D}')=|\mathcal{D}'(1\bar{2})|\ge {n-4\choose b-3}\ge \min \{{n-4\choose b-3}, \gamma(\mathcal{B})\}.$$ 

Hence, in either case, taking $\m A_a:=\m C$ and $\m B_b:=\m D'$ verifies conditions 1,2,3 a contradiction. 

\textbf{Case 2.2: $\gamma(\mathcal{B})\le\binom{n-3}{b-2}+\binom{n-4}{b-2}-1$.}
W.l.o.g. $\mathcal{B}_\Delta=\mathcal{B}(1)$. By Lemma~\ref{11-23-3} and the bound 
on $\gamma(\mathcal{B})$, we replace (\ref{a1lower}) by
$|\m B(1)|\ge\binom{n-2}{b-2}+\binom{n-4}{b-3}.$
Keeping $\mathcal{C}, \m D$ as defined in Case 2.1.  
Let $\mathcal{C}'$ be the maximum family cross-intersecting with $\mathcal{D}$. 
Then for $\m D$, (a)--(c) hold; for $\m C'$, (d)--(h) hold. 
The argument now proceeds exactly as in Case 2.1, with $\binom{n-4}{b-3}$ 
playing the role of $\binom{n-1-(a-b+3)}{a-(a-b+3)}$ in the case analysis on whether
$|\m D(\bar12)|$ is bigger or smaller than $|\mathcal{D}(1\bar{2})|$. We obtain 
cross-intersecting $\mathcal{A}_a:=\mathcal{C}',\mathcal{B}_b:=\mathcal{D}$ 
satisfying conditions 1--3, contradicting the choice of $(\mathcal{A},\mathcal{B})$.
Full details are provided in the
Appendix.

\section{Proof of Proposition \ref{10-25-3}}\label{secpro3}
We begin in Subsection \ref{outline} with a high‑level overview, splitting the proof into four stages, and describing the objective of each stage. Subsection \ref{tool} establishes the technical tool--Lemmas \ref{SUV'} and \ref{1-7-1} that will be frequently used. The complete, detailed proof is then presented in Subsection \ref{proof}. 

We start with a lower bound of $|\mathcal{B}|$. 
Since $\mathcal{B}'(1)$ is maximally cross‑intersecting with $\mathcal{M}$ and $|\cap\, \mathcal{M}| = m$, 
we have $|\mathcal{B}'(1)| \ge \binom{n-1}{b-1} - \binom{n-1-m}{b-1}$; 
together with $|\mathcal{T}| = |\mathcal{B}'(\bar1)|$ this yields
\begin{equation}\label{lowerbound}
    |\m B|\ge \mathcal{B}'| \ge \binom{n-1}{b-1} - \binom{n-1-m}{b-1} + |\mathcal{T}|.
\end{equation}

\subsection{Outline and Strategy}\label{outline}
The proof of Proposition~\ref{10-25-3} resembles that of Lemma~\ref{lemkey} at a high level, but rests on a completely different underlying reason. Instead of doing bipartite switching and using the upper bound of the sum of the sizes of cross-intersecting families, we perform certain $S_{U,V}$-shifts to preserve the sizes of both $\m A$ and $\m B$ while keeping $\m M$ and $\m T$ in the diversity part of the resulting families. This is more complicated to control. Before giving the proof, we develop several tools required for the argument in Subsection \ref{tool}.

We argue by contradiction. The proof proceeds in four steps.
Step 1: establishes a lower bound for $|\m B(\bar1)|$;
Steps 2,3, and 4 successively tighten the upper bound for $|\m B(\bar1)|$.
A contradiction between the resulting lower and upper bounds of $|\m B(\bar1)|$ then completes the argument. More specifically, 

\begin{itemize}
    \item 
{\bf Step 1.}
We will determine an integer $m'$ such that $m'\le m-1$,
$|\m A(\bar{1})|>{n-1-m'\choose a-m'}$ and $|\m B(1)|<{n-1\choose b-1}-{n-1-m'\choose b-1}$. 
By (\ref{lowerbound}), we have 
$
|\m B(\bar1)|=|\m B|-|\m B(1)|>{n-1-(m'+1)\choose b-2}+|\m T|.
$

\item 
{\bf Step 2.}
We will prove Claim \ref{11-10-2} showing that for any $B\in \m B(\bar1)$, $[2,t+1]\subset B$. Then $|\m B(\bar1)|\le {n-1-t\choose b-t}$. 

\item 
{\bf Step 3.}
We will prove Claim \ref{11-13-1} showing that for any $B\in \m B(\bar1)\setminus \m T$, $[2,t+1+\ell]\subset B$. Then $|\m B(\bar1)|\le {n-1-(t+\ell)\choose b-(t+\ell)}+|\m T|$. 

\item 
{\bf Step 4.}
We will prove Claim \ref{11-19-1} showing that for any $B\in \m B(\bar1)\setminus \m T$, $[2,m'+2]\subset B$. Then $|\m B(\bar1)|\le {n-1-(m'+1)\choose b-(m'+1)}+|\m T|$. 
\end{itemize} 

Since $n>a+b$ and $a\ge b$, $n-1-(m'+1)>b-(m'+1)+b-2$,   
 there is a contradiction between the lower and upper bounds in Step 1 and Step 4.

We will apply shifting tools:
Daykin's $S_{U,V}$-shift and
a refined version of Lemma \ref{SUV} (see Lemma \ref{SUV'}).
The refinement gives a new approach to preserve structures (see Lemma \ref{1-7-1}).

\subsection{Technical tools: a new approach for preserving structures}\label{tool}
Recall that for two families $\m A\subset {[n]\choose a}, \m B\subset {[n]\choose b}$, if $(\m A,\m B)$ is not L-initial, then there exists $(U,V)$ satisfying {\bf P}. 
Lemma \ref{SUV} states that if $\m A, \m B$ are cross-intersecting, then $S_{U,V}(\m A), S_{U,V}(\m B)$ are also cross-intersecting.
Now we give a stronger result---Lemma \ref{SUV'}---showing that a weaker property {\bf Q} also guarantees cross-intersection. 
Furthermore, we give the `structure-preserving lemma'---Lemma \ref{1-7-1}.

Let us first define this property.

\begin{itemize}
\item[{\bf Q}] Let $U,V\subset [n]$ satisfy 
$U\cap V=\emptyset$, $|U|=|V|$ and $U\prec V$.
We say that $(U,V)$ satisfies {\bf Q} for a family $\m F$ if $S_{U,V}(\m F)\precneqq \m F$, and 
for every $V'\subsetneqq V$, there exists $U'\subsetneqq U$ such that $U'\prec V'$, $|U'|=|V'|$, and $S_{U',V'}(\m F)=\m F$. 
\end{itemize}

Note the difference with the property {\bf P}: instead of `for any $V'$ and any $U'$' we have `for any $V'$ there exists $U'$'.

To emphasize the choice of $(U,V)$ that satisfies {\bf Q}, we denote the corresponding $S_{U,V}$-shift by $S_{U,V}^Q$ (note that $S_{U,V}^Q$ and $S_{U,V}$ act identically).
For a family $\m F$, an $S_{U,V}^Q$-shift of $\m F$ means that $(U,V)$ satisfies {\bf Q} for $\m F$;
for a pair of families $(\mathcal A,\mathcal B)$,
an $S_{U,V}^Q$-shift of $(\mathcal A,\mathcal B)$
means that $(U,V)$ satisfies {\bf Q} 
for at least one of $\mathcal A$ and $\mathcal B$.
Throughout, an $S_{U,V}^Q$-shift of $\m F$ (or $(\mathcal A,\mathcal B)$) is understood to act non-trivially. 
From the definition, if $S_{i,j}$-shift acts on $\m F$ non-trivially, then it is also an $S_{\{i\},\{j\}}^Q$-shift of the family.

\begin{lemma}\label{SUV'}
Let $n\ge a+b$, \(\mathcal{A}\subset {[n]\choose a}\) and \(\mathcal{B}\subset {[n]\choose b}\) be cross-intersecting. 
Let $(U,V)$ satisfy {\bf Q} for $(\m A, \m B)$. Then \(S_{U,V}^{Q}(\mathcal{A})\) and \(S_{U,V}^{Q}(\mathcal{B})\) are also cross-intersecting.
\end{lemma}

\begin{proof}
Suppose for contradiction that \(S_{U,V}^{Q}(\mathcal{A})\) and \(S_{U,V}^{Q}(\mathcal{B})\) are not cross-intersecting.
Then there exist $A\in S_{U,V}^Q(\mathcal{A})$ and $B\in S_{U,V}^Q(\mathcal{B})$ such that $A\cap B=\emptyset$.
By the definition of $\bf Q$, $(S_{U,V}^Q(\mathcal{A}),S_{U,V}^Q(\mathcal{B}))\precneqq (\mathcal{A},\mathcal{B})$, so at least one of $\m A, \m B$ is changed. 
We may w.l.o.g. assume this $S_{U,V}^Q$-shift is of $\m A$.

Consider the case $A\in S_{U,V}^Q(\mathcal{A})\setminus \m A$ and $B\in S_{U,V}^Q(\mathcal{B})\cap \m B$. Let $A':=(A\setminus U )\cup V$ and $V':=B\cap A'$. So $A'\in \m A$, and $V' \ne \emptyset$. 
 We show that $V'\subsetneqq V$. 
As $B\cap A=\emptyset$ and $A=(A'\setminus V) \cup U$, we have
$B\cap (A'\setminus V)=\emptyset$ and $B\cap U=\emptyset$. Thus, $V'\subset V$. 
If $V'=V$, then $S_{U,V}(B)=(B\setminus V) \cup U$; thus, $S_{U,V}(B)\cap A'=\emptyset$. This implies $S_{U,V}(B)\not\in \m B$. 
However, given that $B\in S_{U,V}^Q(\mathcal{B})\cap \m B$, we have $S_{U,V}(B)\in \m B$, 
a contradiction. Thus  $V'\subsetneqq V$. 
Fix an arbitrary $U'$ satisfying: $U'\subsetneqq U$, $|U'|=|V'|$ and $U'\prec V'$. Let $A'':=A'\setminus V' \cup U'$. Then $A''\not\in \m A$, since otherwise, $A''\cap B=\emptyset$, contradicting the cross-intersection. So $S_{U',V'}(\mathcal{A})\precneqq \m A$ and $(S_{U',V'}(\m A), S_{U',V'}(\m B))\precneqq (\m A, \m B)$. Since $U'$ was arbitrary, contradicting that $(U,V)$ satisfies {\bf Q} for $\m A$.  

Consider the case $B\in S_{U,V}^Q(\mathcal{B})\setminus \m B$ and $A\in S_{U,V}^Q(\mathcal{A})\cap \m A$. Write $B':=(B\setminus U )\cup V$ and $V':=B'\cap A$. So $B'\in \m B$, and $V' \ne \emptyset$. 
We show that $V'\subsetneqq V$. 
As $B\cap A=\emptyset$ and $B=(B'\setminus V) \cup U$, we have
$A\cap (B'\setminus V)=\emptyset$ and $A\cap U=\emptyset$. Thus, $V'\subset V$. 
If $V'=V$, then $S_{U,V}(A)=(A\setminus V) \cup U$, and hence $S_{U,V}(A)\cap B'=\emptyset$. This implies $S_{U,V}(A)\not\in \m A$. 
However, since $A\in S_{U,V}^Q(\mathcal{A})\cap \m A$, $S_{U,V}(A)\in \m A$, 
a contradiction. Thus  $V'\subsetneqq V$. 
Fix an arbitrary $U'$ satisfying: $U'\subsetneqq U$, $|U'|=|V'|$ and $U'\prec V'$. Let $A':=A\setminus V' \cup U'$. Then $A'\not\in \m A$, since otherwise, $A'\cap B'=\emptyset$, contradicting the cross-intersection. So $S_{U',V'}(\mathcal{A})\precneqq \m A$ and $(S_{U',V'}(\m A), S_{U',V'}(\m B))\precneqq (\m A, \m B)$. Since $U'$ was arbitrary, contradicting that $(U,V)$ satisfies {\bf Q} for $\m A$. 
\end{proof}
Recall that for  a set $C$ with $C\subset [n]$ and $|C|\le k$, denote by $\max C$ the largest element of $C$,
with the convention that $\max \emptyset =0$.
Define
\begin{equation}\label{ec}
    \mathcal{E}^k_C:=\{C\cup T: T\in \tbinom{[\max C+1, n]}{ k-|C|}\}.
\end{equation}
The following definitions will be used frequently in the remainder of the proof.
\begin{itemize}
    \item We say that $C$ is \emph{full in $\mathcal{F}$} if every $k$-set $F$ with $C \subset F \subset [n]$ belongs to $\mathcal{F}$.
    \item We say that $C$ is \emph{empty in $\mathcal{F}$} if 
    $
    \mathcal{F} \cap \mathcal{E}^k_C = \emptyset.
    $
    \item We say that $C$ is {\it non-full in $\mathcal{F}$} if $\mathcal{E}^k_C \setminus \mathcal{F} \ne \emptyset $. 

Note that non-full is not the opposite of full.  
\end{itemize}

Finally, consider an $S_{U,V}^Q$-shift acting on $\mathcal{F}$, and let $\mathcal{S} \subset \mathcal{F}$ be a subfamily. We say $\mathcal{S}$ is \emph{stable under the $S_{U,V}^Q$-shift} if an isomorphic copy of $\mathcal{S}$ is contained in $S_{U,V}^Q(\mathcal{F})$. 
When the shift is clear from context, we simply say $\mathcal{S}$ is \emph{stable}.

Take a family $\m F\subset {X\choose k}$, two sets $R, T\subset X$ and an element $i\not\in R,$ that satisfy the following: $|R|, |T|<k$, $r:=\min X\in R$, $R\cap T=\emptyset$. We define the following properties w.r.t. $\m F, R,T, i$: 

\begin{itemize}
\item[(P1)] For any $S_{U,V}$-shift of $\mathcal{F}$ with $r\in U$, if $S_{U,V}(\m F)\ne \m F$, then $|V|\ge |R|$;
\end{itemize}
\vspace{-0.3em}
\noindent
\begin{minipage}[t]{0.48\textwidth}
\begin{itemize}
\item[(P2)] $\exists F\in\mathcal{F}$ with $R\cap F=\emptyset$; 
\item[(P4)] $R\cup\{i\}$ is non-full in $\mathcal{F}$;
\item[(P6)] $\exists F\in\m F$ with $(R\cup T)\cap F=\emptyset$;
\end{itemize}
\end{minipage}
\hfill
\begin{minipage}[t]{0.48\textwidth}
\begin{itemize}
\item[(P3)] $R$ is non-full in $\mathcal{F}$;
\item[(P5)] $\exists F\in\m F$ with $R\cap F=\emptyset$ and $i\in F$;
\item[(P7)] Every $F$ satisfying (P6) has $i\in F$.
\end{itemize}
\end{minipage}

\begin{lemma}[Structure-preserving lemma]\label{1-7-1}
Suppose $\m F, R, T, i,r$ are defined as above.  
\begin{itemize}
\item[(i)] If (P1)--(P3) hold, then  there exists an $S_{U,V}^Q$-shift of $\m F$ with $R\subset U$. 
\item[(ii)] If (P1)--(P5) hold, then there exists an $S_{U,V}^Q$-shift of $\m F$ with  $R\subset U$ and $i\not\in V$.
\item[(iii)] If (P1)--(P3) and (P6) hold, then there exists an $S_{U,V}^Q$-shift of $\m F$ with $T\cap V=\emptyset$ and $R\subset U$. 
\item[(iv)] If (P1)--(P4) and (P6)--(P7) hold, then there exists an $S_{U,V}^Q$-shift of $\m F$ with $(T\cup\{i\})\cap V=\emptyset$ and $R\subset U$. 
\end{itemize}
\end{lemma}

\begin{proof}
By (P2) and (P3), 
we observe that $\m F$ is not L-initial on $X$. 
We prove (i) by the following three steps.
\begin{itemize}
\item[Step 1.] 
Define
\[
\m F_1:=\{F\in \m F: R\cap F =\emptyset\}, \,\, \m G_1:=\{G\in \tbinom{X}{k}: R\subset G, G\not\in \m F\}.
\]
By (P2) and (P3), $\m F_1\ne \emptyset$ and $\m G_1\ne \emptyset$.

\item[Step 2.] For each $F\in\mathcal F_1$ and $G\in\mathcal G_1$, define
\[
U:=G\setminus F, \,\, V:=F\setminus G.
\]
Then $U\cap V=\emptyset$, $|U|=|V|$, and
$
G=(F\setminus V)\cup U=S_{U,V}(F).
$
Denote by $\m W_1$ the collection of all such pairs $(U,V)$.

As $R\subset G$ and $R\cap F=\emptyset$, $R\subset U$. 
Note that $r=\min X\in R$, $F\cap R=\emptyset$ and $R\subset G$.
We have $G\precneqq F$. 
Consequently, $S_{U,V}(\m F)\precneqq \m F$ for every pair $(U,V)\in \m W_1$, and this provides an $S_{U,V}$-shift of $\m F$. Moreover, by (P1), we have $|V|=|U|\ge |R|$. 

\item[Step 3.] 
Choose $(U,V)\in \m W_1$ with $V$ inclusion-minimal in the following sense: for any $V'\subsetneqq V$ with $|V'|\ge |R|$, there exists $U'\subsetneqq U$ with  $R\subset U'$ and $|U'|=|V'|$ such that $S_{U',V'}$ acts trivially on $\m F$. 
We claim that $(U,V)$ satisfies {\bf Q}. 
Indeed, fix an arbitrary $V'\subsetneq V$. 
If $|V'|\ge |R|$, then by the inclusion-minimal choice of $V$,
every $S_{U',V'}$-shift with
$R\subset U'\subset U$ and $|U'|=|V'|$
acts trivially on $\mathcal F$.
If $|V'|<|R|$, then (P1) shows that every such shift acts trivially as well. 
Hence $(U,V)$ satisfies {\bf Q}.
\end{itemize}

Guided by the above three steps:
Step 1: Define $(\m F_2, \m G_2), (\m F_3, \m G_3)$, $(\m F_4, \m G_4)$.
$\longrightarrow$
Step 2: Construct $\m W_2$, $\m W_3$, $\m W_4$.
$\longrightarrow$
Step 3: Choose $(U,V)\in \m W_2$, $(U,V)\in \m W_3$, $(U,V)\in \m W_4$.

To prove (ii), we define
\[
\m F_2:=\{F\in \m F: R\cap F =\emptyset, i\in F\}, \,\, \m G_2:=\{G\not\in \m F: R\cup \{i\}\subset G\}.
\]
By (P4) and (P5), 
$\m F_2\ne \emptyset$ and $\m G_2\ne \emptyset$. Applying the same three-step argument as in (i) to $\m F_2,\m G_2$ in place of $\m F_1,\m G_1$, there exists an $S_{U,V}^Q$-shift of $\m F$ with $R\subset U$. 
From the definition of $\m F_2$ and $\m G_2$, we have $i\in F$ and $i \in G$. Since $G=(F\setminus V)\cup U$ and $V\cap U=\emptyset$, $i\not\in V$. This proves (ii).

To prove (iii), define
\[
\m F_3:=\{F\in \m F: R\cap F =\emptyset, T\cap F =\emptyset\}, \,\, \m G_3:=\{G\not\in \m F: R\subset G\}.
\]
By (P3) and (P6), 
$\m F_3\ne \emptyset$ and $\m G_3\ne \emptyset$. Applying the same three-step argument to $\m F_3,\m G_3$ as in (i), there exists an $S_{U,V}^Q$-shift of $\m F$ with $R\subset U$. Since every $F\in\mathcal F_3$ satisfies $T\cap F=\emptyset$ and $V\subset F$, we have $T\cap V=\emptyset$. This proves (iii).

To prove (iv), define
\[
\m F_4:=\{F\in \m F: (R\cup T)\cap F =\emptyset\}, \,\, \m G_4:=\{G\not\in \m F: R\cup \{i\}\subset G\}.
\]
By (P4) and (P6),
$\m F_4\ne \emptyset$ and $\m G_4\ne \emptyset$. By (P7), every $F\in \m F_4$ satisfies $i\in F$, and thus, arguing as in (ii), we get that $i\notin V.$ Applying the same three-step argument as in (i), there exists an $S_{U,V}^Q$-shift of $\m F$ with $R\subset U$ and $i\not\in V$. Clearly, from the definition of $\m F_4$, $T\cap V=\emptyset$. This proves (iv).
\end{proof}

\subsection{Proof of Proposition \ref{10-25-3}}\label{proof}

Let $\m A, \m B, \m A', \m B', \m M, \m T$ satisfy Assumption \ref{12-14-2}, and $|\m A|\ge |\m A'|$, $|\m B|\ge |\m B'|$, 
$\Delta(\mathcal{A})=|\mathcal{A}(1)|\ge {n-2\choose a-2}+\gamma(\mathcal{A})$ and 
$\Delta(\mathcal{B})=|\mathcal{B}(1)|\ge {n-2\choose b-2}+\gamma(\mathcal{B})$.
Proposition \ref{10-25-3} states that there exist $\m A^a, \m B^b$ satisfying conditions 1--4 in Lemma \ref{lemmain}.
Arguing by contradiction, among all pairs of families that violate the statement of Proposition \ref{10-25-3}, choose $(\mathcal{A}, \mathcal{B})$ minimal in the following sense: there is no other counterexample pair $(\mathcal{C}, \mathcal{D})$  such that $(\mathcal{C}, \mathcal{D}) \precneqq (\mathcal{A}, \mathcal{B})$. We refer to $(\m A, \m B)$ as the \emph{minimal counterexample} (or simply \emph{minimal}). 
Clearly, $\mathcal{A}$ and $\mathcal{B}$ are not L-initial; otherwise, one of the families would be contained in a star, contradicting Assumption \ref{12-14-2}. Throughout this section, we identify isomorphic copies of
$\mathcal M$ and $\mathcal T$.

\begin{obser}\label{obser2}
    For a family $\m F\subset {[n]\choose k}$ and $i\in [n]$, if $|\m F(i)|\ge {n-2\choose k-2}+|\m F(\bar i)|$, then we may assume that $\m F_{\Delta}=\m F(i)$.
\end{obser}
\begin{proof}
    Let $j\ne i$. Then $|\m F(j)|=|\m F(\{i,j\})|+|\m F(\bar i j)|\le {n-2\choose k-2}+|\m F(\bar i)|\le |\m F(i)|$. Thus, $|\m F(i)| = \Delta(\m F)$.
\end{proof}

Notice that for a disjoint pair $U,V$, if $U\prec V$ and $1\not\in U$, then $1\not\in V$. We have the following observation.
\begin{obser}\label{obser}
Let $\m F\subset {[n]\choose k}$. Then for any $S_{U,V}$-shift of $\m F$ with $U\prec V$, we have $|S_{U,V}(\m F)(1)|\ge |\m F(1)|$. 
\end{obser}


The following claim is a consequence of the minimality of $(\m A, \m B)$ and will be  used frequently. 

\begin{claim}\label{11-3-4}
Let $\mathcal{C}\subset{[n]\choose a}, \mathcal{D}\subset {[n]\choose b}$ be cross-intersecting families satisfying  $|\mathcal{C}|= |\mathcal{A}|$, $|\mathcal{D}|= |\mathcal{B}|$,
$|\m C(1)|\ge |\m A(1)|$, $|\m D(1)|\ge |\m B(1)|$,
 and $(\mathcal{C}, \mathcal{D})\precneqq (\mathcal{A},\mathcal{B})$. 
 Then $\mathcal{M}\not\subset \mathcal{C}(\bar1)$ or $\mathcal{T}\not\subset \mathcal{D}(\bar1)$. 
 In particular, 
for any $S^Q_{U,V}$-shift  of $(\m A, \m B)$, 
we have $\mathcal{M}\not\subset S^Q_{U,V}(\mathcal{A})(\bar1)$ or $\mathcal{T}\not\subset S^Q_{U,V}(\mathcal{B})(\bar1)$.  
\end{claim}

\begin{proof}
Assume for contradiction that $\mathcal{M}\subset \mathcal{C}(\bar1)$ and $\mathcal{T}\subset\mathcal{D}(\bar1)$.
Since $|\mathcal{C}|= |\mathcal{A}|$ and $|\m C(1)|\ge |\m A(1)|$, $|\m C(\bar1)|\le |\m A(\bar1)|$.
By $\Delta(\m A)=|\m A(1)|$, we have $|\m C(1)|\ge |\m A(1)|\ge {n-2\choose a-2}+|\m A(\bar1)|\ge {n-2\choose a-2}+|\m C(\bar1)|$.
By Observation \ref{obser2}, $\m C_{\Delta}=\m C(1)$ and $\m C_{\gamma}=\m C(\bar1)$.
Therefore, $\Delta(\m C)=|\m C(1)|\ge {n-2\choose a-2}+\gamma(\m C)$.
Similarly, we have $\m D_{\Delta}=\m D(1)$, $\m D_{\gamma}=\m D(\bar1)$ and $\Delta(\m D)=|\m D(1)|\ge {n-2\choose b-2}+\gamma(\m D)$. 
Then $\mathcal{M}\subset \mathcal{C}_{\gamma}$ and $\mathcal{T}\subset\mathcal{D}_{\gamma}$.
Since $(\mathcal A,\mathcal B)$ was chosen to maximize $|\mathcal A|+|\mathcal B|$ among all pairs satisfying
Assumption~\ref{12-14-2}, 
and since $|\mathcal{C}|= |\mathcal{A}|$ and $|\mathcal{D}|= |\mathcal{B}|$,
$\mathcal{C}, \mathcal{D}$ also satisfy Assumption \ref{12-14-2}, and hence the hypotheses of Proposition \ref{10-25-3}. 
Since $(\mathcal{A}, \mathcal{B})$ is minimal and $(\mathcal{C}, \mathcal{D})\precneqq (\mathcal{A},\mathcal{B})$, Proposition \ref{10-25-3} holds for $\mathcal{C}$, $\mathcal{D}$. 
However, any families $\mathcal{A}^a$ and $\mathcal{B}^b$ satisfying the conclusion of Proposition~\ref{10-25-3} for $\mathcal{C}$ and $\mathcal{D}$ would also satisfy it for the original pair $\mathcal{A}$ and $\mathcal{B}$, a contradiction.

In particular, 
consider an $S_{U,V}^Q$-shift of $(\m A, \m B)$, and 
let
$
\mathcal C=S_{U,V}^Q(\mathcal A)$ and $
\mathcal D=S_{U,V}^Q(\mathcal B).
$
By Observation~\ref{obser} and Lemma~\ref{SUV'},
the pair
$(\mathcal C,\mathcal D)$
satisfies the hypotheses above.
Hence the desired conclusion follows from the first part.
\end{proof}

Observe that if $\gamma(\mathcal{A})\le {n-1-(a-b+2)
\choose a-(a-b+2)}$ and $\gamma(\mathcal{B})\le {n-3\choose b-2}$, then $\m A, \m B$ themselves satisfy the statement of Proposition \ref{10-25-3}, a contradiction. So  $\gamma(\mathcal{A})> {n-1-(a-b+2)
\choose a-(a-b+2)}$ or $\gamma(\mathcal{B})> {n-3\choose b-2}$.
 By Proposition \ref{1-4-3}, 
 we may assume that $$\gamma(\mathcal{A})> {n-1-(a-b+2)
\choose a-(a-b+2)}.$$ 

\begin{claim}\label{step1.1}
Every set in $\mathcal{A}(\bar{1})\cup \mathcal{B}(\bar{1})$ contains $2$, and we may assume that $\cap \m M=[2,m+1]$, $\cap \m T=[2,t+1]$.
\end{claim} 

\begin{proof}
Note that $n>a+b$ and neither $\m A$ nor $\m B$ is a star.
We have $|\m A|<{n-1\choose a-1}$ and $|\m B|<{n-1\choose b-1}$. 
Combining this together with $\Delta(\m A)=|\m A(1)|\ge {n-2\choose a-2}+\gamma(\m A)={n-2\choose a-2}+|\m A(\bar1)|$,
yields $|\m A(\bar1)|<{n-2\choose a-1}$. Hence,
$\{2\}$ is non-full in $\mathcal{A}$, and a same conclusion for $\m B$. 

If the pair $(\mathcal{A}(\bar{1}), \mathcal{B}(\bar{1}))$ is L-initial on $[2,n]$, then we may assume that $\cap\, \m M=[2,m+1]$, $\cap\, \m T=[2,t+1]$, and every set in $\mathcal{A}(\bar{1})\cup \mathcal{B}(\bar{1})$ contains element $2$ (since $|\m A(\bar1)|<{n-2\choose a-1}$ and $|\m B(\bar1)|<{n-2\choose b-1}$).  
Assume that $(\mathcal{A}(\bar{1}), \mathcal{B}(\bar{1}))$ is not L-initial.  

First, we may assume that $2 \in (\cap \m M)\cap (\cap \m T)$. Indeed, if $2 \not \in \cap \m M$ (resp. $2 \not \in \cap \m T$), then fix some  $i\in \cap \m M$ (resp. $i\in \cap \m T$); clearly $i>2$, and up to isomorphism we have $\m M \subset S_{2,i}(\m A)(\bar1)$ and  $\m T \subset S_{2,i}(\m B)(\bar{1})$.  

Next, we show that  every set in $\mathcal{A}(\bar{1})\cup \mathcal{B}(\bar{1})$ contains $2$. 
Suppose for contradiction, w.l.o.g. that there exists a set $A\in \m A(\bar1)$ with $2\not\in A$. 
Apply Lemma \ref{1-7-1}(i) to $\m A$ with $r=2$ and $R=\{2\}$. 
Conditions (P1)--(P3) are verified as follows: since $\m A$ is not L-initial (otherwise the positive diversity of $\m A$ makes $\m B$ be a star)  and $|R|=|\{2\}|=1$, P1 holds; since $R=\{2\}$, P2 holds by the existence of $A$; P3 holds since $\{2\}$ is non-full in $\mathcal{A}$. 
Hence, there exists an $S_{U,V}^Q$-shift of $\mathcal{A}$ such that $2\in U$. This also gives an $S_{U,V}^Q$-shift of $(\m A, \m B)$. 
Since $2 \in (\cap \m M)\cap (\cap \m T)$ and $2\in U$, $\m M=S_{U,V}^Q(\m M)\subset S_{U,V}^Q(\m A)(\bar{1})$ and $\m T=S_{U,V}^Q(\m T)\subset S_{U,V}^Q(\m B)(\bar1)$, contradicting Claim \ref{11-3-4}. So every set in $\mathcal{A}(\bar{1})\cup \mathcal{B}(\bar{1})$ contains $2$.

Finally, if $\cap \m M=[2,m+1]$ and $\cap \m T=[2,t+1]$, then we are done.
Otherwise, 
we successively apply the $S_{i,j}$-shifts for all $i\in [3,m+1]$, $j>i$ with $j\in \cap \m M$ to $\mathcal{A}(\bar{1})$ and $\mathcal{B}(1)$,  
and for all $i\in [3,t+1]$, $j>i$ with $j\in \cap \m T$ to $\mathcal{B}(\bar{1})$ and $\mathcal{A}(1)$;
and denote by $\m C, \m D$ the resulting families. 
Then $\mathcal{M}'\subset \m C(\bar1)$, $\mathcal{T}'\subset \m D(\bar1)$  are isomorphic to $\mathcal{M}$, $\m T$ with $[2,m+1]=\cap \m M'$, $[2,t+1]=\cap \m T'$, respectively. Note that each set in $\mathcal{A}(\bar{1})\cup \mathcal{B}(\bar{1})$ contains element $2$, and these $S_{i,j}$-shifts preserve cross-intersection. After doing these shifts, $\m C, \m D$ are cross-intersecting.
Moreover, $|\m C|=|\m A|$, 
$|\m D|=|\m B|$,
$|\m C(1)|=|\m A(1)|$,
$|\m D(1)|=|\m B(1)|$.
Since Claim \ref{11-3-4}, $(\m C, \m D)=(\m A, \m B)$, and hence we may take $\m M=\m M'$ and$\m T=\m T'$.
\end{proof}

\begin{claim}\label{12-11-1}
$\mathcal{A}(\bar{1})=\mathcal{L}([2, n], |\mathcal{A}(\bar{1})|, a)$ and $\mathcal{B}(1)=\mathcal{L}([2, n], |\mathcal{B}(1)|, b-1)$.
\end{claim}

\begin{proof}
Assume for contradiction that $\mathcal{A}(\bar{1})\ne \mathcal{L}([2, n], |\mathcal{A}(\bar{1})|, a)$. Construct new families $\m C$ and $\m D$ as follows:  replace $\mathcal{B}(1), \mathcal{A}(\bar{1})$ by $\mathcal{L}([2, n], |\mathcal{B}(1)|, b-1)$, $\mathcal{L}([2, n], |\mathcal{A}(\bar{1})|, a)$, respectively, 
while keeping the parts $\mathcal{A}(1)$ and $\mathcal{B}(\bar{1})$ unchanged.
By Theorem \ref{12-1-1}, together with the fact that every set in $\m A(\bar1)\cup \m B(\bar1)$ contains $2$ (by Claim \ref{step1.1}), the families $\m C$ and $\m D$ are cross-intersecting.  

Note that $|\m C|=|\m A|$, $|\m D|=|\m B|$, $|\m C(1)|=|\m A(1)|$, $|\m D(1)|=|\m B(1)|$, and $(\m C, \m D)\precneqq (\m A, \m B)$. 
Note that $|\mathcal{A}(\bar{1})|\ge {n-1-(a-b+2)\choose a-(a-b+2)}>{n-1-(a-b+3)\choose a-(a-b+3)}\ge {n-1-m \choose a-m}$ and $[2,m+1]=\cap \m M$ (by Claim \ref{step1.1}). 
Since $\m C(\bar{1})=\m L([2,n], |\m A(\bar1)|,a)$, we obtain $\mathcal{M}\subset \m C(\bar{1})$. 
Clearly, $\mathcal{T}\subset \m D(\bar{1})$.
This contradicts Claim \ref{11-3-4}. Thus, $\mathcal{A}(\bar{1})=\mathcal{L}([2, n], |\mathcal{A}(\bar{1})|, a)$. This forces $\mathcal{B}(1)=\mathcal{L}([2, n], |\mathcal{B}(1)|, b-1)$. Otherwise, replace  $\mathcal B(1)$ by
$\mathcal L([2,n],|\mathcal B(1)|,b-1)$ while keeping the other
three parts unchanged. By the same argument, we arrive at a contradiction with Claim \ref{11-3-4}. 
\end{proof}

Note that the proof of Claim~\ref{12-11-1}
uses only the condition $\gamma(\m A)> {n-1-m\choose a-m}$. 
The following claim follows directly by a symmetric argument.
\begin{claim}\label{gammaB>}
 If $\gamma(\m B)> {n-1-t\choose b-t}$, then 
$\mathcal{B}(\bar{1})=\mathcal{L}([2, n], |\mathcal{B}(\bar{1})|, b)$
and
$\mathcal{A}(1)=\mathcal{L}([2, n], |\mathcal{A}(1)|, a-1)$.
\end{claim}

Let $m'$ be the smallest integer such that $\gamma(\mathcal{A})> {n-1-m' \choose a-m'}$. Since $\gamma(\mathcal{A})> {n-1-(a-b+2)\choose a-(a-b+2)}$ and $m\ge a-b+3$, 
we have 
\begin{equation}\label{m'<m}
    m'\le a-b+2< m.
\end{equation}

Since $\mathcal A(\bar1)$ is L-initial on $[2,n]$  by Claim \ref{12-11-1}, together with maximality of $(\m A, \m B)$,  the minimality of $m'$ implies the following.
\begin{claim}\label{abcd}
\begin{itemize}
\item[(i)] For each $A\in \mathcal{A}(\bar{1})$, we have $[2, m']\subset A$, and there exists $A'\in \mathcal{A}(\bar{1})$ such that $m'+1\not\in A'$.
\item[(ii)] For each $i \in [2, m']$, $\{1, i\}$ is full in $\mathcal{B}$.
\item[(iii)] $\{1, m'+1\}$ is non-full in $\mathcal{B}$.
\item[(iv)] For each $i\ge m'+2$, $\{1, i\}$ is empty in $\mathcal{B}$. 
\end{itemize}
\end{claim}

Since $\mathcal{B}(1)=\mathcal{L}([2, n], |\mathcal{B}(1)|, b-1)$, 
 by Claim \ref{abcd}(ii)--(iii),  we have
 \begin{equation}\label{12-12-2}
|\mathcal{B}(1)|< {n-1\choose b-1}-{n-1-m' \choose b-1}.
\end{equation}
Now, we have finished Step 1. We are going into Step 2.   

\begin{claim}
\begin{equation}\label{gammaB<}
\gamma(\m B)\le {n-1-t\choose b-t}. 
\end{equation}
\end{claim}

\begin{proof}
Assume for contradiction that $\gamma(\m B)> {n-1-t\choose b-t}$. By Claims \ref{12-11-1} and \ref{gammaB>}, $\m A(1), \m A(\bar{1}), \m B(1), \m B(\bar{1})$ are L-initial on $[2,n]$. 
Let $t'$ be the smallest integer such that  $\gamma(\mathcal{B})> {n-1-t' \choose b-t'}$. Then $t'\le t$. Moreover, since $\m B(\bar{1})$ is L-initial, 
 for each $B\in \mathcal{B}(\bar{1})$, we have $[2, t']\subset B$, and there exists $B'\in  \mathcal{B}(\bar{1})$ such that $t'+1\not\in B'$. By the maximality of $(\m A, \m B)$,  the minimality of $t'$ implies the following.
 \begin{itemize}
 \item[(a)] For each $i\in [2,t']$, $\{1,i\}$ is full in $\mathcal{A}$;
 \item[(b)] $\{1,t'+1\}$ is  non-full in $\mathcal{A}$;
 \item[(c)] For each $i\ge t'+2$, $\{1,i\}$ is empty in $\m A$ ( and is also non-full in $\m A$).
 \end{itemize}
We claim
\begin{equation}\label{1-6-1}
\text{For every $S_{U,V}^Q$-shift of $(\m A, \m B)$, we have $1\in U$ and $|V|\ge 2$.}
\end{equation}
 
To prove~\eqref{1-6-1},
fix an $S_{U,V}^Q$-shift of $(\m A,\m B)$. 
As $\m A(1), \m A(\bar{1}), \m B(1), \m B(\bar{1})$ are L-initial on $[2,n]$, $1\in U$. 
Since $|\m A(1)|\ge {n-2\choose a-2}+\gamma(\m A)={n-2\choose a-2}+|\m A(\bar{1})|$, $\{1,2\}$ is full in $\mathcal{A}$ and $\m A(2\bar{1})\subset \m A(1\bar{2})$, and similarly for $\m B$. This implies that $S_{1,i}(\m A)=\m A$,  and $S_{1,i}(\m B)=\m B$. Thus, $|V|\ge 2$, as claimed.

In the rest of the proof of (\ref{gammaB<}), when we consider an $S_{U,V}^Q$-shift of $(\m A, \m B)$, by symmetry, we may assume that it is an $S_{U,V}^Q$-shift of $\m A$. In all subsequent applications of Lemma~\ref{1-7-1},
we take $r=1$.

Consider the case $t'= m'$. We apply Lemma \ref{1-7-1} (i) to $\m A$ with $R=\{1,m'+1\}$. The required conditions are verified as follows: condition (P1) follows from (\ref{1-6-1}); condition (P2) holds by Claim \ref{abcd} (i); condition (P3) holds by (b) and $m'=t'$. By Lemma \ref{1-7-1}, there exists an $S^Q_{U,V}$-shift of $\m A$ (and hence of $(\m A, \m B)$) such that $\{1,m'+1\}\subset U$. 
Consequently, $\m M$ and $\m T$ are stable, since $m'+1=t'+1 \le \min \{m, t\}+1 \in [2,m+1]\cap [2,t+1]= (\cap \m M)\cap (\cap \m T)$. This contradicts Claim \ref{11-3-4}.

Consider the case $m' \ne t'$. By symmetry, we may assume $m' > t'$ (the case $m' < t'$ is analogous).  
In light of (b), the rest of the proof splits into two cases.

{\bf Case 1: $\{1,t'+1,m'+1\}$  is non-full in $\m A$.} We apply Lemma \ref{1-7-1}(ii) to $\m A$ with $R=\{1, m'+1\}$ and $i=t'+1$.  
The required conditions are verified as follows: (P1) follows from (\ref{1-6-1});   condition (P2) is given by Claim \ref{abcd}(i); (P3) holds by (c) and $m'+1\ge t'+2$; (P4) is exactly the hypothesis of this case; (P5) follows from Claim \ref{abcd} (i) and $t'+1\le m'$. 
Hence, by Lemma \ref{1-7-1}, there exists an $S^Q_{U,V}$-shift of $\m A$, and hence of $(\m A, \m B)$ such that $\{1,m'+1\}\subset U$ and $t'+1\not\in V$. We claim that $\m T$ is stable under this $S^Q_{U,V}$-shift. Indeed, on the one hand, Claim \ref{abcd}(ii) together with $t'+1\le m'$ yields that $\{1,j\}$ is full in $\m B$ for each $j\in [2,t'+1]$. On the other hand, we have $t'+1\in \cap \m T$, $t'+1\not\in V$ and $1\in U$. Thus, $T$ were to shift into a set containing $\{1,t'+1\}$, but we have all of those in $\m B.$ Thus, $S_{U,V}^Q (T)=T$ for each $T\in \m T$, as claimed.  
Note that $m'+1\in \cap \m M$ and $m'+1\in U$, which implies that $\m M$ is also stable. This contradicts Claim \ref{11-3-4}.

{\bf Case 2: $\{1,t'+1,m'+1\}$  is full in $\m A$.}
Since $\m A(1)$ is L-initial, $\{1,t'+1,j\}$  is full in $\m A$ for each $j\in [t'+2, m'+1]$. 
We apply Lemma \ref{1-7-1} (iv) to $\m A$ with $R=\{1\}$, $T=\{m'+1\}$ and $i=t'+1$. 
The required conditions are verified as follows: 
Condition (P1) follows from (\ref{1-6-1});
conditions (P2) and (P3) are immediate since $R=\{1\}$; (P4) is given by (b); (P6) follows from Claim \ref{abcd}(i); (P7) follows from Claim \ref{abcd}(i) together with $t'+1\le m'$. 
Hence, by Lemma \ref{1-7-1}, there exists an $S^Q_{U,V}$-shift of $\m A$ (and consequently of the pair $(\m A, \m B)$) such that $1\in U$ and $\{t'+1,m'+1\}\cap V=\emptyset$. 
From Claim \ref{abcd}(ii) and $t'+1\le m'$, we know that $\{1,j\}$ is full in $\m B$ for each $j\in [2,t'+1]$. 
Since $t'+1 \in [2,t+1]= \cap \m T$ and $t'+1\not\in V$, 
$\m T$ is stable under this $S^Q_{U,V}$-shift. 
Since $m'<m$ and $t'<m'$, we have $\{t'+1, m'+1\} \subset \cap \mathcal{M}$. By hypothesis, $\{1, t'+1, m'+1\}$ is full in $\mathcal{A}$. 
As $\{t'+1,m'+1\}\cap V=\emptyset$ and $1\in U$, we conclude that  $\m M$ is stable, contradicting Claim \ref{11-3-4}.
\end{proof}

\begin{claim}\label{11-10-2}
For every $B\in \mathcal{B}(\bar{1})$, we have $[2, t+1]\subset B$.
\end{claim}

\begin{proof}
Recall that every set in $\m B(\bar1)$ contains $2$ (by Claim \ref{step1.1}).
Let $i$ be the smallest integer such that every set in $\m B(\bar1)$ contains $[2,i]$. 
Since $\m T \subset \m B(\bar1)$, $i\le t+1$.
If $i= t+1$, then we are done. Assume $i\le t$.
Then there exists $B\in \m B(\bar1)$ satisfying $i+1\not\in B$ and $[2,i]\subset B$.
Write 
\begin{equation*}
\m B_1:=\m B[[2,i],[i+1]], \, \,\m B_2:=\mathcal{E}^b_{[2,i+1]} \setminus \mathcal{B}(\bar{1}).
\end{equation*}
Recall that $\mathcal{B}[[2,i],[i+1]]$ denotes the collection of sets in $\mathcal{B}(\bar{1})$ that contain $[2,i]$ but avoid $i+1$, and $\mathcal{E}^b_{[2,i+1]}$ is defined in (\ref{ec}).
Then $B \in \mathcal{B}_1$, and hence $\mathcal{B}_1 \neq \emptyset$.
Note that $i\le t$, $|\m E^b_{[2,i+1]}|=\tbinom{n-1-i}{b-i}$ and $\gamma(\mathcal{B})\le {n-1-t \choose b-t}$ by (\ref{gammaB<}). 
Since $B\in\mathcal B(\bar1)\setminus \mathcal E^b_{[2,i+1]}$,
not all members of $\mathcal E^b_{[2,i+1]}$ can belong to
$\mathcal B(\bar1)$.
Hence $\mathcal B_2\neq\emptyset$.
Denote 
\[
\m F:=\m A(1)\cap {[i+1,n]\choose a-1} \text{ and } \m G:=\m B([2,i],[i]).
\]
Then $\m F \subset {[i+1, n]\choose a-1}$ and $\m G\subset {[i+1,n]\choose b-i+1}$ are cross-intersecting. Note that $|\m G|=|\m B(\bar 1)|$. 
We apply Lemma \ref{1-7-1}(i) to $\m G$ with $r=i+1$, $R=\{i+1\}$. The required conditions are verified as follows: condition (P1) follows from $|R|=|\{i+1\}|=1$; condition (P2) holds by the existence of $B$ (note that $i+1\not\in B\setminus [2,i] \in \m G$);
since $\m B_2\ne \emptyset$,  $\{i+1\}$ is non-full in $\m G$, so condition (P3) holds. 
Hence, there exists an $S^Q_{U,V}$-shift of $\m G$, and hence of $(\m F, \m G)$ such that $i+1\in U$. 
Now in the pair $(\m A, \m B)$, replace $\m F, \m G$ by $S_{U,V}^Q(\m F), S_{U,V}^Q(\m G)$, respectively, and keep other parts of $\m A, \m B$ unchanged. This yields a new pair, denoted by $(\m I, \m J)$.

We claim that $\m I, \m J$ are cross-intersecting. 
Indeed, by Lemma \ref{SUV'}, $S_{U,V}^Q(\m F), S_{U,V}^Q(\m G)$ are cross-intersecting; combined with the assumption that every set in $\m B(\bar1)$ contains $[2,i]$, shows that $\m I(1)$ and $\m J(\bar1)$ are cross-intersecting. 
Since every set in both $\m A(\bar1)$ and $\m J(\bar1)$ contains $2$, and $\m A(\bar1)=\m I(\bar1)$, it follows that  $\m I$ and $\m J(\bar1)$ are cross-intersecting. We have $\m B[\{1\}]=\m J[\{1\}]$, thus $\m J(1)$ cross-intersects $\m A(\bar 1)$. Also, $\m J[\{1\}]$ and $\m I[\{1\}]$ are clearly cross-intersecting. We conclude that $\m I$ and  $\m J$ are cross-intersecting as well. 

Clearly, $|\m I|=|\m A|$, $|\m J|=|\m B|$, $|\m I(1)|=|\m A(1)|$, $|\m J(1)|=|\m B(1)|$ and $(\m I, \m J)\precneqq (\m A, \m B)$. 
Moreover, $\m M\subset \m A(\bar1)=\m I(\bar1)$. 
As $i+1\in U$ and $i+1\in [2,t+1] =\cap \m T$, $\m T\subset \m J(\bar1)$, contradicting Claim \ref{11-3-4}.
\end{proof}

Now, we have finished Step 2. We are going into Step 3. 

Recall that $(\m A, \m B)$ is a pair of maximal cross-intersecting families.
In view of Claim \ref{11-10-2}, the first statement of the following claim holds since $n>a+b$.
\begin{claim}\label{propertyA}
\begin{itemize}
\item[(i)] For each $i \in [2, t+1]$, $\{1, i\}$ is full in $\mathcal{A}$.
\item[(ii)] $\{1, t+2\}$ is non-full in $\mathcal{A}$.
\end{itemize}
\end{claim}
The second statement follows from the assumption that $\m T \subset \m B(\bar1)$, and that there exists a set in $\m T$ not containing $t+2$.

By \eqref{12-12-2} and \eqref{gammaB<} we obtain the upper bound
\begin{equation}\label{B<}
    |\mathcal{B}|<\binom{n-1}{b-1}-\binom{n-1-m'}{b-1}+\binom{n-1-t}{b-t}.
\end{equation} 

On the other hand, since $m'<m$ and $|\m B|\ge |\m B'|$, inequality (\ref{lowerbound}) gives
\begin{equation}\label{12-12-3}
|\mathcal{B}|\ge {n-1\choose b-1}-{n-1-m\choose b-1}+|\mathcal{T}|\ge {n-1\choose b-1}-{n-1-m'\choose b-1}+{n-2-m'\choose b-2}+|\mathcal{T}|.
\end{equation}
Comparing (\ref{B<}) and (\ref{12-12-3}), we have $t\le m'$. Moreover, we claim that if $m'=t$, then $t\ge 4$.
Indeed, if $t=m'=3$ (note that $t\ge 3$), then (\ref{B<}) gives 
$|\m B|<{n-1\choose b-1}-{n-4\choose b-1}+{n-4\choose b-3}$; combining this with Lemma \ref{11-23-3}, we obtain $|\m B|={n-1\choose b-1}-{n-4\choose b-1}+{n-4\choose b-3}-1$, $b=4$ and $|\m T|=1$, 
which implies $t=4$, a contradiction. Hence, we conclude
\begin{equation*}\label{m'>t}
m'>t\,\, \text{or}\,\, m'=t\ge 4.
\end{equation*}

\begin{claim}\label{1-11-1}
There exists an $S_{U,V}^Q$-shift of $(\m A, \m B)$  such that $1\in U$. Moreover, for every $S_{U,V}^Q$-shift of $(\m A, \m B)$ with $1\in U$, we have $|V|=|U|\ge 3$.   
\end{claim}

\begin{proof}
Since $\{1\}$ is non-full in $\mathcal A$ and
$\mathcal A(\bar1)\ne\emptyset$, Lemma~\ref{1-7-1}(i)
applied with $R=\{1\}$ gives an $S^Q_{U,V}$-shift of
$\mathcal A$ with $1\in U$, and hence also an
$S^Q_{U,V}$-shift of the pair.

Now take any $S_{U,V}^Q$-shift of $(\mathcal{A},\mathcal{B})$ with $1\in U$.
By definition the $S_{U,V}^Q$-shift acts non-trivially on at least one of $\mathcal{A}$ and $\mathcal{B}$.

If $S_{U,V}^Q$-shift acts on $\mathcal{A}$, then there are $A\in\mathcal{A}(\bar1)$, $A'\notin\mathcal{A}$ with $S_{U,V}^Q(A)=(A\setminus V)\cup U=A'$. Since $1\in U$, $A\in \m A(\bar1)$, and hence $[2,m']\subset A$ (by Claim \ref{abcd}(i)). 
Recall that  $t\le m'$, and   
If $t=m'$, then  $[2,t]\subset A$. Combining this with Claim~\ref{propertyA}(i),  yields 
$[2,m']=[2,t]\subset V$ (since otherwise, $A'\in \m A$, contradicting $A'\notin\mathcal{A}$). 
If $t+1\le m'$, then $[2,t+1]\subset A$, combining this with~\ref{propertyA}(i), yield $[2,t+1]\subset V$.
Note that $t\ge 3$, and $t=m'$ holds only if $t\ge 4$. We conclude  $|U|=|V|\ge 3$.

If $S_{U,V}^Q$-shift acts on $\mathcal{B}$, the situation is completely symmetric: 
Claims~\ref{abcd}(ii) and~\ref{11-10-2} give  
$[2,t]\subset V$ (when $t=m'$) or $[2,t+1]\subset V$ (when $t+1\le m'$), 
which again forces $|U|=|V|\ge 3$.

In all possible scenarios we obtain $|U|\ge 3$, completing the proof.
\end{proof}

Write $\mathcal{M}=\{M_1, \dots, M_{k}\}$ and $\mathcal{T}=\{T_1, \dots, T_{\ell}\}$.
Since $\mathcal{M}, \mathcal{T}$ are intersection-minimal, for each $i\in [k]$, $i\in [\ell]$ there exists 
$
p_i\in \cap\,(\mathcal{M}\setminus \{M_i\})\setminus \cap \,\mathcal{M}$ and $q_i\in \cap\,(\mathcal{T}\setminus \{T_i\})\setminus (\cap \,\mathcal{T})$.
We may w.l.o.g. assume $\{p_1, \dots, p_k\}=[m+2, m+1+k]$ and $\{q_1, \dots, q_{\ell}\}=[t+2, t+\ell+1]$.
So for each $i\in [k]$ and $j\in [\ell]$, we have
$m+1+i\not\in M_i$, $t+1+j\not\in T_j$,
$[2, m+1+k]\setminus \{m+1+i\}\subset M_i$ and $[2,t+1+\ell]\setminus \{t+1+j\}\subset T_j$. 

\begin{claim}\label{11-13-1}
Every set of $\mathcal{B}(\bar{1})\setminus \mathcal{T}$ contains $[2, t+1+\ell]$. 
\end{claim}

\begin{proof}
To prove Claim \ref{11-13-1}, it is enough to show that
for every $x=t+1+i$ with $i\in[\ell]$ and every $y\in T_i$,
the pair $\{x,y\}$ is a cover of $\mathcal B(\bar1)$.

Arguing indirectly, we choose $x=t+1+i$ to be the smallest integer such that $\{x,y\}$ is not a cover of $\m B(\bar1)$ for some $y\in T_i$, i.e., there exist sets in $\m B(\bar{1})$ that are disjoint from $\{x,y\}$.

First, we show that $x<y$. Indeed, if not, then $x\ne y$ forces $y<x$. By Claim \ref{11-10-2}, $y\ge t+2$; thus, there exists some $1\le i'<i$ such that $y=t+1+i'<t+1+i=x$. 
Recall that $[2,t+1+\ell]\setminus \{t+1+i'\}\subset T_{i'}$. 
Hence, $x\in T_{i'}$. By the choice of $x$, we see that the pair $\{t+1+i', t+1+i\}=\{x,y\}$ is a cover of $\m B(\bar1)$, a contradiction. Thus $x<y$, as claimed.  

Next, we claim that
\begin{equation}\label{x<}
    x\le m'+2.
\end{equation}

\begin{proof}[Proof of (\ref{x<})]
Suppose for contradiction that $x>m'+2$. 
By the choice of $x$, we see that for all $t+1+j\in [t+1+1, m'+2]$, $T_j$ ($j\in [m'-t+1]$) is the only set that belongs to $\mathcal{B}(\overline{\{1,t+1+j\}})$.   
Then 
\begin{equation}
|\mathcal{B}(\bar{1})|\le {n-1-(m'+1)\choose b-(m'+1)}+|\{T_1, \dots, T_{m'-t+1}\}|<{n-1-(m'+1)\choose b-(m'+1)}+|\m T|.
\end{equation}
Therefore, 
\[
|\mathcal{B}|<{n-1\choose b-1}-{n-1-m'\choose b-1}+{n-1-(m'+1)\choose b-(m'+1)}+|\m T|.
\]
Combining this with (\ref{12-12-3}), we get ${n-1-(m'+1)\choose b-2}<{n-1-(m'+1)\choose b-(m'+1)}$, which is a contradiction since $m'+1>2$ and $n-1-(m'+1)>b-2+b-(m'+1)$.
\end{proof}

Now we are ready to prove Claim \ref{11-13-1}.
Since $\m B$ is not L-initial, there exist $S_{U,V}^Q$-shifts of $\m B$, and hence of $(\mathcal{A}, \mathcal{B})$. 
Our idea is to show that we can find an $S^Q_{U, V}$-shift of $\m B$ with $\mathcal{M}$, $\mathcal{T}$ stable, which contradicts Claim \ref{11-3-4}. We will apply Lemma \ref{1-7-1} with $r=1$ in the rest of the proof.

{\bf Case 1: $x\le m'$, $y=m'+1$.} 
We apply Lemma \ref{1-7-1}(iii) to $\m B$ with $R=\{1, m'+1\}$ and $T=\{x\}$. 
The required conditions are verified as follows:
 (P1) holds by Claim \ref{1-11-1}; 
(P2) and (P6) hold by the assumption that $\{x,y\}$ is not a cover of $\m B(\bar1)$; 
 (P3) holds by Claim \ref{abcd} (iii). 
 Thus, by Lemma \ref{1-7-1}, we obtain an $S_{U,V}^Q$-shift of $\m B$, and hence of $(\m A, \m B)$, such that $\{1, m'+1\}\subset U$ and $x\not\in V$. 
Note that $x\not\in V$ and $x=t+1+i \in T_j$ for each $j\ne i$. By $x\le m'$ and Claim \ref{abcd} (ii), 
$T_j$ is stable for all  $j\ne i$.  
Since $y=m'+1\in U\cap T_i$, $T_i$ is stable. 
Since $m'+1< m+1$, $m'+1\in \cap \, \m M$. 
Since $m'+1\in U$, $\m M$ is stable.

{\bf Case 2: $x\le m'$ and $y\ne m'+1$.}
Note that $\{x,y\}$ is not a cover of $\m B(\bar1)$. 
We start with the case that $\{x,y,m'+2\}$ is also not a cover of $\m B(\bar1)$, i.e., there exists $B\in \m B(\bar{1})$ with $\{x,y,m'+2\}\cap B=\emptyset$. 
We apply Lemma \ref{1-7-1} (iii) to $\m B$ with $R=\{1, m'+2, y\}$ and $T=\{x\}$ if $y\ge m'+2$; or  $R=\{1, m'+2\}$ and $T=\{x,y\}$ if $y\le m'$.
The required conditions are verified as follows:
 (P1) holds by Claim \ref{1-11-1}; 
 (P2) and (P6) hold by the existence of $B$;
 (P3) holds by Claim \ref{abcd} (iv) and $y\ge m'+2$ (note that $\{1,m'+2\}$ is empty in $\m B$).
 Thus, by Lemma \ref{1-7-1}, we obtain an $S_{U,V}^Q$-shift of $\m B$, and hence of $(\m A, \m B)$, such that $R\subset U$ and $T\cap V=\emptyset$. 
Now we show that $\m M, \m T$ are stable. 
Since $m'+2\le m+1$ (see (\ref{m'<m})), $m'+2\in \cap \m M$. 
Since $m'+2\in R\subset U$, $\m M$ is stable.
Note that $\{x,y\}$ is a cover of $\m T$ and $x\in T_j$ for each $j\ne i$. 
Note that $x\le m'$, and so $\{1,x\}$ is full in $\m B$ (see Claim \ref{abcd} (ii)) and $x\not\in V$. Thus, $T_j$ is stable for each $j\ne i$. 
Note that $y\in T_i$. 
If $y\le m'$, then $T_i$ is stable; if $y\ge m'+2$, then since $y\in R \subset U$, $T_i$ is stable as well. This proves that $\m T$ is stable. 

Next, assume that $\{x,y, m'+2\}$ is a cover of $ \mathcal{B}(\bar{1})$. This implies that all sets $B\in \m B(\bar1)$ with $B \cap \{x, y\}=\emptyset$ contain $m'+2$. 
Recall that $\m A$ and $\m B$ are maximal cross-intersecting families, and $a\ge 4$. Hence, $\{1,x,y,m'+2\}$ is full in $\m A$.

If $y\ge m'+2$, then we apply Lemma \ref{1-7-1}(iv) to $\m B$ with $R=\{1, y\}$, $T=\{x\}$ and $i=m'+2$;
if $y\le m'$, then we apply Lemma \ref{1-7-1}(iv) to $\m B$ with $R=\{1\}$, $T=\{x,y\}$ and $i=m'+2$. 
 The 
required conditions are verified as follows: 
(P1) holds by Claim \ref{1-11-1} and $|R|\le 2$; 
(P2) holds by the assumption that $\{x,y\}$ is not a cover of $\m B(\bar1)$ (trivially, there exists $B\in \m B(\bar1)$ such that $y\not\in B$); 
(P3) and (P4) hold since $\{1\}$ is non-full in $\m B$, by Claim \ref{abcd}(iv), and, in the case $y\ge m'+2,$ by $y\ge m'+2$; 
(P6) holds by the assumption that $\{x,y\}$ is not a cover of $ \mathcal{B}(\bar{1})$; (P7) holds by the assumption that $\{x,y, m'+2\}$ is a cover of $ \mathcal{B}(\bar{1})$. 
Thus, by Lemma \ref{1-7-1}, we obtain an $S_{U,V}^Q$-shift of $\m B$, and hence of $(\m A, \m B)$, such that $R\subset U$ and $(T\cup \{m'+2\})\cap V=\emptyset$.
Now we show that $\m M, \m T$ are stable. 
Note that $x\le m'$, $\{1,x\}$ is full in $\m A$ (see Claim \ref{abcd} (ii)) and $x\not\in V$. Thus, $T_j$ is stable for each $j\ne i$. 
Note that $y\in T_i$. 
If $y\le m'$, then $T_i$ is stable; if $y\ge m'+2$, then since $y\in R \subset U$, $T_i$ is stable as well. This proves that $\m T$ is stable.  
Note that $m'+2\le m+1$ and $x\le m'$. Then $\{x, m'+2\}\subset M$ for every $M\in \mathcal{M}$. 
Since $\{1,x,y,m'+2\}$ is full in $\m A$, we see that 
if $y\le m'$, $\m M$ is stable due to $1\in R\subset U$, $\{x, y, m'+2\}\subset M$ for every $M\in \mathcal{M}$ and $\{x,y,m'+2\}\cap V=\emptyset$;
if $y\ge m'+2$, then $\m M$ is stable due to $\{1, y\}\subset U$, $\{x, m'+2\}\subset M$ for every $M\in \mathcal{M}$ and $\{x,m'+2\}\cap V=\emptyset$.

{\bf Case 3: $x = m'+1$.}
We know that $\{1,x\}$ is non-full in $\m B$ (see Claim \ref{abcd} (iii)). 
Since $y>x$, we distinguish cases depending on whether $\{1,x,y\}$ is full in $\mathcal{B}$. 

Consider the case that $\{1,x,y\}$ is non-full in $\m B$. 
We apply Lemma \ref{1-7-1} (i) with $R=\{1,x,y\}$.
The required conditions are verified as follows:
(P1) holds by Claim \ref{1-11-1}; 
(P2) holds by the assumption that $\{x,y\}$ is not a cover of $\m B(\bar1)$; 
(P3) holds by the assumption that $\{1,x,y\}$ is non-full in $\m B$. 
Thus, by Lemma \ref{1-7-1}, we obtain an $S_{U,V}^Q$-shift of $\m B$, and hence of $(\m A, \m B)$, such that $\{1, x,y\}\subset U$. 
Note that $x=m'+1\in R\subset U$ and $m'+1<m+1$. Thus, $x\in \cap\, \m M$, and therefore, $\mathcal{M}$ is stable. 
Note that $\{x,y\}$ is a cover of $\m T$ and $\{x,y\}\subset U$. Thus,
$\mathcal{T}$ is stable. 

Assume that $\{1,x,y\}$ is full in $\m B$.  
We apply Lemma \ref{1-7-1} (iii) with $R=\{1,x\}$ and $T=\{y\}$. 
The required conditions are verified as follows:
 (P1) holds by Claim \ref{1-11-1}; 
 (P2) and (P6) hold by the assumption that $\{x,y\}$ is not a cover of $\m B(\bar1)$; 
 (P3) holds by Claim \ref{abcd} (iii). 
 Thus, by Lemma \ref{1-7-1}, we obtain an $S_{U,V}^Q$-shift of $\m B$, and hence of $(\m A, \m B)$, such that $\{1, x\}\subset U$ and $y\not\in V$. 
Similar to the above, since $x\in\cap\,\m M$ and $x\in U$,  $\mathcal{M}$ is stable. 
Since $x\in U$ and $x\in T_j$ for each $j\ne i$, $\m T \setminus \{T_i\}$ is stable. 
We have $y\in T_i$, so if $T_i$ is not stable, then $\{1, x, y\}\subset S_{U,V}^Q(T_i)$. However, 
 $\{1,x,y\}$ is full in $\m B$, thus, $T_i$ is stable.

{\bf Case 4: $x = m'+2$.}
We apply Lemma \ref{1-7-1} (i) with $R=\{1,x, y\}$. 
The required conditions are verified as follows:
 (P1) holds by Claim \ref{1-11-1}; 
 (P2) holds by the assumption that $\{x,y\}$ is not a cover of $\m B(\bar1)$; 
 (P3) holds by Claim \ref{abcd} (iv) ($\{1,x,y \}$ is empty in $\m B$). 
 Thus, by Lemma \ref{1-7-1}, we obtain an $S_{U,V}^Q$-shift of $\m B$, and hence of $(\m A, \m B)$, such that $\{1, x, y\}\subset U$. 
Note that $x=m'+2\le m+1$, $x\in \cap \, \m M$ and $x\in U$. 
We see that $\mathcal{M}$ is stable. 
Note that $\{x,y\}$ is a cover of $\m T$ and $\{x,y\}\subset U$. Thus,
$\mathcal{T}$ is stable. 

 This completes the proof of Claim \ref{11-13-1}.
\end{proof}

Claim \ref{11-13-1} gives  
\[
|\mathcal{B}(\bar{1})|\le {n-1-(t+\ell)\choose b-(t+\ell)}+|\mathcal{T}|.
\]

Now, we have finished Step 3. We now proceed to the final step. 
Note that we may assume $b\ge t+\ell$ in the rest of the proof; otherwise, $\mathcal{B}(\bar{1})=\mathcal{T}$, which, together with (\ref{12-12-2}), contradicts  (\ref{12-12-3}). 


Denote
\[
\m I :=\{\{x_1,\dots,x_{\ell}\}:x_i\in T_i\cap [t+2+\ell, n], i\in [\ell] \}.
\]
Then the size of a set in $\m I$ is at most $\ell$.
Note that for every $B\in \m B(\bar{1})\setminus \m T$ and $T\in \m T$, we have $|B\cap [t+2+\ell,n]|=b-t-\ell<b-t-\ell+1=|T\cap  [t+2+\ell,n]|$. 
We claim that for every $B\in \m B(\bar{1})\setminus \m T$, there exists $I\in \m I$ such that $I\cap B=\emptyset$.
Indeed, since $|T\cap  [t+2+\ell,n]|=b-t-\ell+1$ and $|B\cap [t+2+\ell,n]|=b-t-\ell$, there exists $x_i\in T_i\setminus B$ for each $i\in [\ell]$, thereby, the collection of these $x_i$ is a cover of $\m T$, which is disjoint from $B$. 

For every $B\in \m B(\bar{1})\setminus \m T$, we denote by $I_B$ a cover of $\m T$ that is disjoint from $B$ (note that the choice of $I_B$ may not be unique). 

\begin{claim}\label{11-19-1}
Every set of $\mathcal{B}(\bar{1})\setminus \mathcal{T}$ contains $[2, m'+2]$.
\end{claim}

\begin{proof} 
Since Claim \ref{11-13-1}, every set in $\m B(\bar1) \setminus \m T$ contains $[2,t+1+\ell]$. 
If $t+\ell\ge m'+1$, then we are done. Assume 
\[
t+\ell\le m'
\]

\begin{obser}\label{obser3}
Let $B \in \m B(\bar1) \setminus \m T$, and let $U,V\subset [n]$ with $U\cap V=\emptyset$, $1\in U$ and $|U|=|V|$. If  $S_{U,V}(B)\not\in \m B$, then $[2,t+\ell]\subset V$ and $|U|=|V|\ge \ell+2$. 
\end{obser}

\begin{proof}
Since $B \in \m B(\bar1) \setminus \m T$, Claim \ref{11-13-1} gives $[2,t+1+\ell]\subset B$. 
Since $1\in U$ and $U\cap V=\emptyset$, $U\prec V$. 
By Claim \ref{abcd}(ii), $S_{U,V}(B)\cap [2,m']=\emptyset$.
If $t+1+\ell\le m'$, then $[2,t+1+\ell]\subset V$. 
If $t+1+\ell\ge m'+1$, then since $t+\ell \le m'$, we have $m'=t+\ell$. Thus $[2,m'+1]=[2,t+1+\ell]\subset B$.
As $S_{U,V}(B)\cap [2,m']=\emptyset$, $[2, t+\ell]\subset V$. Since $t\ge 3$ and $|U|=|V|$, $|V|=|U|\ge \ell+2$.
\end{proof}

We will find a pair $(U,V)$ such that $(S_{U,V}(\m A), S_{U,V}(\m B\setminus \m T)\cup \m T)\precneqq (\m A, \m B)$, $|S_{U,V}(\m B\setminus \m T)\cup \m T|=|\m B|$,
$S_{U,V}(\m A)$ and $S_{U,V}(\m B\setminus \m T)\cup \m T$ are cross-intersecting, and $\m M $, $\m T$ are stable. This contradicts Claim \ref{11-3-4}, thereby completing the proof of Claim \ref{11-19-1}. 

We choose $U$ and $V$ according to the following cases.  
We shall later verify that this $S_{U,V}$-shift preserves $\m M, \m T$ and the cross-intersecting property.

For a family $\m F\subset {[n]\choose k}$ and a set $C\subset [n]$ with $|C|\le k$, we define
\begin{equation}\label{fc}
    \mathcal{F}^k_C:=\{C\cup R: R\in \tbinom{[n]\setminus C}{ k-|C|}\}.
\end{equation}
Then $\mathcal{E}^k_C\subset \mathcal{F}^k_C$ (the family $\mathcal{E}^k_C$ is defined in (\ref{ec})). Moreover, if $C$ is non-full in $\m F$, then $\mathcal{E}^k_C\setminus \m F \ne \emptyset$, and hence
$\mathcal{F}^k_C\setminus \m F \ne \emptyset$. 

\begin{itemize}
\item[Case 1.] There exist $B\in \m B(\bar{1})\setminus \m T$ and $I_B$ such that $m'+2\not\in B $ and  $\{1,m'+2\}\cup I_B$ is non-full in $\m B$. 

Fix $B_1$ that satisfies Case 1. 
Let $R_1:=\{1,m'+2\}\cup I_{B_1}$.
Define 
\begin{equation*}
    \m F_1:=\{F\in \m B(\bar1): R_1\cap F=\emptyset\}, \,\,\,
    \m G_1:=\m F_{R_1}^b\setminus \m B.
\end{equation*}
Since $B_1\in \m F_1$, $\m F_1\ne \emptyset$.
Since $R_1$ is non-full in $\m B$, $\m G_1\ne \emptyset$. 
For any $F\in \m F_1$ and $G\in \m G_1$, there exists a pair $(U,V)$ such that $S_{U,V}(F)=G$. Among all such pairs $(U,V)$, choose one with $V$ inclusion-minimal.

\item[Case 2.] 
There exist $B\in \m B(\bar{1})\setminus \m T$ and $I_B$ such that $m'+2\not\in B $ and $\{1,m'+2\}\cup I_B$ is full in $\m B$. 

Fix $B_2$ that satisfies Case 2.
Define $I'_{B_2}$ as follows: 
if $\{1,m'+2\}\cup (I_{B_2}\setminus [2,m'])$ is non-full, then 
let $I'_{B_2}=I_{B_2}\setminus [2,m']$; otherwise, let $I'_{B_2}=I_{B_2}\setminus [2,m'+1]$. 
Note that $\{1,m'+2\}$ is empty in $\m B$. Thereby $\{1,m'+2\}\cup I'_{B_2}$ is non-full in $\m B$. Moreover, if $I'_{B_2}=I_{B_2}\setminus [2,m'+1]$, then $\{1,m'+1,m'+2\}\cup I'_{B_2}$ is full. 
Let $R_2:=\{1, m'+2\}\cup I'_{B_2}$.
Define 
\begin{equation*}
    \m F_2:=\{F\in \m B(\bar1): (\{m'+2\}\cup I_{B_2})\cap F=\emptyset\}, \,\,\,
    \m G_2:=\m F_{R_2}^b\setminus \m B.
\end{equation*}
Since $B_2\in \m F_2$, $\m F_2\ne \emptyset$. Since $R_2$ is non-full in $\m B$, $\m G_2\ne \emptyset$. 
For any $F\in \m F_2$ and $G\in \m G_2$, there exists a pair $(U,V)$ such that $S_{U,V}(F)=G$. Among all such pairs $(U,V)$, choose one with $V$ inclusion-minimal.

\item[Case 3.] Every set in $\m B(\bar{1})\setminus \m T$ contains $m'+2$, and there exist $B\in \m B(\bar{1})\setminus \m T$ and $I_B$ such that $\{1,m'+2\}\cup I_B$ is non-full in $\m B$.

Fix $B_3$ that satisfies Case 3. 
Let $R_3:=\{1\}\cup I_{B_3}$.
Define
\begin{equation*}
    \m F_3:=\{F\in \m B(\bar1): R_3\cap F=\emptyset\}, \,\,\,
    \m G_3:=\m F_{R_3\cup \{m'+2\}}^b\setminus \m B.
\end{equation*}
Since $B_3\in \m F_3$, $\m F_3\ne \emptyset$.
Since $R_3\cup \{m'+2\}$ is non-full in $\m B$, $\m G_3\ne \emptyset$. 
For any $F\in \m F_3$ and $G\in \m G_3$, there exists a pair $(U,V)$ such that $S_{U,V}(F)=G$. Among all such pairs $(U,V)$, choose one with $V$ inclusion-minimal. 

\item[Case 4.] Every set in $\m B(\bar{1})\setminus \m T$ contains $m'+2$, and there exists $B\in \m B(\bar{1})\setminus \m T$
such that $\{1,m'+2\}\cup I_B$ is full in $\m B$. 
Furthermore, assume that $\{1\}\cup I_{B}$ is non-full in $\m B$.

Fix $B_4$ that satisfies Case 4. 
Let $R_4:=\{1\}\cup I_{B_4}$. 
Define
\begin{equation*}
    \m F_4:=\{F\in \m B(\bar1): I_{B_4}\cap F=\emptyset\}, \,\,\,
    \m G_4:=\m F_{{R_4}}^b\setminus \m B.
\end{equation*}
Since $B_4\in \m F_4$, $\m F_4\ne \emptyset$.  
Since $R_4$ is non-full in $\m B$,  $\m G_4\ne \emptyset$.
For any $F\in \m F_4$ and $G\in \m G_4$, there exists a pair $(U,V)$ such that $S_{U,V}(F)=G$. Among all such pairs $(U,V)$, choose one with $V$ inclusion-minimal.

\item[Case 5.] Every set in $\m B(\bar{1})\setminus \m T$ contains $m'+2$, and there exists $B\in \m B(\bar{1})\setminus \m T$ such that $\{1\}\cup I_{B}$ is full in $\m B$ and
$\{1,m'+2\}\cup I_{B}\setminus [2,m']$ is non-full in $\m B$.

Fix $B_5$ that satisfies Case 5. 
Let $I'_{B_5}:=I_{B_5}\setminus [2,m']$ and $R_5:=\{1\}\cup I'_{B_5}$.
Define
\begin{equation*}
    \m F_5:=\{F\in \m B(\bar1): I_{B_5}\cap F=\emptyset\}, \,\,\,
    \m G_5:=\m F_{R_5\cup \{m'+2\}}^b\setminus \m B.
\end{equation*}
Since $B_5\in \m F_5$ and $\{m'+2\}\cup R_5$ is non-full in $\m B$, $\m F_5\ne \emptyset$ and $\m G_5\ne \emptyset$. 
For any $F\in \m F_5$ and $G\in \m G_5$, there exists a pair $(U,V)$ such that $S_{U,V}(F)=G$. Among all such pairs $(U,V)$, choose one with $V$ inclusion-minimal.

\item[Case 6.] Every set in $\m B(\bar{1})\setminus \m T$ contains $m'+2$, and there exists $B\in \m B(\bar{1})\setminus \m T$ such that both $\{1\}\cup I_{B}$ and
$\{1,m'+2\}\cup I_{B}\setminus [2,m']$ are full in $\m B$, and $\{1\}\cup I_{B}\setminus [2,m']$ is non-full in $\m B$.

Fix $B_6$ that satisfies Case 6. 
Let $I'_{B_6}:=I_{B_6}\setminus [2,m']$ and $R_6:=\{1\}\cup I'_{B_6}$.
Define
\begin{equation*}
    \m F_6:=\{F\in \m B(\bar1): I_{B_6}\cap F=\emptyset\}, \,\,\,
    \m G_6:=\m F_{R_6}^b\setminus \m B.
\end{equation*}
Since $B_6\in \m F_6$, $\m F_6\ne \emptyset$. Since $R_6$ is non-full in $\m B$, $\m G_6\ne \emptyset$. 
For any $F\in \m F_6$ and $G\in \m G_6$, there exists a pair $(U,V)$ such that $S_{U,V}(F)=G$. Among all such pairs $(U,V)$, choose one with $V$ inclusion-minimal.

\item[Case 7.] Every set in $\m B(\bar{1})\setminus \m T$ contains $m'+2$, and there exists $B\in \m B(\bar{1})\setminus \m T$ such that $\{1\}\cup I_{B}\setminus [2,m']$ is full in $\m B$ (so both $\{1\}\cup I_{B}$ and 
$\{1,m'+2\}\cup I_{B}\setminus [2,m']$ are full in $\m B$). 

Fix $B_7$ that satisfies Case 7. 
Let $I'_{B_7}:=I_{B_7}\setminus [2,m'+1]$ and $R_7:=\{1\}\cup I'_{B_7}$.
Define
\begin{equation*}
    \m F_7:=\{F\in \m B(\bar1): I_{B_7}\cap F=\emptyset\}, \,\,\,
    \m G_7:=\m F_{R_7\cup \{m'+2\}}^b\setminus \m B.
\end{equation*}
Since $B_7\in \m F_7$, $\m F_7\ne \emptyset$. 
Since $\{1,m'+2\}$ is empty in $\m B$ and $I'_{B_7}\subset [m'+2, n]$, $I'_{B_7}\cup \{1, m'+2\}=R_7\cup \{m'+2\}$ is empty in $\m B$, $\m G_7\ne \emptyset$. 
For any $F\in \m F_7$ and $G\in \m G_7$, there exists a pair $(U,V)$ such that $S_{U,V}(F)=G$. Among all such pairs $(U,V)$, choose one with $V$ inclusion-minimal.
\end{itemize}

We claim that
\begin{itemize}
    \item[(a)] For each $i\in [7]$, $R_i\subset U$; 
    \item[(b)] In Case 3, Case 5, Case 7, $m'+2\not\in V$; 
    \item[(c)] In all cases, $[2,t+\ell]\subset V$;
    \item[(d)] For each $j\in [7]$, $V\cap I_{B_j}=\emptyset.$
\end{itemize}

We see that (a), (b) and (d)  follow directly from 
the definition of $\m F_i, \m G_i$. 
For each $i\in [7]$, since (d) holds and $I_{B_i}$ is a cover of $\m T$, $\m F_i\subset \m B(\bar 1)\setminus \m T$.
By Observation \ref{obser3} and $1\in R_i\subset U$, (c) holds. 


{\bf Next, we show that $\m M$ and $\m T$ are stable under the chosen $S_{U,V}$-shift. }

Since sets in $\m F_i$ avoid $I_{B_i},$ which is a cover for $\m T,$ we have that $\m T$ is stable. 
We show that $\m M$ is stable. Take $M\in\m M$.
Since $[2,m+1]\subset M$ and $m'<m$, $[2,m'+2]\subset M $.
In Case 1 and Case 2, $m'+2\in U \cap M$, so $\m M$ is stable. In Case 4, since $\{1,m'+2\}$ is empty in $\m B$ (see Claim \ref{abcd}(iv)) and $\{1,m'+2\}\cup I_{B_4}$ is full in $\m B$, we have $I_{B_4}\cap [2,m'+1]\ne \emptyset$.
On the other hand, since $\{1\}\cup I_{B_4}$ is non-full in $\m B$, Claim \ref{abcd}(ii) yields $I_{B_4}\subset [m'+1,n]$. Thus, $m'+1\in I_{B_4}\subset U$. Then $m'+1\in M \cap U$, and $\m M$ is stable.  
In Case 6, since $\{1,m'+2\}$ is empty in $\m B$ and $\{1,m'+2\}\cup I_{B_6}\setminus [2,m']$ is full in $\m B$, 
we have $m'+1\in I_{B_6}\setminus [2,m']=I'_{B_6}\subset R_6\subset U$. Hence $m'+1\in M \cap U$, and $\m M$ is stable.

It remains to treat Case 3, Case 5 and Case 7 together. Let $i\in \{3,5,7\}$ and assume, for contradiction, that $S_{U,V}(M)\not\in \m A$ for some $M\in \m M$. 
  
Since every set in $\m B(\bar{1})\setminus \m T$ contains $m'+2$ and $I_{B_i}$ is a cover of $\m T$, the maximality of $(\m A, \m B)$ implies that $\{1,m'+2\} \cup I_{B_i}$ is full in $\m A$. 
Thus, once we show that
\begin{equation}\label{subsetsuvm}
    \{1,m'+2\}\cup I_{B_i}\subset S_{U,V}(M),
\end{equation}
we obtain $S_{U,V}(M)\in \m A$, a contradiction. 

We now prove (\ref{subsetsuvm}).
Note that $m'+2\in M$, $m'+2\not\in V$ (by (b)) and $I'_{B_i}\subset R_i\subset U$. We have $\{1,m'+2\}\cup I'_{B_i}\subset S_{U,V}(M)$. 
In Case 3, we have $I_{B_3}\subset R_3 \subset U$, and hence $\{1,m'+2\}\cup I_{B_3}\subset S_{U,V}(M)$. 
For $i\in \{5,7\}$, $I_{B_i}\setminus I'_{B_i}\subset [2,m'+1]\subset M$, and since $V\cap I_{B_i}=\emptyset$ (see (d)), it follows that $I_{B_i}\setminus I'_{B_i}\subset S_{U,V}(M)$. 
Thus $\{1,m'+2\}\cup I_{B_i}\subset S_{U,V}(M)$.

{\bf Finally,  we show that $S_{U,V}(\m A)$ and $S_{U,V}(\m B\setminus \m T)\cup \m T$
are cross-intersecting.} 

Note that the above chosen $(U,V)$ may satisfy neither {\bf P} nor {\bf Q}, so the cross-intersecting property does not follow directly from Lemmas \ref{SUV} and \ref{SUV'}, and an additional argument is required. 


Assume for contradiction that there exist $A'\in S_{U,V}(\m A)$ and $B'\in S_{U,V}(\m B\setminus \m T)\cup \m T$ that are disjoint.
Let $A\in \m A, B\in \m B$ be the original sets corresponding to $A', B'$, respectively. To be precise, $A'=S_{U,V}(A)$, $B'=S_{U,V}(B)$ if $B\in \m B\setminus \m T$, and  $B'=B$ if $B\in \m T$.
Put $V':=A\cap B$. 
By the definition of $S_{U,V}$-shift,
we only need to consider the following two cases:
\begin{itemize}
    \item[(I)] $B\ne B'$, $B'\not\in \m B$ and $A'=A$; 
    \item[(II)] $A\ne A'$, $A' \not\in \m A$ and $B=B'$.
\end{itemize}
Since $A'\cap B'=\emptyset$ and $A'=A$ or $B'=B$ in the above two cases,
we conclude that 
\[
A\cap U=\emptyset, \,\,B\cap U=\emptyset, \,\, V'\subsetneqq V,
\]
to see $V'\subsetneqq V$, it's clear that $V'\subset V$, and for case (I), if $V'=V$, then since $A\cap U=\emptyset$, we have $A'=(A\setminus V)\cup U\not\in \m A$, a contradiction; and similarly for case (II).

We claim that for every $i\in [7]$, 
\begin{equation}\label{bcapi}
   B\cap I_{B_i}\ne \emptyset.
\end{equation}
\begin{proof}[Proof of (\ref{bcapi})]
Let $i\in [7]$, and assume for contradiction that $B\cap I_{B_i}= \emptyset$. 
Since $1\in U$ and $B\cap U=\emptyset$, $B\in \m B(\bar1)$.
We claim that $B\in \m F_i$. Indeed, it is clear for $i\ne 2$, and for Case 2, if $B\not\in \m F_2$, then $m'+2\in B$, and hence $m'+2\in B\cap U$, contradicting $B\cap U=\emptyset$.

As $I_{B_i}$ is a cover of $\m T$, $B\in \m B(\bar1)\setminus \m T$, and hence $[2,t+\ell +1]\subset B$. 
Since $[2,m']\subset A$ and $m'\ge t+\ell$, 
$[2,t+\ell]\subset A$, and hence $[2,t+\ell]\subset V'$.
Since $B\cap U=\emptyset$ and $V'\subsetneqq V$, $S_{U',V'}(B)\not\in \m B$ for any $U'\subset U$ with $|U'|=|V'|$. Note that $R_i\subset U$ and $|R_i|\le \ell+2\le |V'|$. 
So there exists $U'$ satisfying $R_i\subset U'\subset U$ and $|U'|=|V'|$, moreover, we have $S_{U',V'}(B)\in \m G_i$. To see $S_{U',V'}(B)\in \m G_i$, it is trivial for $i\in \{1,2,4\}$, and for $i\in \{3,5,6,7\}$, $B\in \m B(\bar1)\setminus \m T$ gives $m'+2\in B$, so $S_{U',V'}(B)\in \m G_i$. This contradicts the inclusion-minimality of $V$.   
\end{proof}

Let $i\in [7]$.
By (\ref{bcapi}),
Case 1, Case 3 and Case 4 cannot occur, since in these cases $I_{B_i}\subset R_i\subset U$ and $B\cap U=\emptyset$.
So $i\in \{2,5,6,7\}$. 
Since $B\cap I_{B_i}\ne \emptyset$ and $I'_{B_i}\subset R_i\subset U$, we have $B\cap (I_{B_i}\setminus I'_{B_i})\ne \emptyset$.
If there exists $p\in B\cap (I_{B_i}\setminus I'_{B_i})\cap [2,m']$, then since $[2,m']\subset A$, we have $p\in A\cap B=V'$ and so $V'\cap I_{B_i}\ne \emptyset$. This contradicts $V'\cap I_{B_i}\subset V\cap I_{B_i}= \emptyset$ (see (d)). 
Since $I_{B_i}\setminus I'_{B_i} \subset [2,m'+1]$,
\begin{equation}\label{bcapi=m'+1}
 B\cap (I_{B_i}\setminus I'_{B_i} )=\{m'+1\}.   
\end{equation}
Consequently, $I'_{B_i}=I_{B_i}\setminus [2,m'+1]$, and hence 
$\{1,m'+1,m'+2\}\cup I'_{B_2}$ is full in $\m B$. 

We claim that Case 5 and Case 6 cannot occur, since in these cases $I_{B_i}\setminus I'_{B_i} \subset [2,m']$, contradicting (\ref{bcapi=m'+1}).

By (\ref{bcapi=m'+1}) and $V\cap I_{B_i}=\emptyset$, we conclude 
\begin{equation}\label{5-30}
m'+1\not\in A\,\, m'+1\in B, \,\,m'+1\not\in V.    
\end{equation}

Consider Case 7. In this case, $\{1\}\cup (I_{B_7}\setminus [2,m'])$ is full in $\m B$. 
So $A\cap (I_{B_7}\setminus [2,m'])\ne \emptyset$.
Note that $I'_{B_7}=I_{B_7}\setminus [2,m'+1]$.
Since $I'_{B_7}\subset R_7\subset U$ and $A\cap U=\emptyset$, 
$A\cap I'_{B_7}=\emptyset$. Thus $m'+1\in A$, contradicting (\ref{5-30}).

Consider Case 2. In this case, $\{1,m'+2\}\cup I'_{B_2}=R_2\subset U$. 
If (I) occurs, then $B'=(B \setminus V)\cup U$.  We have 
$\{1,m'+1,m'+2\}\cup I'_{B_2}\subset B'$. 
As $\{1,m'+1,m'+2\}\cup I'_{B_2}$ is full in $\m B$, $B'\in \m B$, a contradiction.
If (II) occurs, then $A'=(A \setminus V)\cup U$. 
We have $A\cap (\{1,m'+2\}\cup I'_{B_2})=\emptyset$.  Also, $m'+1\notin A$ by (\ref{5-30}).  
At the same time, $\{1,m'+1,m'+2\}\cup I'_{B_2}$ is full in $\m B$ and $n>a+b$, which implies that $A$ cannot intersect all these sets.
\end{proof}

Now, we have finished Step 4. This completes the proof of Proposition \ref{10-25-3}.

\section{Proof of Theorem \ref{10-12-1}}
In this section, we prove Theorem \ref{10-12-1} by the `local unimodality' method, which was introduced by Huang and Peng in \cite{12-3-1}. Let us make some preparations first.

Throughout this section, we always write the elements of a set in increasing order.
Let $F\in \m F$. 
There exists $k\ge 0$ such that $[n-k+1,n]\subset F$ (if $k=0$, then $[n-k+1,n]=\emptyset$).
Let $F':=F\setminus [n-k+1,n]$ and $q:=\max F'$.
There exists $c\ge 1$ such that $[q-c+1,q]\subset F'$.
Put $F'':=F'\setminus [q-c+1,q]$. So we can write 
\begin{equation}
F=F''\sqcup[q-c+1,q]\sqcup[n-k+1,n].    
\end{equation}
Note that the above $k$ and $c$ are not necessarily unique. 
For example: $n=8$, $F=\{1,2,7,8\}$, we can choose $k=1$ and $c=1$; or $k=2$ and $c=1$; or $k=2$ and $c=2$.
We define
\begin{itemize}
    \item[Type I.] We say that $F$ is of {\it Type I in $\m F$} if there exist $G, H \in \m F$ such that 
    \begin{equation}\label{writegh1}
    G=F''\sqcup[q-c,q-1]\sqcup[n-k+1,n], \,\, H=F''\sqcup[q-c+2,q+1]\sqcup[n-k+1,n].
    \end{equation}
    \item[Type II.] We say that $F$ is of {\it Type II in $\m F$} if there exist $G, H \in \m F$ such that 
    \begin{equation}\label{writegh2}
    G=F''\sqcup[q-c+1,q+1]\sqcup[n-k+2,n], \,\, H=F''\sqcup[q-c+2,q]\sqcup[n-k,n].
     \end{equation}
\end{itemize}

We need the following notation.
Let $F, H \subset [n]$ satisfy that $\max F=\max H=q$, $|F|=f$, $|H|=h$, $F\cap H=\{q\}$ and $F\cup H=[q]$. Let $k$ be such that $k\leq n-f$. 
We define the {\it $k$-partner} $K$ of $F$ such that $|K|=k$, $K\prec H$, and there is no other $k$-set $K'$ satisfying $K\precneqq K'\prec H$. By the definition, we see that $K=H$ if $k=h$; $K= H\cup [n-k+h+1, n]$ if $k>h$. 

In \cite{k_1+k_3}, Huang and Peng proved the following result.

\begin{lemma}[Huang--Peng \cite{k_1+k_3}]\label{3-12-1}
Let $n\ge a+b$, $A\subset [n]$, $|A|=a$, and let $B$ be the $b$-partner of $A$. Then
$\mathcal{L}([n], B, b)$ and $\mathcal{L}([n], A, a)$ are cross-intersecting. Moreover, $\mathcal{L}([n], B, b)$ is maximal cross-intersecting with $\mathcal{L}([n], A, a)$.
\end{lemma}

Let $n\ge a+b$. For each $A\in {[n]\choose a}$, let $B\in {[n]\choose b}$ be the $b$-partner of $A$. Denote
\begin{equation}\label{12-3-4}
f(A)=|\mathcal{L}([n], A, a)|+|\mathcal{L}([n], B, b)|.
\end{equation}
The following two lemmas embody, in a certain sense, the `local unimodality' of the function defined in (\ref{12-3-4}).
\begin{lemma}[Huang--Peng \cite{12-3-1}]\label{clm28}
Let $n> a+b$ and $A\in \m A\subset {[n]\choose a}$. If $A$ is of Type I in $\m A$, then there exist $G, H\in \m A$ satisfying (\ref{writegh1}), and $f(A)<\max \{f(G), f(H)\}$.
\end{lemma}

\begin{lemma}[Huang--Peng \cite{12-3-1}]\label{clm29}
Let $n> a+b$ and $A\in \m A\subset {[n]\choose a}$. If $A$ is of Type II in $\m A$, then there exist $G, H\in \m A$ satisfying (\ref{writegh2}), and $f(A)<\max \{f(G), f(H)\}$.
\end{lemma}

Now we are ready to give the proof of Theorem \ref{10-12-1}.

\begin{proof}[Proof of Theorem \ref{10-12-1}]
If $n=a+b$ or $a=b$ or $\m A=\emptyset$, then we are done. We may assume that $n>a+b$, $a>b$ and $\m A\ne\emptyset$. 

By the Kruskal--Katona theorem, we may assume that $\mathcal{A}$ and $\mathcal{B}$ are L-initial, in other words, we may assume that there are $A\in {[n]\choose a}$ and $B\in {[n]\choose b}$ such that 
$\mathcal{A}=\mathcal{L}([n], |\mathcal{A}|, a)=\mathcal{L}([n], A, a)$ and $\mathcal{B}=\mathcal{L}([n], |\mathcal{B}|, b)=\mathcal{L}([n], B, b).$
Let $B'$ be the $b$-partner of $A$. 
By Lemma \ref{3-12-1}, $\m L([n], B', b)$ is maximal cross-intersecting with $\m A$. Hence $\m B\subset \m L([n], B', b)$, and 
\[
f(A)=|\mathcal{L}([n], A, a)|+|\mathcal{L}([n], B', b)|=|\m A|+|\m B'|.
\]
Let $\m F:=\{F\in {[n]\choose a}: [a-b]\subset F\}$.
Since $\emptyset\ne |\m A|\le {n-a+b\choose b}$, $A\in \m F$. 
In view of Lemmas \ref{clm28} and \ref{clm29},
if $A$ is of Type I or Type II in $\m F$, then $f(A)<\max \{f(F): F\in \m F\}$. 
Observe that $[a]$ and $[a-b]\sqcup[n-b+1,n]$ are the only sets that are neither of Type I nor of Type II in $\m F$, since they are the first and the last sets of $\m F$, respectively.  
Thus,
\[
|\m A|+|\m B|\le |\m A|+|\m B'|=f(A)\le \max \{f([a]), f([a-b]\sqcup[n-b+1,n])\}=\tbinom{n}{b},
\]
where the last equality holds by the following calculation: $f([a])=1+\tbinom{n}{b}-\tbinom{n-a}{b}\le \tbinom{n}{b}$ and $f([a-b]\sqcup[n-b+1,n])=\tbinom{n-a+b}{b}+\tbinom{n}{b}-\tbinom{n-a+b}{b}= \tbinom{n}{b}$.
\end{proof}

\end{document}